\documentclass[12pt]{amsart}
\usepackage{amsmath, amsthm, amscd, amssymb, amsfonts, latexsym, color}
\usepackage{fullpage}
\usepackage{bm}
\usepackage{graphicx}
\usepackage[all]{xy}
\usepackage{tikz}

\newcommand\cube{\begin{tikzpicture}[scale=2.3]
    \coordinate (A1) at (0, 0);
    \coordinate (A2) at (0, 0.1);
    \coordinate (A3) at (0.1, 0.1);
    \coordinate (A4) at (0.1, 0);
    \coordinate (B1) at (0.03, 0.03);
    \coordinate (B2) at (0.03, 0.13);
    \coordinate (B3) at (0.13, 0.13);
    \coordinate (B4) at (0.13, 0.03);

    \draw (A1) -- (A2);
    \draw (A2) -- (A3);
    \draw (A3) -- (A4);
    \draw (A4) -- (A1);
    \draw[densely dotted] (A1) -- (B1);
    \draw[densely dotted] (B1) -- (B2);
    \draw (A2) -- (B2);
    \draw (B2) -- (B3);
    \draw (A3) -- (B3);
    \draw (A4) -- (B4);
    \draw (B4) -- (B3);
    \draw[densely dotted] (B1) -- (B4);
\end{tikzpicture}}

\newcommand\tinycube{\begin{tikzpicture}[scale=1.3]
    \coordinate (A1) at (0, 0);
    \coordinate (A2) at (0, 0.1);
    \coordinate (A3) at (0.1, 0.1);
    \coordinate (A4) at (0.1, 0);
    \coordinate (B1) at (0.03, 0.03);
    \coordinate (B2) at (0.03, 0.13);
    \coordinate (B3) at (0.13, 0.13);
    \coordinate (B4) at (0.13, 0.03);

    \draw (A1) -- (A2);
    \draw (A2) -- (A3);
    \draw (A3) -- (A4);
    \draw (A4) -- (A1);
    \draw[densely dotted] (A1) -- (B1);
    \draw[densely dotted] (B1) -- (B2);
    \draw (A2) -- (B2);
    \draw (B2) -- (B3);
    \draw (A3) -- (B3);
    \draw (A4) -- (B4);
    \draw (B4) -- (B3);
    \draw[densely dotted] (B1) -- (B4);
\end{tikzpicture}}

\newcommand{\F}{\mathbb F}

\newcommand{\Z}{\mathbb Z}
\newcommand{\Q}{\mathbb Q}
\newcommand{\R}{\mathbb R}

\newcommand{\C}{\mathbb C}

\newcommand{\kommentar}[1]{}

\newcommand{\cF}{\chi_{F_1} \overline{\chi_{F_2}}}
\newcommand{\cFeval}[1]{\chi_{F_1}(#1)\overline{\chi_{F_2}}(#1)}

\newcommand{\cFbeval}[1]{\overline{\chi_{F_1}}(#1){\chi_{F_2}}(#1)}

\DeclareMathOperator{\tr}{tr}

\DeclareMathOperator{\Res}{Res}

\DeclareMathOperator{\re}{Re}

\renewcommand{\pmod}[1]{\,(\mathrm{mod}\,#1)}

\definecolor{pink}{rgb}{1,.2,.6}
\definecolor{orange}{rgb}{0.7,0.3,0}
\definecolor{blue}{rgb}{.2,.6,.75}
\definecolor{green}{rgb}{.4,.7,.4}
\definecolor{purple}{RGB}{127,0,255}

\newcommand{\ccom}[1]{{\color{pink}{CD: #1}} }

\newcommand{\mcom}[1]{{\color{orange}{Matilde: #1}} }

\newcommand{\acom}[1]{{\color{blue}{Alexandra: #1}} }

\newtheorem{lem}{Lemma}[section]
\newtheorem{prop}[lem]{Proposition}
\newtheorem{thm}[lem]{Theorem}

\newtheorem{cor}[lem]{Corollary}

\theoremstyle{definition}

\newtheorem{rem}[lem]{Remark}

\author{Chantal David}
\address{Department of Mathematics and Statistics Concordia University, 1455 de Maisonneuve West Montr\'eal, Qu\'ebec, Canada H3G 1M8}
\email{chantal.david@concordia.ca}

\author{Alexandra Florea}
\address{Columbia University, Mathematics Department, Rm 606, MC 4417, 2990 Broadway, New York NY 10027, USA}
\email{aflorea@math.columbia.edu}

\author{Matilde Lalin}
\address{Universit\'e de Montr\'eal, Pavillon Andr\'e-Aisenstadt, D\'epartement de math\'ematiques et de statistique, CP 6128, succ. Centre-ville, Montr\'eal, Qu\'ebec, Canada H3C 3J7}
\email{mlalin@dms.umontreal.ca}

\begin{document}

\title{The mean values of cubic $L$-functions over function fields}

\begin{abstract}
We obtain an asymptotic formula for the mean value of $L$--functions associated to cubic characters over $\F_q[T]$. We solve this problem in the non-Kummer setting when $q \equiv 2 \pmod 3$ and in the Kummer case when $q \equiv 1 \pmod 3$. The proofs rely on obtaining precise asymptotics for averages of cubic Gauss sums over function fields, which can be studied using the theory of metaplectic Eisenstein series. In the non-Kummer setting we display some explicit cancellation between the main term and the dual term coming from the approximate functional equation of the $L$--functions.
\end{abstract}


\maketitle

\tableofcontents{}
 
\section{Introduction}

The problem we consider in this paper is that of computing the mean value of Dirichlet $L$--functions $L_q(s,\chi)$ evaluated at the critical point $s=1/2$ as $\chi$ varies over the primitive cubic Dirichlet characters of $\F_q[T]$. We will solve this problem in two different settings: when the base field $\F_q$ contains the cubic roots of unity (or equivalently when $q \equiv 1 \pmod 3$; we call this the Kummer setting) and when $\F_q$ does not contain the cubic roots of unity (when $q \equiv 2 \pmod 3$; we call this the non-Kummer setting.) 

There are few papers in literature about moments of cubic Dirichlet twists over number fields, especially compared to the abundance of papers on quadratic twists.  For the case of quadratic characters over $\Q$,  
 the first moment was computed by Jutila \cite{Jutila}, and the second and third moments by Soundararajan \cite{Sound}. 
 For the case of quadratic characters over $\F_q[T]$, the first 4 moments were computed by the second author of this paper \cite{F17c, F17b, F17}. In particular, the improvement of the error term for the first moment in \cite{F17c} showed the existence of a secondary term (of size approximately the cube root of the main term) which was not predicted by any heuristic.
 A secondary term of size $X^{3/4}$ was explicitly computed by Diaconu and Whitehead in the number field setting \cite{DI2} for the cubic moment of quadratic $L$--functions and by Diaconu in the function field setting \cite{DI1}.

For the case of cubic characters, Baier and Young \cite{BY} considered the cubic Dirichlet characters over $\Q$ and obtained for the smoothed first moment that
\begin{equation} \label{BY}
\sum_{(q, 3) = 1} \sum_{\substack{\chi \;\text{primitive mod} \; q \\\chi^3 = \chi_0}} L({1}/{2}, \chi ) w \left(\frac{q}{Q} \right) = c \hat{w}(0) Q + O \left( Q^{37/38 + \varepsilon} \right),
\end{equation} 
with an explicit constant $c$. Using an upper bound for higher moments of $L$--functions, Baier and Young also show that the number of primitive Dirichlet characters $\chi$ of order $3$ with conductor less than or equal to $Q$ for which $L(1/2,\chi) \neq 0$ is bounded below by $Q^{\frac{6}{7}-\varepsilon}$.

Another result related to \cite{BY} is that of Cho and Park \cite{CP}, where the authors consider the $1$--level density of zeros in the same family as that of Baier and Young. They compute the $1$--level density when the support of the Fourier transform of the test function is in $(-1,1)$ and show agreement with the prediction coming from the Ratios Conjecture.

The first moment of the cubic Dirichlet twists over $\Q(\xi_3)$  was considered by Luo in \cite{Luo}, and his main term has the same size as the first moment over $\Q$, because the author considers only a thin subsets of the cubic characters, namely those given by the cubic residue symbols $\chi_c$ where $c \in \Z[\xi_3]$ is square-free. This does not count the 
conjugate characters $\chi_c^2 = \chi_{c^2}$, and in particular, the first moment of \cite{Luo} is not real.

The problem of computing the mean value of cubic $L$--functions over function fields was considered by Rosen in \cite{Rosen-cyclic}, where he averages over all monic polynomials of a given degree. This problem is different than the one we consider, since the  counting is not done by genus and obtaining an asymptotic formula relies on using a combinatorial identity.

Before stating our results, we first introduce some notation. Let $q$ be an odd prime power, and let $\F_q[T]$ be the set of polynomials over the finite field $\F_q$. A Dirichlet character $\chi$ of modulus $m \in \F_q[T]$  is a multiplicative function from $(\F_q[T]/(m))^*$ to $\C^*$, extended to  $ \F_q[T]$ by periodicity if $(a,m)=1$, and defined by $\chi(a) = 0$ if
$(a,m) \neq 1$. A cubic Dirichlet character is such that $\chi^3$ equals the principal character $\chi_0$, and it takes values in $\mu_3$, the cubic roots of 1 in $\C^*$.
The smallest period of $\chi$ is called the conductor of the character. We say that $\chi$ is a primitive character of modulus $m$ when $m$ is the smallest period. 

We denote by $L_q(s,\chi)$ the $L$--function attached to the character $\chi$ of $\F_q[T]$. We keep the index $q$ in the notation to avoid confusion, as we will also work over the quadratic extension $\mathbb{F}_{q^2}$ of $\F_q$.

The set of cubic characters differs when $\F_q$ contains the third roots of unity or not.
 If  $q \equiv 1 \pmod 3$, $\F_q$ contains the third roots of unity,
and the number of primitive cubic Dirichlet characters with conductor of degree $d$ is asymptotic to $B_{\mathrm{K},1} d q^d + B_{\mathrm{K},2} q^d$ for some explicit constants $B_{\mathrm{K},1}, B_{\mathrm{K},2}$ (see Lemma 
\ref{count-kummer}). If $q \equiv 2 \pmod 3$, $\F_q$ does not contain the third roots of unity,  and the number of primitive cubic Dirichlet characters with conductor of degree $d$
is asymptotic to $B_{\mathrm{nK}} q^d$ for some explicit constant $B_{\mathrm{nK}}$ (see Lemma \ref{count-non-kummer}). 

We will count primitive  cubic characters ordering them by the degree of their conductor, or equivalently by the genus $g$ of the cyclic cubic field extension of $\F_q[T]$ associated to such a character (see formula \eqref{genus}). 

We compute the first moment of cubic $L$--functions for the two settings. In the non-Kummer case, we have the following.
 \begin{thm} \label{first-moment-non-Kummer} \label{thm-non-Kummer}
 Let $q$ be an odd prime power such that  $q \equiv 2 \pmod 3$. Then
 \begin{align*}
\sum_{\substack{\chi  \; \mathrm{primitive \; cubic} \\ \mathrm{genus}(\chi) = g}} L_q(1/2, \chi)
=\frac{\zeta_q(3/2)}{\zeta_q(3)}  \mathcal{A}_{\mathrm{nK}} \left (\frac{1}{q^2}, \frac{1}{q^{3/2}} \right )  q^{g+2} +O(q^{\frac{7g}{8}+\varepsilon g}),
  \end{align*}
  with $\mathcal{A}_{\mathrm{nK}}(q^{-2},q^{-3/2})$ given in Lemma \ref{mt_expression_nk}.
 \end{thm}

In the Kummer case, we have the following.
\begin{thm} \label{first-moment-Kummer} Let $q$ be an odd prime power such that $q \equiv 1 \pmod 3$, and let $\chi_3$ be the cubic character on $\F_q^*$ given by \eqref{def-chi3}. Then,
$$
\sum_{\substack{\chi  \; \mathrm{ primitive \; cubic} \\ \mathrm{genus}(\chi) = g\\ \chi \vert_{\F_q^*} = \chi_3}} L_q(1/2, \chi)
= C_{\mathrm{K},1} g q^{g+1} + C_{\mathrm{K},2} q^{g+1} + O \left( q^{g\frac{1+\sqrt{7}}{4}+\varepsilon g} \right),$$
where $C_{\mathrm{K},1}$ and $C_{\mathrm{K},2}$ are given by equations \eqref{c-1} and \eqref{c-2} respectively.
\end{thm}
The hypothesis that $\chi$ restricts to the character $\chi_3$ on $\F_q$ is not important, but simplifies the computations by ensuring that  the $L$--functions have the same functional equation. It is analogous to the restriction in the case of quadratic characters to those with conductor of degree either $2g$ or $2g+1$.

Since $L$-functions satisfy the Lindel\"of hypothesis over function fields (see Lemma \ref{lindelof}), one can easily bound the second moment, and we get the following corollary.
\begin{cor} Let $q$ be an odd prime power. Then, 
$$
\# \left\{ \chi \, \mbox{cubic, primitive of genus $g$} \;:\;  L_q(1/2, \chi) \neq 0 \right\} \gg q^{(1-\varepsilon) g}.
$$
\end{cor}

Translating from the function field to the number field setting, we associate $q^g$ with $Q$. Note that Theorem \ref{first-moment-non-Kummer} is the function field analog of \eqref{BY}, and the proof of our Theorem  \ref{first-moment-non-Kummer} has many similarities with the work of  \cite{BY}.
The better quality of our error term can be explained in part by the fact that we can use the Riemann Hypothesis to bound the error term.
In the number field case, the same quality  of error term  can be obtained without the  Riemann Hypothesis for some  families using the appropriate version of the large sieve (for example in the case of the family of quadratic characters, with the quadratic large sieve due to Heath-Brown \cite{HB95}). However the cubic large sieve, also due to Heath-Brown \cite{HB00}, provides a weaker upper bound. There is  also an asymmetry between the sum over the cubic characters, which is naturally a sum over $\Q(\xi_3)$, and the truncated Dirichlet series of the $L$-function, which is a sum over $\Z$.
The asymmetry of the sums also exists in the function field setting. 

Another difference from the work of Baier and Young is that we explicitly exhibit cancellation between the main term and the dual term coming from using the approximate functional equation for the $L$--functions. In their work Baier and Young \cite{BY} prove an upper bound for the dual term without obtaining an asymptotic formula for it, which is what we do in the function field case.

The first steps of our proofs are the usual ones, using the approximate functional equation to write the special value 
\begin{eqnarray} \label{L-value}
L_q(1/2, \chi) &=& \sum_{f \in \mathcal{M}_q} \frac{\chi(f)}{|f|_q^{1/2}}, \end{eqnarray}
as a sum of two terms (the principal sum and the dual sum), where for a polynomial $f \in \F_q[T]$ the norm is defined by $|f|_q=q^{\deg(f)}$. Inspired by the work of Florea \cite{F17} to improve the quality of the error term, we evaluate exactly the dual sum and the secondary term of the main sum (corresponding to taking $f$ cube in the approximate functional equation) in order to obtain cancellation of those  terms. This is similar to the work of Florea for the first moment of quadratic Dirichlet characters over functions fields, replacing quadratic Gauss sums by cubic Gauss sums. Of course, this is not a trivial difference, as the behavior of quadratic Gauss sums is very regular since they are multiplicative functions. However cubic Gauss sums are different as they are no longer multiplicative. Handling the cubic Gauss sums is significantly more difficult than working with quadratic Gauss sums. This is one of the main focuses of our paper. 

The distribution of Gauss sums over number fields was adressed by Heath-Brown and Patterson \cite{HBP}, using the deep work of Kubota for automorphic forms associated to the metaplectic group. This was generalised by Hoffstein \cite{hoffstein} and Patterson \cite{patterson} for the function field case, and we review their work in Section \ref{sec:GS}. The main goal of Section \ref{sec:GS} is to obtain an exact formula for the residues of the generating series
$$
\tilde{\Psi}_q(f,u) = \sum_{\substack {F \in \mathcal{M}_q\\ (F, f)=1}} G_q(f, F) u^{\deg(F)},
$$
where $G_q(f, F)$ is the generalized shifted Gauss sum over $\F_q$ as defined by \eqref{gen-GS}. With those residues in hand, we can evaluate precisely the main term of the dual sum, and indeed we can show that it (magically!) cancels with the secondary term of the principal sum. Unfortunately obtaining the cancellation is not enough to improve the error term, as we do not have good bounds for $\tilde{\Psi}_q(f, u)$ beyond the pole at $u^3=1/q^4$. 
We prove that the convexity bound in Lemma \ref{our-psi} holds, and any improvement of the convexity bound would allow an improvement of the error term of Theorem \ref{first-moment-non-Kummer} coming from the cancellation that we exhibit.

Proving Theorem \ref{first-moment-Kummer} is more difficult than obtaining the asymptotic formula in the non-Kummer case, and our error term is not as good as that in Theorem \ref{first-moment-non-Kummer}. To our knowledge, Theorem \ref{first-moment-Kummer} is the first result when one considers all the primitive cubic characters (with the technical restriction that $\chi \vert_{\F_q^*} = \chi_3$, which does not change the size of the family). This explains the (maybe surprising) asymptotic for the first moment in Theorem \ref{first-moment-Kummer}, which is of the shape $g q^g P(1/g)$ where $P$ is a polynomial of degree $1$. 

Because of the size of the family of cubic twists in the Kummer case, we are not able to obtain cancellation between the dual term and the error term from the main term. Certain cross-terms seem to contribute to the cancellation, but we cannot obtain an asymptotic formula for these cross terms. Instead we bound them using the convexity bound for $\tilde{\Psi}_q(f,u)$, which explains the bigger error term from Theorem \ref{first-moment-Kummer}.

We remark that the  results of Theorems  \ref{first-moment-non-Kummer} and \ref{first-moment-Kummer} both correspond to a family with unitary symmetry, as expected. Note that for our results, we fix the size $q$ of the finite field and let the genus $g$ go to infinity. If instead one fixes the genus and lets $q$ go to infinity, it should be possible to obtain asymptotic formulas for moments using equidistribution results as in the work of Katz and Sarnak \cite{KS} and then a random matrix theory computation as in the work of Keating and Snaith \cite{Keating-Snaith}.
\kommentar{ Do we want to say more (if anything) about this? Just wanted to point out that when $q \to \infty$ the constant we get for the non-Kummer case (which doesn't involve the Euler product) is $1$ because the ratio of $\zeta$'s will tend to $1$. This is the same as in the case of the second moment of $\zeta$, which is good because both families have unitary symmetry.}

As mentioned before, a lower order term of size the cube root of the main term was computed in \cite{F17} in the case of the mean value of quadratic $L$--functions. We remark that in the case of the mean value of cubic $L$--functions, we can explicitly compute a term of size $q^{5g/6}$ in the non-Kummer case and a term of size $gq^{5g/6}$ in the Kummer setting
(see remarks \ref{remark56} and \ref{remark56b} respectively). Due to the size of the error terms, these terms do not appear in the asymptotic formulas in Theorems \ref{first-moment-non-Kummer} and \ref{first-moment-Kummer}. However, we suspect these terms do persist in the asymptotic formulas. Improving the convexity bound on $\Tilde{\Psi}_q(f,u)$ would allow us to improve the error terms, and maybe to detect the lower order terms. We remark that a similar sized term was conjectured by Heath-Brown and Patterson \cite{HBP} for the average of the arguments of cubic Gauss sums in the number field setting. We believe the matching size of these terms is not a coincidence, as the source of our $q^{5g/6}$ comes from averaging cubic Gauss sums over function fields.

\medskip
{\bf Acknowledgements.} The authors would like to thank Roger Heath-Brown, Maxim Radziwill, Kannan Soundararajan, and Matthew Young for helpful discussions.
The research of the first and third authors is supported by the National Science and Engineering Research Council of Canada (NSERC) and 
the Fonds de recherche du Qu\'ebec -- Nature et technologies (FRQNT). The second author of the paper was supported by a National Science Foundation (NSF) Postdoctoral Fellowship during part of the research which led to this paper.

\section{Notation and Setting}\label{sec:notation}

Let $q$ be an odd prime power such that $q \equiv 1 \pmod 3$. 
We denote by $\mathcal{M}_q$ the set of monic polynomials of $\F_q[T]$, by $\mathcal{M}_{q,d}$ the set of monic polynomials of degree exactly $d$, by
$\mathcal{M}_{q, \leq d}$ the set of monic polynomials of degree smaller than or equal to $d$,  
 by $\mathcal{H}_{q}$ the set of monic square-free polynomials of $\F_q[T]$ and analogously for 
 $\mathcal{H}_{q,d}$ and $\mathcal{H}_{q,\leq d}$. Note that $| \mathcal{M}_{q,d}| = q^d$ and for $d \geq 2$, we have that $| \mathcal{H}_{q,d}| = q^d (1-\tfrac{1}{q}).$
 
In general, unless stated otherwise, all polynomials are monic. 
As for the $L$--functions in the introduction, we keep the index $q$ in the notation to avoid confusion, as we will have to consider polynomials over the quadratic extension $\F_{q^2}$ of $\F_q$ when $q \equiv 2 \pmod 3$. 

We define the norm of a polynomial $f(T) \in \F_q[T]$ over $\F_q[T]$ by
$$ |f|_{q} = q^{\deg(f)}.
$$
Then, if $f(T) \in \F_q[T]$, we have $|f|_{q^n} = q^{n\deg(f)},$ for any positive integer $n$.

For $q \equiv 1 \pmod{3}$ we fix once and for all an isomorphism $\Omega$ between $\mu_3$, the cubic roots of 1 in $\C^*$, and the cubic roots of 1 in 
$\F_q^*.$
We also fix a cubic character $\chi_3$ on $\F_q^*$ by
\begin{equation} \label{def-chi3}
\chi_3(\alpha) =  \Omega^{-1} \left( \alpha^{\frac{q-1}{3}} \right).
\end{equation}
For any character $\chi$ on $\F_q[T]$, we say that $\chi$ is even if it is trivial on $\F_q^*$, and odd otherwise. 
Then, when $q$ is an odd prime power such that $q \equiv 1 \pmod 3$, any cubic character on $\F_q[T]$ falls in three natural classes depending on its restriction to $\F_q^*$ which is either $\chi_3$, $\chi_3^2$ or the trivial character 
(in the first 2 cases,  the character is odd, and in the last case, the character is even).\footnote{
We will see in Section \ref{section-cubic}  that when $q \equiv 2 \pmod 3$, any cubic character on $\F_q[T]$ is even.} 

For any odd character $\chi$ on $\F_q[T]$, we denote by $\tau(\chi)$ the Gauss sum of the restriction of $\chi$ to $\F_q$ (which is either $\chi_3$ or $\chi_3^2$), i.e.
\begin{eqnarray} \label{standard-GS} 
\tau(\chi) = \sum_{a\in \F_q^*} \chi(a) e^{2 \pi i \tr_{\F_q/\F_p}(a)/p}. \end{eqnarray}

Then, $|\tau(\chi)| = q^{1/2}$, and we denote the sign of the Gauss sum by 
\begin{eqnarray} \label{sign-restriction}
\epsilon(\chi) = q^{-1/2} {\tau(\chi)}.
\end{eqnarray}
When $\chi$ is even, we set $\epsilon(\chi)=1$.

We will often use the fact that when $q \equiv 1 \pmod 6$, the cubic reciprocity law is very simple.

\begin{lem}[Cubic Reciprocity] \label{cubic-reciprocity} Let $a, b \in \F_q[T]$ be relatively prime monic polynomials, and let $\chi_a$ and $\chi_b$ be the cubic residue symbols defined above. If $q \equiv 1 \pmod 6$, then
$$\chi_a(b) = \chi_b(a).$$ 
\end{lem} 
\begin{proof} This is Theorem 3.5 in \cite{Rosen} in the case where $a$ and $b$ are monic and $q \equiv 1 \pmod 6$. \end{proof}
Finally, we recall Perron's formula over $\F_q[T]$ which we will use many times throughout the paper.

\begin{lem}[Perron's Formula] \label{perron}
If the generating series $\mathcal{A}(u) =\sum_{f \in \mathcal{M}_q} a(f) u^{\deg(f)}$ 
 is absolutely convergent in $|u|\leq r<1$, then
\[\sum_{{f} \in \mathcal{M}_{q,n}} a(f)=\frac{1}{2\pi i}\oint_{|u|=r} \frac{\mathcal{A}(u)}{u^n} \; \frac{du}{u}\]
and
\[\sum_{{f} \in \mathcal{M}_{q,\leq n}} a(f)=\frac{1}{2\pi i}\oint_{|u|=r} \frac{\mathcal{A}(u)}{u^{n}(1-u)} \; \frac{du}{u},\]
where, in the usual notation, we take $\oint$ to signify the integral over the circle oriented counterclockwise. 
\end{lem}

\subsection{Zeta functions  and the approximate functional equation}
The {\it affine zeta function} over $\F_q[T]$ is defined by 
$$ \mathcal{Z}_q(u) = \sum_{f \in \mathcal{M}_q} u^{\deg(f)}$$ for $|u|<1/q$. By grouping the polynomials according to the degree, it follows that
$$ \mathcal{Z}_q(u) = \sum_{n=0}^{\infty} u^n q^n = \frac{1}{1-qu},$$ and this provides a meromorphic continuation of $\mathcal{Z}_q(u)$ to the entire complex plane. We remark that $\mathcal{Z}_q(u)$ has a simple pole at $u=1/q$
with residue $-\frac{1}{q}$. We also define
\[\zeta_q(s)=\mathcal{Z}_q(q^{-s}).\]

Note that $\mathcal{Z}_q(u)$ can be expressed in terms of an Euler product as follows
$$ \mathcal{Z}_q(u) = \prod_{P} (1-u^{\deg(P)})^{-1},$$ where the product is over monic irreducible polynomials in $\F_q[T]$.

Let $C$ be a curve over $\F_q(T)$ whose function field is a cyclic cubic extension of $\F_q(T)$. 
From the Weil conjectures, the zeta function of the curve $C$ can be written as
 $$ \mathcal{Z}_{C}(u) =  \frac{\mathcal{P}_C(u)}{(1-u)(1-qu)},$$ 
 where
$$\mathcal{P}_{C}(u) = \prod_{j=1}^{g} \left( 1 - \sqrt{q} u e^{2 \pi i \theta_j} \right)  \prod_{j=1}^{g} \left( 1 - \sqrt{q} u e^{- 2 \pi i \theta_j} \right)$$
for some eigenangles $\theta_j$, $j = 1, \dots, g$. 

We can write $\mathcal{P}_{C}(u)$ in terms of the $L$-functions of the two cubic Dirichlet characters $\chi$ and $\overline{\chi}$ of the function field of $C$.  Let 
$h$ be the conductor of the non-principal character $\chi$. 
Define 
\begin{eqnarray}
\mathcal{L}_q(u, \chi) := \sum_{f \in \mathcal{M}_q} \chi(u) u^{\deg(f)} = \sum_{d <  \deg(h)}  u^d  \sum_{f \in \mathcal{M}_{q, d}} \chi(f) , \label{finite-L-1}
\end{eqnarray}
where the second equality follows from the orthogonality relations.

We remark that setting $u=q^{-s}$, we have $L_q(s,\chi)=\mathcal{L}_q(u, \chi)$. From now on 
we will mainly use the notation $\mathcal{L}_q(u, \chi)$. The $L$--function has the following Euler product
$$ \mathcal{L}_q(u,\chi) = \prod_{P \nmid h} (1-\chi(P) u^{\deg(P)})^{-1},$$ where the product is again over monic irreducible polynomials $P$ in $\F_q[T]$. From now on, the Euler products we consider are over monic, irreducible polynomials and if there is an ambiguity as to whether the polynomials belong to $\F_q[T]$ or $\F_{q^2}[T]$ we will indicate so.


Considering the prime at infinity, we write
\begin{eqnarray} \label{L}
\mathcal{L}_C(u, \chi) = \begin{cases}  \mathcal{L}_q(u, \chi) &\mbox{if $\chi$ is odd,} \\ \\
\displaystyle\frac{\mathcal{L}_q(u, \chi)}{1-u}  & \mbox{if $\chi$ is even.} \end{cases}
\end{eqnarray}
Then we have 
\[\mathcal{P}_{C}(u)=\mathcal{L}_C(u, \chi)\mathcal{L}_C(u, \overline{\chi}).\]

Furthermore, using the Riemann--Hurwitz formula, we have that
\begin{eqnarray} \label{genus}
\deg(h) = g+2 - \begin{cases} 0 & \mbox{if $\chi$ is even,} \\ 1 & \mbox{if $\chi$ is odd.} \end{cases}
\end{eqnarray}

\begin{lem} \label{lemma-FE-gen}  \label{sign-FE} Let $\chi$ be a primitive cubic character to the modulus $h$.

If $\chi$ is odd, then 
$\mathcal{L}_q(u, \chi)$ satisfies the functional equation
\begin{equation} \label{eq:funceq-gen}
\mathcal{L}_q(u, \chi) = \omega(\chi) (\sqrt{q}u)^{\deg(h)-1} \mathcal{L}_q \left(\frac{1}{qu}, \overline{\chi}\right),
\end{equation}
where the sign of the functional equation is
\begin{eqnarray} \label{sign}
\omega(\chi) &=& q^{-(\deg(h)-1)/2} \sum_{f \in \mathcal{M}_{q,\deg(h)-1}} \chi (f).
\end{eqnarray}
If $\chi$ is even, then 
$\mathcal{L}_q(u, \chi)$ satisfies the functional equation
\begin{equation*} \label{eq:funceq-gen-even}
\mathcal{L}_q(u, \chi) = \omega(\chi) (\sqrt{q}u)^{\deg(h)-2} \frac{1-u}{1-\frac{1}{qu}}\mathcal{L}_q \left(\frac{1}{qu}, \overline{\chi}\right),
\end{equation*}
where the sign of the functional equation is
\begin{equation} \label{sign-even}
\omega(\chi) = -q^{-(\deg(h)-2)/2} \sum_{f \in \mathcal{M}_{\deg(h)-1}} \chi(f).\end{equation}
\end{lem}

\begin{proof} 
From \eqref{L} and \eqref{genus}, 
if $\chi$ is odd, then $g=\deg(h)-1$,
$\mathcal{L}_q(u, \chi)=\mathcal{L}_C(u, \chi)$,
and the functional equation follows from the Weil conjectures, since we have
\begin{eqnarray} 
\mathcal{L}_C(u, \chi) &=& (u \sqrt{q})^{\deg(h)-1} \prod_{j=1}^{\deg(h)-1}  \left( (u \sqrt{q})^{-1} - e^{2 \pi i \theta_j} \right) \nonumber\\
&=&  (u \sqrt{q})^{\deg(h)-1} (-1)^{\deg(h)-1}  \prod_{j=1}^{\deg(h)-1} e^{2 \pi i \theta_j}  \prod_{j=1}^{\deg(h)-1}  \left( 1 - \frac{e^{- 2 \pi i \theta_j}}{u \sqrt{q}} \right)\nonumber \\
&=& (u \sqrt{q})^{\deg(h)-1} (-1)^{\deg(h)-1} \prod_{j=1}^{\deg(h)-1} e^{2 \pi i \theta_j} \mathcal{L}_C \left( \frac{1}{qu}, \overline{\chi} \right).\label{fe*}
\end{eqnarray}
Since
$$ \sum_{n=0}^{\deg(h)-1} u^n \sum_{f \in \mathcal{M}_{q,n}} \chi (f) = \prod_{j=1}^{\deg(h)-1} (1-u \sqrt{q} e^{2 \pi i \theta_j}),$$  
comparing the coefficients of $u^{\deg(h)-1}$, it follows that
$$ \sum_{f \in \mathcal{M}_{q,\deg(h)-1}} \chi (f) = (-1)^{\deg(h)-1} q^{(\deg(h)-1)/2} \prod_{j=1}^{\deg(h)-1} e^{2 \pi i \theta_j},$$ which gives that
$$\omega(\chi) = q^{-(\deg(h)-1)/2} \sum_{f \in \mathcal{M}_{q,\deg(h)-1}} \chi(f).$$

From \eqref{L} and \eqref{genus}, 
if $\chi$ is even,  then $g=\deg(h)-2$, $\mathcal{L}_q(u, \chi)=(1-u)\mathcal{L}_C^*(u)$,
and we have
\begin{eqnarray} 
\mathcal{L}_q(u, \chi) &=& (1-u)(u \sqrt{q})^{\deg(h)-2} \prod_{j=1}^{\deg(h)-2}  \left( (u \sqrt{q})^{-1} - e^{2 \pi i \theta_j} \right) \nonumber \\
&=& (1-u) (u \sqrt{q})^{\deg(h)-2} (-1)^{\deg(h)-2}  \prod_{j=1}^{\deg(h)-2} e^{2 \pi i \theta_j}  \prod_{i=1}^{\deg(h)-2}  \left( 1 - \frac{ e^{-2 \pi i \theta_j}}{u \sqrt{q}} \right) \nonumber \\
&=& \left(\frac{1-u}{1-\frac{1}{qu}}\right)(u \sqrt{q})^{\deg(h)-2} (-1)^{\deg(h)-2} \prod_{j=1}^{\deg(h)-2} e^{2 \pi i \theta_j} \mathcal{L}_q \left(\frac{1}{qu}, \overline{\chi}\right). \label{fe3}
\end{eqnarray}

Since
$$ \sum_{n=0}^{\deg(h)-1} u^n \sum_{f \in \mathcal{M}_{q,n}} \chi (f) = (1-u)\prod_{j=1}^{\deg(h)-2} (1-u \sqrt{q}  e^{2 \pi i \theta_j}),$$ 
comparing the coefficients of $u^{\deg(h)-1}$, it follows that
$$ \sum_{f \in \mathcal{M}_{q,\deg(h)-1}} \chi (f) = (-1)^{\deg(h)-1} q^{(\deg(h)-2)/2} \prod_{j=1}^{\deg(h)-2}  e^{2 \pi i \theta_j},$$ which gives that
$$\omega(\chi) = -q^{-(\deg(h)-2)/2} \sum_{f \in \mathcal{M}_{q,\deg(h)-1}} \chi(f).$$

\end{proof}

It is more natural to rewrite the sign of the functional equation in terms of Gauss sums over $\F_q[T]$.  In particular, it is not obvious from \eqref{sign} and \eqref{sign-even} that $|\omega(\chi)|=1$.  

As in \cite{F17}, we will use the exponential function which was introduced by D. Hayes \cite{hayes}. For any $a \in \mathbb{F}_q((1/T))$, we define 
\begin{equation}
e_q(a) = e^{\frac{2 \pi i \tr_{\F_q/\F_p}(a_1) }{p}},
\label{exp-ff}
\end{equation} with $a_1$ the coefficient of $1/T$ in the Laurent expansion of $a$.
We then have that $e_q(a+b) = e_q(a)  e_q(b)$, and $e_q(a) = 1$ for $a \in \F_q[T]$. Also, if $a,b,h \in \F_q[T]$ are such that  $a \equiv b \pmod h$, then
$e_q(a/h) = e_q(b/h).$

For $\chi$ a primitive character of modulus  $h$ on $\F_q[T]$, let 
\begin{eqnarray*} 
G(\chi) =\sum_{a \pmod h} \chi(a) e_q \left (  \frac{a}{h}\right ) \end{eqnarray*}
 be the Gauss sum of the primitive Dirichlet character $\chi$ over $\F_q[T].$
 The following corollary expresses the root number in terms of Gauss sums.

\begin{cor}\label{sign-GS} Let $\chi$ be a primitive character of modulus  $h$ on $\F_q[T]$. Then
$$\omega(\chi) = \begin{cases} \frac{1}{\tau(\chi)} q^{-(\deg(h) - 1)/2}  G(\chi) & \mbox{ if } \chi  \,\mathrm{ odd},\\
\frac{1}{\sqrt{q}}  q^{-(\deg(h) - 1)/2} G(\chi)  & \mbox{ if } \chi  \,\mathrm{ even}. \end{cases}$$
\end{cor}

\begin{proof} We prove the following relation
$$ G(\chi)= 
\begin{cases}
\tau(\chi) \sum_{f \in \mathcal{M}_{q,\deg(h)-1}} \chi(h)&\mbox{ if } \chi  \,\mathrm{ odd},\\
- q \sum_{f \in \mathcal{M}_{q,\deg(h)-1}} \chi(h) &\mbox{ if } \chi  \,\mathrm{ even},
\end{cases}$$
which clearly implies the corollary.
Writing \[\mathcal{L}_q(u, \chi)= \sum_{j=0}^{\deg(h)-1}a_ju^j, \qquad  a_j=\sum_{\ell \in \mathcal{M}_{q,j}} \chi(\ell),\]
we have
\[\sum_{\ell \pmod h} \chi(\ell) e_q\left(  \frac{\ell}{h}\right )=\sum_{j=0}^{\deg(h)-2}a_j \sum_{a \in \F_q^*}\chi(a)+a_{\deg(h)-1}\sum_{a \in \F_q^*}\chi(a) e^{2 \pi i \tr_{\F_q/\F_p}(a)/p}.\]
\[=\left\{\begin{array}{ll}  \tau(\chi) a_{\deg(h)-1} &\mbox{ if } \chi  \,\mathrm{ odd},\\
(q-1) \sum_{j=0}^{\deg(h)-2} a_j-a_{\deg(h)-1} & \mbox{ if } \chi  \,\mathrm{ even}.
         \end{array} \right.\]
When $\chi$ is even, 1 is a root of $\mathcal{L}_q(u, \chi)$ and therefore $\sum_{j=0}^{\deg(h)-1} a_j=0$. The result follows. 
\end{proof}


The following result allows us to replace the sum \eqref{finite-L-1} by two shorter sums of lengths $A$ and $g-A-1$, where $A$ is a parameter that can be chosen later, where the relationship between $g$ and $\deg(h)$ is given by \eqref{genus}.

\begin{prop} [Approximate Functional Equation] \label{prop-AFE} Let $\chi$ be a primitive cubic character of modulus $h$. 

If $\chi$ is odd, then
\begin{align*}
\mathcal{L}_q \left( \frac{1}{\sqrt{q}} ,\chi \right) =&\sum_{f \in \mathcal{M}_{q, \leq A}} \frac{ \chi(f)}{ q^{\deg(f)/2}} + \omega (\chi) \sum_{f \in \mathcal{M}_{q, \leq g-A-1}} \frac{ \overline{\chi}(f)}{q^{\deg(f)/2}},\\
\end{align*}
where $g=\deg(h)-1$ by \eqref{genus}.

If $\chi$ is even, then
\begin{align*}
\mathcal{L}_q \left( \frac{1}{\sqrt{q}} ,\chi \right) =&\sum_{f \in \mathcal{M}_{q, \leq A}} \frac{ \chi(f)}{ q^{\deg(f)/2}} + \omega (\chi) \sum_{f \in \mathcal{M}_{q, \leq g-A-1}} \frac{ \overline{\chi}(f)}{q^{\deg(f)/2}} \\
&+ \frac{1}{1-\sqrt{q}} \sum_{f \in \mathcal{M}_{q, A+1}} \frac{ \chi(f)}{q^{\deg(f)/2}} + \frac{\omega(\chi)}{1-\sqrt{q}} \sum_{f \in \mathcal{M}_{q,g-A}} \frac{ \overline{\chi}(f)}{q^{\deg(f)/2}},
\end{align*}
where $g=\deg(h)-2$ by  \eqref{genus}.
\end{prop}

\begin{proof}  

For $\chi$ odd, we use Lemma \ref{lemma-FE-gen} for $\chi$ and then we have that
\begin{equation*}
\mathcal{L}_q(u,\chi)= \omega(\chi) (\sqrt{q} u)^g \mathcal{L}_q \left(  \frac{1}{qu}, \overline{\chi} \right ).
\end{equation*}
Using equation \eqref{finite-L-1} and the functional equation above, it follows that
\begin{equation}
\sum_{f \in \mathcal{M}_{q,n}} \chi(f) = \omega(\chi) q^{n - \frac{g}{2}} \sum_{f \in \mathcal{M}_{q,g-n}} \overline{\chi}(f).\label{id1}
\end{equation}
Writing
\begin{align*}
\mathcal{L}_q(u,\chi) = \sum_{n=0}^A u^n \sum_{f \in \mathcal{M}_{q,n}} \chi(f) + \sum_{n=A+1}^g u^n \sum_{f \in \mathcal{M}_{q,n}} {\chi}(f),
\end{align*}
and using \eqref{id1} for the second sum, it follows that
\begin{equation*}
\mathcal{L}_q(u,\chi) = \sum_{f \in \mathcal{M}_{q,\leq A}} \chi(f) u^{\deg(f)} + \omega(\chi) (\sqrt{q} u)^{g} \sum_{f \in \mathcal{M}_{q,\leq g-A-1}} \frac{\overline{\chi}(f)}{(qu)^{\deg(f)}}.
\end{equation*}
Plugging in $u=1/\sqrt{q}$ finishes the proof.

For $\chi$ even we have 
\[\mathcal{L}_q(u,\chi) = \sum_{n=0}^{g+1} a_n u^n, \qquad a_n = \sum_{f \in \mathcal{M}_{q,n}} \chi(f).\] 
We write $$\mathcal{L}_C^{*} (u) = \prod_{j=1}^{g} \left( 1 - u \sqrt{q} e^{2\pi i \theta_j} \right) = \sum_{n=0}^g b_n u^n.$$
 By the functional equation \eqref{fe3},
\begin{align*}
\sum_{n=0}^{g}b_n u^n = &\omega(\chi) (\sqrt{q}u)^g \sum_{n=0}^g \overline{b_n} q^{-n}u^{-n}\\
=& \omega(\chi) \sum_{n=0}^g \overline{b_n} q^{g/2-n}u^{g-n}
=\omega(\chi) \sum_{m=0}^g \overline{b_{g-m}} q^{m-g/2}u^{m}, 
\end{align*}
from where
\[b_n=\omega(\chi) \overline{b_{g-n}} q^{n-g/2}.\]
Thus, we can write
$$ \mathcal{L}_C^{*} (u)  = \sum_{n=0}^{A} b_n u^n + \omega(\chi) (\sqrt{q} u)^g \sum_{n=0}^{g-A-1} \frac{\overline{b_n}}{q^nu^n}.$$
Now since $\mathcal{L}_q(u,\chi) = (1-u) \mathcal{L}_C^{*}(u)$, we get that
$$ a_n = b_n -b_{n-1}$$ for $n = 0, \ldots, g$ and $a_{g+1} = -b_g$. Hence
\begin{equation} b_n = a_0 + \ldots +a_n
\label{eqb}
\end{equation} for $n =0,\ldots, g$. Now plugging in $u = 1/\sqrt{q}$, we get that
$$ \mathcal{L}_q \left( \frac{1}{\sqrt{q}},\chi \right ) = \sum_{n=0}^{A} \frac{b_n}{q^{n/2}} \left ( 1- \frac{1}{\sqrt{q}} \right )  + \omega(\chi) \sum_{n=0}^{g-A-1} \frac{\overline{b_n}}{q^{n/2}} \left ( 1-\frac{1}{\sqrt{q}}\right ).$$

Now using equation \eqref{eqb} for $b_n$ and $b_{n+1}$, substracting the two equations and using the functional equation for $b_n$, we get that
$$ \overline {a_0} + \ldots + \overline{a_{g-n-1}} = \frac{1}{q-1} a_{n+1} \overline{\omega(\chi)} q^{\frac{g}{2}-n} + \frac{\overline{a_{g-n}}}{q-1},$$ and hence
$$  a_0 + \ldots + a_{g-n-1} = \frac{1}{q-1} \overline{a_{n+1}} \omega(\chi) q^{\frac{g}{2}-n} + \frac{a_{g-n}}{q-1}.$$
Now we use the equations above for $n = g-1-A$ and $n = A$ and after some manipulations, we get that
\begin{align*}
\mathcal{L}_q \left( \frac{1}{\sqrt{q}},\chi \right ) =& \sum_{f \in \mathcal{M}_{q,\leq A}} \frac{\chi(f)}{q^{\deg(f)/2}} + \omega(\chi) \sum_{f \in \mathcal{M}_{q,\leq  g-A-1 }}  \frac{\overline{\chi}(f)}{q^{\deg(f)/2}} \\
& + \frac{a_{A+1}}{(1-\sqrt{q})q^{\frac{A+1}{2}}}+ \omega(\chi) \frac{\overline{a _{g-A}} }{(1-\sqrt{q}) q^{\frac{g-A}{2}}},
\end{align*}
and the result follows.

\end{proof}

The following lemmas provide upper and lower bounds for $L$--functions. 
\begin{lem} \label{lindelof}
Let $\chi$ be a primitive cubic character of conductor $h$ defined over $\F_q[T]$. Then, for 
$\re(s) \geq 1/2$ and for all $\varepsilon > 0$,
$$\left|L_q(s, \chi) \right| \ll q^{ \varepsilon \deg(h)}.$$
\end{lem}
\begin{proof}
This is the Lindel\"of hypothesis in function fields. It is Theorem 5.1 in \cite{BCDGLD}. For the quadratic case see also the proof of Corollary $8.2$ in
\cite{F17c} and Theorem 3.3 in \cite{ATsimerman}. 
 
\end{proof}


\begin{lem}\label{folednil}
Let $\chi$ be a primitive cubic character of conductor $h$ defined over $\F_q[T]$. Then, for $\re(s) \geq 1$ and for all $\varepsilon > 0$, 
$$\left| L_q(s, \chi) \right| \gg q^{ - \varepsilon \deg(h)}.$$
\end{lem}

\begin{proof} First assume that $\chi$ is an odd character. 
Recall that $g= \deg(h)-1$. Then
$$ L_q(s,\chi) = \prod_{j=1}^g \left(1- q^{\frac{1}{2}-s}e^{2\pi i \theta_j}\right),$$
and
$$  \frac{1}{\log q} \frac{L_q'}{L_q}(s,\chi)= -g+\sum_{j=1}^g \frac{1}{1-q^{\frac{1}{2}-s}e^{2\pi i \theta_j}}.$$
From the above it follows that if $\re(s) \geq 1$ then 
\begin{equation}
 \left|  \frac{L_q'}{L_q} (s,\chi) \right| \ll \deg(h).
 \label{ineq_deriv}
 \end{equation}
Now for $\re(s)= \sigma>1$ we have
$$ \log L_q(s,\chi) = \sum_{f \in \mathcal{M}_q} \frac{\Lambda(f) \chi(f)}{|f|^s \deg(f)},$$ 
where $\Lambda (f)$ is the von Mangoldt function, equal to $\deg(P)$ when $f=P^n$ for $P$ prime, and zero otherwise. 

Hence
$$ | \log L_q(s,\chi)| \leq \sum_{f \in \mathcal{M}_q} \frac{\Lambda(f)}{|f|^{\sigma} \deg(f)}= \log \zeta_q(\sigma)= - \log\left(1-q^{1-\sigma}\right).$$
If $\sigma \geq 1+ \frac{1}{\deg h}$ then it follows that
\begin{equation}
 | \log L_q(s,\chi)| \ll \log ( \deg(h)).
 \label{ineq_log}
 \end{equation}
Now if $s= 1+it$ and $s_1= 1+\frac{1}{\deg(h)}+it$, we have that
$$ \log L_q(s,\chi) - \log L_q(s_1,\chi) = \int_{s_1}^s \frac{L_q'}{L_q}(z) \, dz \ll |s_1-s| \deg(h) \ll 1,$$ where the first inequality follows from \eqref{ineq_deriv}. Combining the above and \eqref{ineq_log} it follows that when $\re(s)=1$ we have
$$ | \log L_q(s,\chi) | \ll \log( \deg(h)).$$
Now $$ \left|  \log \frac{1}{|L_q(s,\chi)|} \right|= | \re \log L_q(s,\chi) | \leq | \log L_q(s,\chi)| \ll \log( \deg(h)),$$ and then
$$|L_q(s,\chi)| \gg \deg(h)^{-1} \gg q^{-\varepsilon \deg(h)}.$$

When $\chi$ is an even character, the $L$-function has an extra factor of $1-q^{-s}$ which does not affect the bound. 

\end{proof}

Note that using ideas as in the work of Carneiro and Chandee \cite{CC} one could prove that
$$ |L_q(s,\chi)| \gg \frac{1}{ \log (\deg(h))},$$ when $\re(s)=1$. For our purposes the lower bound of $\deg(h)^{-1}$ is enough and we do not have to follow the method in \cite{CC}.

\kommentar{\begin{rem} We will often work with the odd character $\chi_{F_1}\overline{\chi_{F_2}}$ (to be defined in the next subsection), where
$F_1,F_2$ are monic polynomials of degrees $d_1, d_2$, $(F_1,F_2)=1$, and $d_1+2d_2\equiv 1 \pmod{3}$. In this case, the approximate functional equation takes the shape
 \begin{align*}
\mathcal{L}_q \left( \frac{1}{\sqrt{q}},\cF \right ) =&\sum_{f \in \mathcal{M}_{q,\leq A}} \frac{\cFeval{f}}{\sqrt{|f|_q}} + \omega(\chi_{F_1}\overline{\chi_{F_2}}) 
\sum_{f \in \mathcal{M}_{q,\leq g-A-1}} \frac{\cFbeval{f}}{\sqrt{|f|_q}},\\
\end{align*}
where $g=d_1+d_2-1$. 
\end{rem}}

\subsection{Primitive cubic characters over $\F_q[T]$} \label{section-cubic}

Let $q$ be an odd power of a prime. In this section we describe the cubic characters over $\F_q[T]$ when $q \equiv 1 \pmod 3$ (the Kummer case) and $q \equiv 2 \pmod 3$ (the non-Kummer case).

We first suppose that $q$ is odd and $q \equiv 1 \pmod 3$.

We define the cubic residue symbol $\chi_P$, for $P$ an irreducible monic polynomial in $\F_q[T]$.  Let $a \in \F_q[T]$. If $P \mid a$, then $\chi_P(a) = 0$, and otherwise
$\chi_{P}(a) = \alpha,$
where $\alpha$ is the unique root of unity in $\mathbb{C}$ such that
$$ a^{\frac{q^{\deg(P)}-1}{3}} \equiv \Omega(\alpha) \pmod P.$$
We extend the definition by multiplicativity to any monic polynomial $F \in \F_q[T]$ by defining for $F = P_1^{e_1} \dots  P_s^{e_s}$, with distinct primes $P_i$, 
\begin{equation*}\label{CRS-def}
\chi_F = \chi_{P_1}^{e_1} \dots \chi_{P_s}^{e_s}.
\end{equation*}

Then, $\chi_F$ is a cubic character modulo $P_1 \dots P_s$. It is primitive if and only if all the $e_i$ are 1 or 2. Then it follows that the conductors of the primitive cubic characters are the square-free monic polynomials $F \in  \F_q[T]$, and for each such conductor, there are $2^{\omega(F)}$ characters, where $\omega(F)$ is the number of primes dividing $F$. More precisely, for any conductor $F = F_1 F_2$ with $(F_1, F_2)=1$ we have the primitive character of modulus $F$ given by
$$
\chi_{F_1 F_2^2} = \chi_{F_1} \chi_{F_2}^2 =  \chi_{F_1} \overline{\chi_{F_2}}.
$$

\begin{lem} \label{count-kummer} Suppose $q \equiv 1 \pmod 3$, and let ${N}_{\mathrm{K}}(d)$
be the number of primitive cubic characters with conductor of degree $d$. Then, 
$${N}_{\mathrm{K}}(d) = B_{\mathrm{K},1} d q^{d} + B_{\mathrm{K},2} q^d + O \left( q^{(1/2+\varepsilon) d} \right),$$
where $B_{{\mathrm{K}},1}=\mathcal{F}_{\mathrm{K}}(1/q)$, $B_{\mathrm{K}, 2}=\left( \mathcal{F}_{\mathrm{K}}(1/q) - \frac{1}{q} \mathcal{F}_{\mathrm{K}}'(1/q) \right)$, and $\mathcal{F}_{\mathrm{K}}$ is given by \eqref{mathcalF-K}.
\end{lem}

\begin{proof} Let $a(F)$ be the number of cubic primitive characters of conductor $F$. By the above discussion, the generating series for $a(F)$ is given by
$$\mathcal{G}_{\mathrm{K}}(u) = \sum_{F \in \mathcal{M}_q} a(F) u^{\deg(F)} = \prod_{P} \left( 1 + 2 u^{\deg(P)} \right),
$$
which is analytic for $|u|<1/q$ with a double pole at $u=1/q$.
We write
\begin{eqnarray} \label{mathcalF-K}
\mathcal{F}_{\mathrm{K}}(u) = \mathcal{G}_{\mathrm{K}}(u) (1-qu)^2 = \prod_{P} \left( 1 - 3 u^{2\deg(P)} + 2 u^{3 \deg(P)} \right) .
\end{eqnarray}
Then, using Perron's formula (Lemma \ref{perron}), and moving the integral from $|u|=q^{-2}$ to $|u|=q^{-(1/2 + \varepsilon)}$ while picking the residue of the (double) pole at $u=q^{-1}$, we have
\begin{eqnarray*}
{N}_{\mathrm{K}}(d)  &=& \frac{1}{2 \pi i} \oint_{|u|=q^{-2}} \frac{\mathcal{F}_{\mathrm{K}}(u)}{u^d(1-qu)^2} \frac{du}{u}\\
&=& \mathcal{F}_{\mathrm{K}}(1/q) d q^d + \left( \mathcal{F}_{\mathrm{K}}(1/q) - \frac{1}{q} \mathcal{F}_{\mathrm{K}}'(1/q) \right) q^d + O \left( q^{(1/2+\varepsilon)d} \right).
\end{eqnarray*}
\end{proof}

For each primitive cubic character $\chi_{F_1 F_2^2}$, we have that
for $\alpha \in \F_q^*$,
$$
\chi_{F_1 F_2^2}(\alpha) =  \Omega^{-1} \left( \alpha^{\frac{q-1}{3} \left(  \deg(F_1) + 2 \deg(F_2) \right)}  \right),
$$
and  $\chi_{F_1 F_2^2}$  is even if and only if $\deg(F_1) + 2 \deg(F_2) \equiv 0 \pmod 3$. If  $\chi_{F_1 F_2^2}$  is odd, 
the restriction to $\F_q^*$ is $\chi_3$ when $\deg(F_1) + 2 \deg(F_2)  \equiv 1 \pmod 3$, and 
 $\chi_3^2$ when $\deg(F_1) + 2 \deg(F_2)  \equiv 2 \pmod 3$,  where $\chi_3$ is defined by \eqref{def-chi3}.
 
Then, since the conductor of  $\chi_{F_1 F_2^2}$ is $F=F_1 F_2$, we have from \eqref{genus} that
$$
\deg(F_1) + \deg(F_2) = \begin{cases} g + 2 & \deg(F_1) + 2 \deg(F_2) \equiv 0 \pmod 3, \\   g + 1 & \deg(F_1) + 2 \deg(F_2) \not\equiv 0 \pmod 3. \end{cases}
$$
For convenience, recall that we restrict to the odd cubic primitive characters such that the restriction to $\F_q^*$ is $\chi_3$.

We have then showed the following.
 \begin{lem}  \label{kummer-desc}
Suppose $q$ is odd and $q \equiv 1 \pmod 3$. Then,
\begin{eqnarray*}
\sum_{\substack{\chi  \; \mathrm{ primitive \; cubic} \\ \mathrm{genus}(\chi) = g\\ \chi \vert_{\F_q^*} = \chi_3}} L_q(1/2, \chi) =  
 \sum_{\substack{d_1+d_2 = g+1\\ d_1 + 2 d_2 \equiv 1 \pmod 3}} \sum_{\substack{F_1 \in \mathcal{H}_{q, d_1} \\ F_2 \in \mathcal{H}_{q, d_2} \\ (F_1, F_2) = 1}} L_q \left( \frac{1}{2}, \chi_{F_1} \overline{\chi_{F_2}} \right)  ,
\end{eqnarray*}
and the sign of the functional equation of  $L_q(s, \chi_{F_1} \overline{\chi_{F_2}})$ is equal to
\begin{equation*}\label{eq:omega}
\omega(\chi_{F_1}\overline{\chi_{F_2}})= \overline{\epsilon(\chi_3)}  \;q^{-(d_1+d_2)/2}G(\chi_{F_1} \overline{\chi_{ F_2}})
\end{equation*}
where $\chi_3$ is the cubic residue symbol on $\F_q^*$ defined by \eqref{def-chi3} and $\epsilon(\chi_3)$ is defined by \eqref{sign-restriction}.
\end{lem}


We now suppose that $q \equiv 2 \pmod 3$. Then there are no cubic characters modulo $P$ for primes of odd degree since $3 \nmid q^{\deg(P)} - 1$. For each prime $P$ of even degree and $a \in \mathbb{F}_q[T]$, we have the cubic residue symbol
$\chi_{P}(a) = \alpha,$
where $\alpha$ is the unique cubic root of unity in $\mathbb{C}$ such that
$$ a^{\frac{q^{\deg(P)}-1}{3}} \equiv \Omega(\alpha) \pmod P,$$
where $\Omega$ takes values in the cubic roots of unity in $\F_{q^2}$. 

We extend the definition by multiplicativity to any monic polynomial $F \in \F_q[T]$ supported on primes of even degree by defining for $F = P_1^{e_1} \dots  P_s^{e_s}$, with distinct primes $P_i$ of even degree,
\begin{equation*}\label{CRS-def-q^2}
\chi_F = \chi_{P_1}^{e_1} \dots \chi_{P_s}^{e_s}.
\end{equation*}
Then, $\chi_F$ is a cubic character modulo $P_1 \dots P_s$, and it is primitive if and only if all the $e_i$ are 1 or 2. It follows that the conductors of the primitive cubic characters are the square-free polynomials $F \in  \F_q[T]$ supported on primes of even degree, and for each such conductor, there are $2^{\omega(F)}$ characters, where $\omega(F)$ is the number of primes dividing $F$. 

\begin{lem} \label{count-non-kummer} Suppose $q \equiv 2 \pmod 3$, and let ${N}_{\mathrm{nK}}(d)$
be the number of primitive cubic characters with conductor of degree $d$. Then, 
$${N}_{\mathrm{nK}}(d) = \begin{cases} B_{\mathrm{nK}} q^{d} + O \left( q^{(1/2+\varepsilon) d} \right) & 2 \mid d, \\
0 & \mbox{otherwise.} \end{cases}$$
where $B_{\mathrm{nK}} = \mathcal{F}_{\mathrm{nK}}(1/q)$ and $\mathcal{F}_{\mathrm{nK}}(u)$ is defined by \eqref{mathcalF}.
\end{lem}
\begin{proof} Let $a(F)$ be number of cubic primitive characters of conductor $F$. By the above discussion, the generating series for $a(F)$ is given by 
$$\mathcal{G}_{\mathrm{nK}}(u) = \sum_{F \in \mathcal{M}_q} a(F) u^{\deg(F)} = \prod_{2 \mid \deg(P)} \left( 1 + 2 u^{\deg(P)} \right),
$$
which is analytic for $|u|<1/q$  with simple poles at $u=1/q$ and $u=-1/q$. This follows from the fact that the primes of even degree in $\F_q[T]$ are exactly the primes splitting in the quadratic extension $\F_{q^2}(T)/\F_q(T).$ 
Recall that  
$$\mathcal{Z}_{q^2}(u^2) = \prod_{2 \mid \deg(P)} \left(1-u^{\deg(P)} \right)^{-2} \prod_{2 \nmid \deg(P)} \left(1 - u^{2 \deg(P)} \right)^{-1},$$
where $u = q^{-s}$ and the product is over primes $P$ of $\F_q[T]$. The analytic properties of $\mathcal{G}_{\mathrm{nK}}(u)$ then follow from the analytic properties of $\mathcal{Z}_{q^2}(u^2)$, which is analytic everywhere except for simple poles when $u^2=q^{-2}$.

We write
\begin{eqnarray} \nonumber
\mathcal{F}_{\mathrm{nK}}(u )&=& \mathcal{G}_{\mathrm{nK}}(u) (1-qu) (1+qu) \\ \nonumber &=& \prod_{2 \mid \deg(P)}  \left( 1 + 2 u^{\deg(P)} \right)  \left( 1 -  u^{\deg(P)} \right)^2
 \prod_{2 \nmid \deg(P)}  \left( 1 - u^{\deg(P)} \right)  \left( 1 +  u^{\deg(P)} \right) \\
\label{mathcalF}
&=& \prod_{2 \mid \deg(P)} \left( 1 -3 u^{2\deg(P)} +2 u^{3 \deg(P)} \right) \prod_{2 \nmid \deg(P)} \left( 1 - u^{2\deg(P)} \right),
\end{eqnarray}
which is analytic for $|u| < q^{-1/2}.$
Then, using Perron's formula (Lemma \ref{perron}), and moving the integral from $|u|=q^{-2}$ to $|u|=q^{-(1/2 + \varepsilon)}$ while picking the poles at $u=\pm q^{-1}$ , we have
\begin{eqnarray*}
{N}_{\mathrm{nK}}(d)  &=& \frac{1}{2 \pi i} \oint_{|u|=q^{-2}} \frac{\mathcal{F}_{\mathrm{nK}}(u)}{u^d(1-qu)(1+qu)} \frac{du}{u} \\
&=& \left( \frac{\mathcal{F}_{\mathrm{nK}}(1/q)}{2} +  (-1)^{d} \frac{\mathcal{F}_{\mathrm{nK}}(-1/q)}{2} \right) q^{d} +  O \left( q^{(1/2+\varepsilon)d} \right).
\end{eqnarray*}
Notice that $\mathcal{F}_{\mathrm{nK}}(1/q)=\mathcal{F}_{\mathrm{nK}}(-1/q)$, so the main term is zero when $d$ is odd. In this case, we already knew that there are no primitive cubic characters with conductor of odd degree as every prime which divides the conductor has even degree. For $d$ even, this proves the result.
\end{proof}

It is more natural to describe these characters as characters over $\F_{q^2}[T]$ restricting to characters over $\F_q[T]$ as in the work of  Bary-Soroker and Meisner  \cite{BSM} (generalizing the work of Baier and Young \cite{BY} from number fields to function fields) by counting characters of $\F_{q^2}[T]$ whose restrictions to $\F_q[T]$ are cubic characters over $\F_q[T]$. In what follows, for $f$ in the quadratic extension $\F_{q^2}[T]$ over $\F_{q}[T]$,  we will denote by $\tilde{f}$ the Galois conjugate of $f$.

Notice that $q^2 \equiv 1 \pmod 3$, and we have then described the primitive cubic characters of $\F_{q^2}[T]$ in the paragraph before Lemma \ref{count-non-kummer}.
Supose that $\pi$ is a prime in $\F_{q^2}[T]$ lying over a prime $P \in \F_{q}[T]$ such that $P$ splits as $\pi \tilde{\pi}$. Notice that  $P$ splits in $\F_{q^2}[T]$ if and only if the degree of $P$ is even.
It is easy to see that the restriction of $\chi_{\pi}$ to $\F_q[T]$ is the character $\chi_{P}$, and the restriction of $\chi_{\tilde{\pi}}$ to $\F_q[T]$ is the character $\overline{\chi_{P}}$ (possibly exchanging $\pi$ and $\tilde{\pi}$). Then by running over all the characters $\chi_{F}$ where $F \in \F_{q^2}[T]$ is square-free and not divisible by a prime $P$ of $\F_q[T]$, we are counting exactly the characters over $\F_{q^2}[T]$ whose restrictions are cubic characters over $\F_q[T]$, and each character over $\F_q[T]$ is counted exactly once. For more details, we refer the reader to  \cite{BSM}.

We also remark that any cubic character over $\F_q[T]$ is even when $q \equiv 2 \pmod 3$. Indeed, by the classification above, such a character comes from $\chi_{F}$ with $F \in \F_{q^2}[T]$, and for $\alpha \in \F_q \subseteq \F_{q^2}$,
 we have
$$\chi_F( \alpha) = \Omega^{-1} \left( \alpha^{\frac{q^2-1}{3} \deg(F)} \right).
$$
Since $q$ is odd and $q \equiv 2 \pmod 3$,  we have that $(q-1) \mid (q^2-1)/3.$ 

By  \eqref{genus}, if $F \in \F_q[T]$ is the conductor of a cubic primitive character $\chi$ over $\F_q[T]$,  it follows that $\deg(F) = g + 2$. By the classification above, it follows that $F = P_1 \dots P_s$ for distinct primes of even degree, and the character $\pmod F$ is the restriction of a character of conductor $\pi_1 \dots \pi_s$ over $\F_{q^2}[T]$, where $\pi_i$ is one of the primes lying above $P_i$. Then the degree of the conductor of this character over $\F_{q^2}[T]$ is equal to $g/2 + 1$.

We have then proved the following result.

\begin{lem} \label{non-kummer-desc}
Suppose $q \equiv 2 \pmod 3$. Then,
\begin{eqnarray*}
\sum_{\substack{\chi  \; \mathrm{primitive \; cubic} \\ \mathrm{genus}(\chi) = g}} L_q(1/2, \chi) =  \sum_{\substack{F \in \mathcal{H}_{q^2, g/2+1} \\
P \mid F \Rightarrow P \not\in \F_q[T]}} L_q(1/2, \chi_F).
\end{eqnarray*}
\end{lem}

\subsection{Generalized cubic Gauss sums and the Poisson summation formula}
\label{cubic-GS}

Let $\chi_f$ be the cubic residue symbol defined before for $f \in \F_q[T]$. This is a character of modulus $f$, but not necessarily primitive.
We define the generalized cubic Gauss sum by
\begin{eqnarray} \label{gen-GS} 
G_q(V,f) = \sum_{u \pmod f} \chi_f(u) e_q\left( \frac{uV}{f} \right ), \end{eqnarray}
with the exponential function defined in \eqref{exp-ff}.
We remark that if $\chi_f$ has conductor $f'$ with $\deg(f')<\deg( f)$, then 
$G(\chi_f) \neq G_q(1, f)$.

If $(a,f)=1$, we have
\begin{equation} \label{prop-GS} 
G_q(aV, f) = \overline{\chi_f}(a) G_q(V, f) .\end{equation}

The following lemma shows that the shifted Gauss sum is almost multiplicative as a function of $f$, and we can determine it on powers of primes. We have the following.

\begin{lem} \label{gauss}
Suppose that $q \equiv 1 \pmod 6$.
\begin{itemize}
 \item[(i)] If $(f_1 ,f_2)=1$, then 
\begin{eqnarray*} G_q(V,f_1 f_2) &=& \chi_{f_1}(f_2)^2 G_q(V, f_1 )G_q(V, f_2) \\ &=& G_q(Vf_2, f_1) G_q(V, f_2). \end{eqnarray*}
\item[(ii)] If $V=V_1P^\alpha$ where $P\nmid V_1$, then 
\[G_q(V, P^i) = \left\{\begin{array}{ll}
0 & \mbox{ if }i \leq \alpha \mbox{ and } i \not \equiv 0 \pmod{3},\\
                      \phi(P^{i})& \mbox{ if }i \leq \alpha \mbox{ and } i \equiv 0 \pmod{3},\\
                        -|P|_q^{i-1}&\mbox{ if }i=\alpha+1 \mbox{ and } i \equiv 0 \pmod{3},\\
                       \epsilon(\chi_{P^i})  \omega(\chi_{P^i}) \chi_{P^i}(V_1^{-1}) |P|_q^{i-\frac{1}{2}} &\mbox{ if }i=\alpha+1 \mbox{ and }  i \not \equiv 0 \pmod{3},\\
                                               0 & \mbox{ if }i \geq \alpha+2,
                         \end{array}
 \right.
\]
\end{itemize}
where $\phi$ is the Euler $\phi$-function for polynomials. We recall that $\epsilon(\chi)=1$ when $\chi$ is even. For the case of $\chi_{P^i}$, this happens if $3 \mid \deg(P^i)$.  
\end{lem}
\begin{proof}
The proof of (i) is the same as in \cite{F17}. 
We write $u \pmod {f_1f_2}$ as $u=u_1 f_1 + u_2 f_2$ for $u_1 \pmod {f_2}$ and $u_2 \pmod {f_1}$. Then,
\begin{align*}
G_q(V, f_1f_2)=& \chi_{f_2}(f_1) \chi_{f_1}(f_2) \sum_{u_1 \pmod f_2}\sum_{u_2 \pmod f_1}\chi_{f_1}(u_2)\chi_{f_2}(u_1)e_q\left(\frac{u_1V}{f_2}\right)e_q\left(\frac{u_2V}{f_1}\right)\\
=&\overline{\chi_{f_1}}(f_2) G_q(V, f_1) G_q(V, f_2)
\end{align*}
by cubic reciprocity. The second line of (i) follows from \eqref{prop-GS}.

Now we focus on the proof of (ii).

Assume that $i \leq \alpha$. Then 
$$G_q(V,P^i)= \sum_{u \pmod {P^i}} \chi_{P^i}(u) e_q(uV_1 P^{\alpha-i}).$$
The exponential above is equal to $1$ since $uV_1P^{\alpha-i} \in \F_q[T]$, and if $i \equiv 0 \pmod 3$, then $\chi_{P^i}(u)=1$ when $(u,P)=1$. The conclusion easily follows in this case. If $i \not \equiv 0 \pmod 3$, the conclusion also follows easily from orthogonality of characters.

Now assume that $i = \alpha+1$. Write $u \pmod {P^i}$ as $u = PA+C$, with $A \pmod {P^{i-1}}$ and $C \pmod P$. Then
$$G_q(V, P^i) = \sum_{A \pmod {P^{i-1}}} \sum_{C \pmod P} \chi_{P^i}(C) e_q\left( \frac{CV_1}{P} \right )= |P|_q^{i-1} \chi_{P^i}(V_1^{-1}) \sum_{C \pmod P} \chi_{P^i}(C) e_q\left( \frac{C}{P}\right ) .$$
If $i \equiv 0 \pmod 3$, then $\chi_{P^i}(V_1^{-1}) = 1$ and
$$  \sum_{C \pmod P} \chi_{P^i}(C) e_q\left( \frac{C}{P}\right ) = \sum_{\substack{C  \pmod P \\ C \neq 0}}  e_q\left( \frac{C}{P}\right )  = -1,$$ and the conclusion follows. So assume that $i \not \equiv 0 \pmod 3$. 
Then
\[ \sum_{C \pmod P} \chi_{P^i}(C) e_q\left( \frac{C}{P}\right ) = \left\{\begin{array}{ll}-q \sum_{f \in \mathcal{M}_{q,\deg(P)-1}} \chi_{P^i}(f) & 3\mid \deg(P), \\
           \epsilon(\chi_{P^i}) \sqrt{q} \sum_{f \in \mathcal{M}_{q,\deg(P)-1}} \chi_{P^i}(f) &3 \nmid \deg(P),
\end{array} \right. \]
and using Lemma \ref{sign-FE}, we can rewrite this as
\[\sum_{C \pmod P} \chi_{P^i}(C) e_q\left( \frac{C}{P}\right ) = \left\{\begin{array}{ll}\omega(\chi_{P^i})q^{\deg(P)/2} & 3\mid \deg(P), \\
          \epsilon(\chi_{P^i}) \omega(\chi_{P^i})q^{\deg(P)/2} &3 \nmid \deg(P). \end{array} \right.\]


Thus, we get 
\[G_q(V, P^i) = \left\{\begin{array}{ll}\omega(\chi_{P^i}) \chi_{P^i}(V_1^{-1})  |P|_q^{i-1/2}   &3 \mid \deg(P),\\ \\ \epsilon(\chi_{P^i}) \omega(\chi_{P^i}) \chi_{P^i}(V_1^{-1})  |P|_q^{i-1/2}   &3 \nmid \deg(P). \end{array} \right.\]

If $ i \geq \alpha+2$, then again the proof goes through exactly as in \cite{F17}.
\end{proof}

Now we state the Poisson summation formula for cubic characters. 
Recall that for any non-principal character on $\F_q^{*}$, $\tau(\chi)$ is the standard Gauss sum defined over $\F_q$ by equation \eqref{standard-GS}.
Also recall that for $\chi$ odd,  $|\tau(\chi)| = \sqrt{q}$, and $\tau(\chi)= \epsilon(\chi) \sqrt{q}$.
For $\chi$ even,  $\epsilon(\chi)=1$.




\begin{prop} \label{prop-Poisson} Let $f$ be a monic polynomial in $\F_q[x]$ with $\deg(f)=n$, and let $m$ be a positive integer.
 If $\deg(f) \equiv 0 \pmod{3}$, then
\[\sum_{h \in \mathcal{M}_{q,m}}\chi_{f }(h)=\frac{q^m}{|f|_q}\left[ G_q(0, f)+(q-1)\sum_{V \in \mathcal{M}_{q,\leq n-m-2}} G_q(V,f) - \sum_{V \in \mathcal{M}_{q,n-m-1}}G_q(V,f)\right].\]
If $\deg(f)  \not \equiv 0 \pmod{3}$, then
\[\sum_{h \in \mathcal{M}_{q,m}}\chi_{f }(h)=\frac{q^{m+\frac{1}{2}}}{|f|_q} \overline{\epsilon({\chi_f})} \sum_{V \in \mathcal{M}_{q,n-m-1}}G_q(V,f).\]
\end{prop}

\begin{proof}
As in \cite{F17}, we have
$$ \sum_{h \in \mathcal{M}_{q,m}} \chi_f(h) = \frac{q^m}{|f|_q} \sum_{\deg(V) \leq n-m-1} G_q(V,f) e_q\left(- \frac{Vx^m}{f} \right ).$$ 

Using \eqref{prop-GS}, we have
\begin{align*}
 \sum_{h \in \mathcal{M}_{q,m}} \chi_f(h)  = & \frac{q^m}{|f|_q} \Big[ G_q(0,f) + \sum_{a=1}^{q-1} \overline{\chi_f}(a) \sum_{V \in \mathcal{M}_{q,\leq n-m-2}} G_q(V,f) \\
 &+ 
 \sum_{a=1}^{q-1} \overline{\chi_f}(a) e^{-2 \pi i \text{tr}_{\F_q / \mathbb{F}_p}(a)/p} \sum_{V \in \mathcal{M}_{q,n-m-1}} G_q(V,f) \Big].
\end{align*} 

Now if $\deg(f) \equiv 0 \pmod 3$ then $\chi_f$ is an even character, and 
$$ \sum_{a=1}^{q-1} \overline{\chi_f}(a) =q-1, \;\;\; \sum_{a=1}^{q-1} \overline{\chi_f}(a) e^{-2 \pi i \text{tr}_{\F_q / \mathbb{F}_p}(a)/p} =-1.$$
If $\deg(f) \not \equiv 0 \pmod 3$ then $\chi_f$ is an odd character, and 
$$ \sum_{a=1}^{q-1} \overline{\chi_f}(a) =0, \;\;\; \sum_{a=1}^{q-1} \overline{\chi_f}(a) e^{-2 \pi i \text{tr}_{\F_q / \mathbb{F}_p}(a)/p} = \overline{\tau(\chi_f)}.$$
Also, if $\deg(f) \not \equiv 0 \pmod 3$, then $f$ is not a cube, and the character $\chi_f$ is non-trivial, which implies that $G_q(0, f)=0$ by the orthogonality relations.
\end{proof}

\section{Averages of cubic Gauss sums} \label{sec:GS}

In this section we prove several results concerning averages of cubic Gauss sums which will be needed later. Assume throughout that $q \equiv 1 \pmod 6$. 
For $a, n  \in \Z$ and $n$ positive, we  denote by $[a]_n$ the residue of $a$ modulo $n$ such that
$0 \leq [a]_n \leq n-1$.

We will prove the following.
\begin{prop} \label{big-F-tilde-corrected}
Let $f=f_1 f_2^2 f_3^3$ with $f_1$ and $f_2$ square-free and coprime.
We have 
\begin{align*} \nonumber
\sum_{\substack{F \in \mathcal{M}_{q,d}\\(F,f)=1}} G_q(f,F) = & 
\delta_{f_2=1} \frac{ q^{\frac{4d}{3} - \frac{4}{3} [d+ \deg(f_1)]_3} }{ \zeta_q(2)|f_1|_q^{2/3}} \overline{G_q(1,f_1)} \rho(1, [d+ \deg(f_1)]_3)\prod_{P\mid f_1 f_3^*} \left ( 1+\frac{1}{|P|_q}\right )^{-1} \\ \nonumber
&+ O \left ( \delta_{f_2=1} \frac{q^{\frac{d}{3}+\varepsilon d}} {|f_1|_q^{\frac{1}{6}}}\right )
+ \frac{1}{2\pi i} \oint_{|u|=q^{-\sigma}} \frac{\tilde{\Psi}_q(f,u)}{u^d}\frac{du}{u} 
\end{align*}
with $2/3<\sigma< 4/3$ and where $\Tilde{\Psi}_q(f,u)$ is given by \eqref{tilde-F} and $\rho(1, [d+ \deg(f_1)]_3)$ is given by \eqref{residue}. 

Moreover, we have 
\[\frac{1}{2\pi i} \oint_{|u|=q^{-\sigma}} \frac{\tilde{\Psi}_q(f,u)}{u^d}\frac{du}{u} \ll q^{\sigma d} |f|_q^{\frac{1}{2} ( \frac{3}{2} - \sigma)}.\]

\end{prop}
To prove Proposition \ref{big-F-tilde-corrected} we first need to understand the generating series of the Gauss sums. Let
\[\Psi_q(f,u)= \sum_{F \in \mathcal{M}_q} G_q(f,F)u^{\deg(F)}.\]
and 
\begin{equation} \label{tilde-F}
\Tilde{\Psi}_q(f,u) = \sum_{\substack{F \in \mathcal{M}_q \\ (F,f)=1}} G_q(f,F) u^{\deg(F)}.
\end{equation}
The function $\Psi_q(f,u)$ was studied by Hoffstein \cite{hoffstein}, and we will cite here the relevant results that we need, following the notation of Patterson \cite{patterson}. We postpone the proof of Proposition \ref{big-F-tilde-corrected} to the next sections.

\subsection{The work of Hoffstein and Patterson}

We first study the general Gauss sums associated to the $n^{\text{th}}$ residue symbols as done in \cite{hoffstein, patterson}, and we specialize to $n=3$ later. We always assume that $q \equiv 1 \pmod n$. 
Let $\eta \in (\F_q((1/T))^\times$ and define
\[\psi(f,\eta, u)=(1-u^nq^{n})^{-1} \sum_{\substack{F\in \mathcal{M}_q\\ F \sim \eta}} G_q(f, F)u^{\deg(F)},\]
where the equivalence relation is given by 
\[F\sim \eta \Leftrightarrow F/\eta \in (\F_q((1/T))^\times)^n.\]
There is difference between our definition of  $\psi(f,\eta, u)$ above, and the definition of $\psi(r,\eta, u)$ in \cite[p. 245]{patterson}:  we are summing over monic polynomials in $\F_q[T]$, and not all polynomials in $\F_q[T]$, as in 
\cite{hoffstein}. This explains the extra factors of the type $(q-1)/n$ which appear in \cite{patterson}.
Because our polynomials are monic, it is enough to consider the equivalence classes 
that separate degrees, namely $\eta=\pi_\infty^{-i}$, where $\pi_\infty$ is the uniformizer of the prime at infinity, i.e. $T^{-1}$ in the completion $\F_q((1/T))$.

A little bit of basic algebra in $\F_q((1/T))$ shows that for any $i \in \Z$, 
\[\psi(f,\pi_\infty^{-i}, u)=(1-u^nq^{n})^{-1} \sum_{\substack{F\in \mathcal{M}_q\\ \deg(F) \equiv i \pmod n}} G_q(f, F)u^{\deg(F)}.\]

Then $\psi(f,\pi_\infty^{-i}, u)$ depends only on the value of $i$ modulo $n$. 

We remark that since we have fixed the map between the $n^{\text{th}}$ roots of unity in $\F_q^*$ and $\mu_n \subseteq \C^*$ at the beginning of this paper,
we do not make this dependence explicit in our notation, as it is done in \cite{patterson}.

Then we can write the generating series $\Psi_q(f,u)$ as
\begin{eqnarray} \label{gen-series}
 \Psi_q(f,u)= (1-u^n q^n) \sum_{i=0}^{n-1} \psi(f, \pi_\infty^{-i},u).\end{eqnarray}
The main result of Hoffstein is a functional equation for $\psi(f, \pi_\infty^{-i}, u)$ \cite[Proposition 2.1]{hoffstein}, which we write below using the notation of Patterson.  

\begin{prop}{\cite[Proposition 2.1]{hoffstein}} \label{prop-FE}
For $0 \leq i < n$ and $f \in \mathcal{M}_q$, we have
\begin{eqnarray*}
&&q^{is} \psi(f, \pi_\infty^{-i}, q^{-s}) = q^{n(s-1)E} q^{(2-s)i}  \psi(f, \pi_\infty^{-i}, q^{s-2}) \frac{1-q^{-1}}{ (1 - q^{ns-n-1})} \\
&&\;\;\; + W_{f,i} q^{n(2-s)(B-2)} q^{2n-\deg(f) + 2i-2} q^{(2-s) \left[ 1 + \deg(f)-i \right]_n}  \psi(f, \pi_\infty^{i-1-\deg(f)}, q^{s-2})  \frac{1-q^{n-ns}} {1 - q^{ns-n-1}},
\end{eqnarray*}
 where $B=[(1+\deg(f)-i)/n]$, $E = 1 - [(\deg(f)+1-2i)/n]$, and $W_{f,i} = \tau( \chi_3^{2i-1} \overline{\chi_f})$ 
with $\chi_3$ given by equation \eqref{def-chi3}.
\end{prop}
\begin{rem} \label{remark-fe}
Note that we can rewrite the functional equation in the following form (for $n=3$)
\begin{align} \label{fe-2}
(1-q^4u^3)  \psi(f,\pi_{\infty}^{-i},u) = |f|_q u^{\deg(f)} \left[ a_1(u) \psi \left (f, \pi_{\infty}^{-i}, \frac{1}{q^2u}\right )  +a_2(u) \psi \left (f, \pi_{\infty}^{i-1-\deg(f)}, \frac{1}{q^2u}\right ) \right],
\end{align}
where
$$a_1(u) = -(q^2u) (qu)^{- [\deg(f)+1-2i]_3} (1-q^{-1}), \, \, a_2(u) = -W_{f,i} (qu)^{-2} (1-q^3u^3),$$ with $W_{f,i}$ as above. 
\end{rem}

By setting $u=q^{-s}$ and letting $u \rightarrow \infty$ in the functional equation, Hoffstein showed that
\begin{equation} \label{meromorphic}
\psi(f, \pi_\infty^{-i},u) = \frac{u^i P(f, i, u^n)}{(1-q^{n+1} u^n)},
\end{equation}
where $P(f, i, x)$ is a polynomial of degree at most $[(1+\deg(f)-i)/n]$ in $x$. We remark that while  $\psi(f, \pi_\infty^{-i},u)$ depends only on the value of $i$ modulo $n$, 
this is not the case for $P(f, i, u^n)$.

\begin{rem} \label{no-pole}
Note that, from \eqref{meromorphic}, the left-hand side of equation \eqref{fe-2} above has no pole at $u^3=1/q^2$, so neither does the right-hand side.
\end{rem}

We let 
$$C(f, i) = \sum_{F \in \mathcal{M}_{q,i}} G_q(f, F),$$
By setting $x=u^n=q^{-ns}$, we can write for $0 \leq i \leq n-1$, 
$$ P(f, i, x) = \frac{1-q^{n+1}x}{1-q^nx}  \sum_{j \geq 0} C(f,  i+nj) x^j.$$
If $j \geq [(1+\deg(f)-i)/n]$ with $0 \leq i \leq n-1$, then we have the recurrence relation
$$C(f,  i+n(j+1)) = q^{n+1} C(f, i + nj).$$
Using that, we can rewrite, for any $B \geq [(1+\deg(f)-i)/n]$, 
\begin{eqnarray}
\nonumber
 P(f,  i, x) &=& \frac{1-q^{n+1}x}{1-q^nx}  \left( \sum_{0 \leq j < B } C(f, i+nj) x^j +  \sum_{j \geq B} 
C(f,  i + n B) (q^{n+1})^{j-B} x^j \right) \\ \nonumber
&=& \frac{1-q^{n+1}x}{1-q^nx}  \left( \sum_{0 \leq j < B} C(f, i+nj) x^j +  C(f, i + n B) x^{B} \sum_{j \geq 0}  (q^{n+1})^{j} x^j \right) \\
\label{formula-for-P}
&=& \frac{1-q^{n+1}x}{1-q^nx}  \sum_{0 \leq j < B} C(f, i+nj) x^j +  \frac{C(f, i + n B)}{1-q^nx} x^{B}.
\end{eqnarray}
Let
\begin{equation} \rho(f,i)=\lim_{s \rightarrow 1+\frac{1}{n}} (1-q^{n+1-ns}) q^{is} \psi(f, \pi_\infty^{-i}, q^{-s}) =
P(f, i, q^{-n-1}). \label{residue}
\end{equation}
Using the formula above for  $P(f, i, x)$, it follows that
\begin{eqnarray*} \label{bound-for-residue}
\rho(f,i) = \frac{C(f, i')}{(1-q^{-1})q^{\frac{n+1}{n} (i'-i)}},
\end{eqnarray*}
where $i' \equiv i \pmod n$, and $i' \geq \deg(f).$

To prove Proposition \ref{big-F-tilde-corrected} we need to obtain an explicit formula for the residue in equation \eqref{residue} which we do  in the next subsection.

\subsection{Explicit formula for the residue $\rho(f,i)$}
From now on, we will specialize to $n=3$.

For $\pi$ prime, following Patterson's notation, let  
$$ \psi_{\pi}(f, \pi_{\infty}^{-i},u) = (1-u^3q^3)^{-1} \sum_{\substack{F \in \mathcal{M}_q \\ \deg(F) \equiv i \pmod 3 \\ (F,\pi)=1}} G_q(f,F) u^{\deg(F)}.$$
We will need the following result. 
\begin{lem} Let $\pi$ be a prime such that $\pi \nmid f$. 
We have the following relations
\begin{eqnarray}
\psi_\pi(f,\pi_\infty^{-i}, q^{-s})&=&\psi(f,\pi_\infty^{-i}, q^{-s}) -G_q(f,\pi) |\pi|_q^{-s} \psi_\pi(f\pi, \pi_\infty^{-i+\deg(\pi)},q^{-s}), \label{j=0}\\  
\psi_{\pi} (f \pi , \pi_{\infty}^{-i},q^{-s})&=& \psi(f \pi, \pi_{\infty}^{-i},q^{-s}) - \overline{G_q(f, \pi)} |\pi|_q^{1-2s}  \psi_{\pi}(f, \pi_{\infty}^{-i+2 \deg(\pi)},q^{-s}),\label{j=1}\\
\psi_\pi(f\pi^2, \pi_\infty^{-i}, q^{-s})&=&(1-|\pi|_q^{2-3s})^{-1}\psi(f\pi^2, \pi_\infty^{-i}, q^{-s}).\label{j=2}
\end{eqnarray}
\end{lem}
\begin{proof}
These equations appear in page 249 of \cite{patterson} as part of the ``Hecke theory'' equations. 
For completeness we give here the details of the proof of \eqref{j=1}.  The proofs of the other two identities proceed in a similar fashion. Consider
\begin{align*}
\psi_{\pi}(f \pi , \pi_{\infty}^{-i},q^{-s}) =& (1-q^{3(1-s)})^{-1}\sum_{\substack{F \in \mathcal{M}_q\\ \deg(F) \equiv i \pmod 3 \\ (F, \pi)=1}} \frac{ G_q(f \pi ,F)}{|F|_q^s}\\
=& \psi(f \pi , \pi_{\infty}^{-i},q^{-s}) - (1-q^{3(1-s)})^{-1}\sum_{\substack{F_1 \in \mathcal{M}_q\\\deg(F_1) \equiv i - \deg(\pi) \pmod 3}} \frac{G_q(f \pi , \pi F_1)}{|\pi|_q^s |F_1|_q^s}.
\end{align*}
 Note that in the second sum above, we need $\pi || F_1$, otherwise the Gauss sum will vanish by Lemma \ref{gauss}. 
 We write $F_1 = \pi F_2$ with $\pi \nmid F_2$. Part (i) of Lemma \ref{gauss} implies that 
 $G_q(f \pi , \pi^2 F_2) = G_q(f \pi^3, F_2) G_q(f \pi ,\pi^2)= G_q(f,F_2) G_q( f \pi, \pi^2).$ 
 Moreover, part (ii) of Lemma \ref{gauss} implies that  $G_q(f \pi, \pi^2) = |\pi|_q \overline{G_q(f, \pi)}$, where we have used that $\chi_\pi(-1)=1$ since it is a cubic character. 
Putting all of this together yields \eqref{j=1}.
\end{proof}

\begin{lem}
Let $\pi$ be a prime such that $\pi \nmid f$. 
We have the following relation
\begin{equation}\label{anotherHecke2}
\psi(f\pi^{j+3}, \pi_\infty^{-i}, q^{-s}) -|\pi|_q^{3-3s} \psi(f\pi^{j}, \pi_\infty^{-i}, q^{-s})
= (1-|\pi|_q^{2-3s})\psi_\pi(f\pi^{j}, \pi_\infty^{-i}, q^{-s}).
\end{equation}
\end{lem}
\begin{proof}
 We have 
 \begin{align}
\sum_{\deg(F) \equiv i \pmod 3} \frac{G_q(f \pi^{j}, F)}{|F|_q^s} =& \sum_{\substack{\deg(F) \equiv i \pmod 3 \\ (F, \pi)=1}} \frac{G_q(f \pi^{j},F)}{|F|_q^s}+ \sum_{\ell=1}^{\left [ \frac{j}{3}\right ]} |\pi|_q^{-3\ell s} \sum_{\substack{\deg(F) \equiv i \pmod 3 \\ (F, \pi)=1}} \frac{G_q(f \pi^{j}, \pi^{3\ell} F)}{|F|_q^s}  \label{eq-simpl}\\
&+ |\pi|_q^{-(j+1)s} \sum_{\substack{\deg(F) \equiv i-(j+1)\deg(\pi) \pmod 3 \\ (F, \pi)=1}} \frac{G_q(f \pi^{j}, \pi^{j+1} F)}{|F|_q^s}.\nonumber
\end{align}
Now when $(F,\pi)=1$, by  \eqref{prop-GS},
 it follows that $G_q(f \pi^{k},F) = G_q(f\pi^{[k]_3},F)$. We also have  using Lemma \ref{gauss} and   \eqref{prop-GS},
\begin{align*}
G_q(f \pi^{j},\pi^{3\ell} F)=& G_q(f \pi^{j+3\ell},F) G_q(f \pi^{j}, \pi^{3\ell})= G_q(f\pi^j,F) \phi( \pi^{3\ell}),\\
G_q(f \pi^j, \pi^{j+1} F) =& G_q(f \pi^{2j+1},F) G_q(f\pi^j, \pi^{j+1}) = G_q(f\pi^{2[j]_3+1} ,F) G_q(f\pi^{[j]_3},\pi^{[j]_3+1}) |\pi|_q^{j-[j]_3}\\
=& G_q(f\pi^{[j]_3}, \pi^{[j]_3+1} F) |\pi|_q^{j-[j]_3}.
\end{align*}
Using the relations above in \eqref{eq-simpl} and rearranging, we get that
\begin{align*}
\sum_{\deg(F) \equiv i \pmod 3} \frac{G_q(f \pi^j, F)}{|F|_q^s} =& \left( 1+ \sum_{\ell=1}^{\left [ \frac{j}{3}\right ]}  \frac{\phi(\pi^{3\ell})}{|\pi|_q^{3\ell s}}\right) \sum_{\substack{\deg(F) \equiv i \pmod 3 \\ (F,\pi)=1}} \frac{G_q(f\pi^j,F)}{|F|_q^s}\\&+ |\pi|_q^{(j-[j]_3)(1-s)} \sum_{\substack{ \deg(F) \equiv i \pmod 3 \\ \pi \mid F}} \frac{G_q(f\pi^{[j]_3}, F) }{|F|_q^s}.\\
\end{align*}
We now do the same with $j+3$ and take the difference. Then we have
\begin{equation*}\label{anotherHecke}
\sum_{\deg(F) \equiv i \pmod 3} \frac{G_q(f \pi^{j+3}, F)}{|F|_q^s} -|\pi|_q^{3-3s} \sum_{\deg(F) \equiv i \pmod 3} \frac{G_q(f \pi^j, F)}{|F|_q^s}= (1-|\pi|_q^{2-3s})\sum_{\substack{\deg(F) \equiv i \pmod 3 \\ (F,\pi)=1}} \frac{G_q(f\pi^j,F)}{|F|_q^s}.
\end{equation*}
Dividing by $(1-q^{3(1-s)})$, we obtain the result. 
\end{proof}

We will also use the following periodicity result, which is stated in \cite{patterson} and in \cite[p. 135]{Kazhdan-Patterson}.

\begin{lem}[The Periodicity Theorem] \label{periodicity-theorem} 
Let $\pi$ be a prime such that $\pi \nmid f$. Then 
\[\rho (f\pi^{j+3},i)=\rho(f\pi^j,i).\]
 \end{lem}
\kommentar{\begin{proof} 
Evaluating \eqref{anotherHecke2} as $s=2/3$ yields
\begin{equation} \label{funeq}
\psi(f\pi^{j+3},\pi_\infty^{-i},q^{-2/3})=|\pi|_q \psi(f\pi^{j},\pi_\infty^{-i},q^{-2/3})
\end{equation}

Multiplying both sides of the functional equation in Proposition \ref{prop-FE} by
$(1-q^{4-3s})$  and 
taking the limit as $s \to 4/3$, we obtain 
\begin{align*}
\rho(f,i)=&q^{2/3+4i/3-\deg(f)/3+[\deg(f)+1-2i]_3/3}(q^{-1}-1)\psi(f, \pi_\infty^{-i},q^{-2/3})\nonumber\\
&+W_{f,i}q^{2/3+4i/3-\deg(f)/3}(q-1)\psi(f, \pi_\infty^{-i-1-\deg(f)},q^{-2/3}).
\end{align*}
Applying the above to $f\pi^j$ and $f\pi^{j+3}$, and using equation \eqref{funeq}, we obtain the result.

\end{proof}
}
 
 We also need the following.
 \begin{lem} \label{lemma-res} Let $\pi$ be a prime such that $\pi \nmid f$. Then 
\[\lim_{s \rightarrow 4/3}q^{is}(1-q^{4-3s})\psi_\pi(f\pi^j,\pi_\infty^{-i}, q^{-s})=\frac{\rho(f\pi^j,i)}{1+|\pi|_q^{-1}}.\] 
\end{lem}

\begin{proof}
We multiply relation \eqref{anotherHecke2} by $q^{is}(1-q^{4-3s})/(1-q^{3(1-s)})$ and take the limit as $s \to \frac{4}{3}$. This yields
$$ \rho(f \pi^{j+3},i)-|\pi|_q^{-1} \rho(f\pi^j,i) = (1-|\pi|_q^{-2}) \lim_{s \to 4/3} q^{is} (1-q^{4-3s}) \psi_{\pi}(f\pi^j, \pi_{\infty}^{-i},q^{-s}).$$
Using Lemma \ref{periodicity-theorem} we obtain the result. 
\end{proof}

We now explicitly compute the residue $\rho(f,i)$.

\begin{lem} \label{lemma-residue}
Let $f=f_1f_2^2f_3^3$ with $f_1$, $f_2$ square-free and coprime. 
For $n=3$, we have that $\rho(f,i)=0$ if $f_2\not = 1$ and
\begin{equation}\label{estimate-residues}
\rho(f,i) = \overline{G_q(1,f_1)} |f_1|_q^{-2/3}  q^{\frac{4i}{3}-\frac{4}{3}\left[i-2 \deg(f) \right]_3}\rho\left(1,\left[i- 2\deg(f) \right]_3\right),
\end{equation}
when $f_2=1$. Here
\[\rho(1,0)=1, \qquad \rho(1,1)=\tau(\chi_3)q, \qquad \rho(1,2)=0.\]
\end{lem}
\begin{proof}
We start by computing $\rho(1,[i]_3)$. Recall by definition that 
\begin{align*}
 G_q(1,F)=&\sum_{v \pmod{F}}\chi_F(v)e_q\left(\frac{v}{F}\right)\\
 =& \sum_{\deg(v) \leq \deg(F)-2} \chi_F(v) + \sum_{\deg(v) = \deg(F)-1} \chi_F(v)e_q\left(\frac{v}{F}\right)\\
 =&\sum_{c\in \F_q^*}\sum_{v \in \mathcal{M}_{q,\leq \deg(F)-2}} \chi_F(c) \chi_F(v) + \sum_{c\in \F_q^*} \sum_{v \in \mathcal{M}_{q,\deg(F)-1}} 
\chi_F(c) \chi_F(v)e^{2 \pi i \text{tr}_{\F_q / \mathbb{F}_p} (c)/p}\\
=&\sum_{c\in \F_q^*} \chi_F(c) \sum_{v \in \mathcal{M}_{q,\leq \deg(F)-2}} \chi_F(v) + \tau(\chi_F) \sum_{v \in \mathcal{M}_{q,\deg(F)-1}}\chi_F(v).\\
\end{align*}

First suppose that $[i]_3=0$, i.e., $\deg(F)\equiv 0 \pmod{3}$. Then, $\chi_F$ is even and
\begin{align*}
\sum_{v \in \mathcal{M}_{q,\leq \deg(F)-2}} \chi_F(v) + \sum_{v \in \mathcal{M}_{q,\deg(F)-1}}\chi_F(v)
=&\left \{\begin{array}{cc}0 & F \not = \cube, \\ \frac{\phi(F)}{q-1} & F =\cube. \end{array} \right .
\end{align*}
Then we write
\[G_q(1,F)=\phi(F)\delta_{\tinycube}(F)+ \left(\tau(\chi_F)-\sum_{c\in \F_q^*} \chi_F(c)\right) \sum_{v \in \mathcal{M}_{q,\deg(F)-1}}\chi_F(v),\]
where the term $\delta_{\tinycube}(F)=1$ is 1 if $F=\cube$ and 0 otherwise. 

Since $[i]_3=0$, $\tau(\chi_F)=-1$ and $\sum_{c\in \F_q^*} \chi_F(c)=q-1$, and we have
\begin{align*}
\psi(1,\pi_\infty^0,u)=&(1-u^3q^3)^{-1}\sum_{\substack{F \in \mathcal{M}_q\\3\mid \deg(F)}} 
G_q(1,F) u^{\deg(F)}\\
=& (1-u^3q^3)^{-1}\sum_{\substack{F \in \mathcal{M}_q\\\deg(F)=\tinycube}} 
\phi(F) u^{\deg(F)}-q (1-u^3q^3)^{-1}\sum_{\substack{F \in \mathcal{M}_q\\3 \mid \deg(F)}} 
\sum_{v \in \mathcal{M}_{q,\deg(F)-1}}\chi_F(v) u^{\deg(F)}.\\
\end{align*}
Notice that
\begin{equation}\label{eq:cubicsum}
\sum_{F \in \mathcal{M}_{q,k}} \sum_{v \in \mathcal{M}_{q,k-1}}\chi_F(v)=\sum_{v \in \mathcal{M}_{q,k-1}}
\sum_{F \in \mathcal{M}_{q,k}} \chi_v(F)=q^k \# \{v \in \mathcal{M}_{q,k-1} \, |\, v = \cube\},
\end{equation}
and this gives zero when $k \not \equiv 1 \pmod 3$. 

This gives 
\begin{align*}
\psi(1,\pi_\infty^0,u)=&(1-u^3q^3)^{-1}\sum_{F \in \mathcal{M}_q} 
\phi(F^3) u^{3\deg(F)}\\
=& (1-u^3q^3)^{-1}\sum_{k=0}^\infty \sum_{F \in \mathcal{M}_{q,k}} 
\phi(F)|F|_q^2 u^{3\deg(F)}\\
=& (1-u^3q^3)^{-1}\sum_{k=0}^\infty q^{2k}u^{3k}\sum_{F \in \mathcal{M}_{q,k}} 
\phi(F)\\
=& (1-u^3q^3)^{-1}\sum_{k=0}^\infty q^{4k}u^{3k}(1-q^{-1}),
\end{align*}
where we have used Proposition 2.7 in \cite{Rosen}. Finally, we get
\[\psi(1,\pi_\infty^0,u)=\frac{1-q^{-1}}{(1-u^3q^3)(1-u^3q^4)},\]
and taking the residue,
\[\rho(1,0)=1.\]

When $[i]_3\not=0$, $\sum_{c\in \F_q^*} \chi_F(c)=0$ and we obtain,
\begin{align*}
\psi(1,\pi_\infty^{-i},u)=&(1-u^3q^3)^{-1}\sum_{\substack{F \in \mathcal{M}_q\\\deg(F)\equiv i \pmod 3}} 
G_q(1,F) u^{\deg(F)}\\
=& \frac{\tau(\chi_3^i)}{1-u^3q^3}\sum_{\substack{F \in \mathcal{M}_q\\\deg(F)\equiv i \pmod 3}} 
 \sum_{v \in \mathcal{M}_{q,\deg(F)-1}}\chi_F(v) u^{\deg(F)}\\
\end{align*}
When $[i]_3=2$, from equation \eqref{eq:cubicsum}, we immediately get that the sum above is zero and
\[\rho(1,2)=0.\]
On the other hand, if $[i]_3=1$ we have, by cubic reciprocity,
\begin{align*}
\psi(1,\pi_\infty^{-1},u)=&\frac{\tau(\chi_3)}{1-u^3q^3}\sum_{j=0}^\infty  u^{3j+1} \sum_{v \in \mathcal{M}_{q,3j}} \sum_{F \in \mathcal{M}_{q,1+3j}} \chi_v(F) \\
=&\frac{\tau(\chi_3)}{(1-u^3q^3)} \sum_{j=0}^\infty  u^{3j+1} \sum_{w \in \mathcal{M}_{q,j}} \frac{\phi(w^3)}{|w^3|_q} q^{3j+1} \\
=&\frac{\tau(\chi_3)}{(1-u^3q^3)} \sum_{j=0}^\infty  u^{3j+1} q^j (1- q^{-1}) q^{3j+1} \\
=&\frac{\tau(\chi_3) (q-1) u}{(1-u^3q^3)}  \sum_{j=0}^\infty (q^4 u^3)^j = \frac{\tau(\chi_3) (q-1) u}{(1-u^3q^3) (1 - u^3 q^4)},
\end{align*}
where we have used again Proposition 2.7 in \cite{Rosen}. Taking the residue, we get
\[\rho(1,1)= \lim_{s \rightarrow 4/3} 
\frac{\tau(\chi_3) (q-1) }{(1-u^3q^3)} =
\tau(\chi_3)q .\]

To obtain equation \eqref{estimate-residues}, we start by multiplying equation \eqref{j=1} by $q^{is}(1-q^{4-3s})$ and taking the limit as
$s \to 4/3$. By Lemma \ref{lemma-res} for $\pi \nmid f$ we get that 
\begin{align*}
\rho(f \pi, i)\left[1-\frac{1}{1+|\pi|_q^{-1}}\right]= \overline{G_q(f, \pi)} |\pi|_q^{-\frac{5}{3}} q^{\frac{8}{3}\deg(\pi)}  \frac{\rho(f, i-2 \deg(\pi))}{1+|\pi|_q^{-1}},
\end{align*}
which simplifies to
\begin{equation}\label{Patterson-7}
\rho(f \pi , i) =  \overline{G_q(f, \pi)}|\pi|_q^{-\frac{2}{3}}  q^{\frac{8}{3}\deg(\pi)} \rho(f, i-2 \deg(\pi)).
\end{equation}

Multiplying equation \eqref{j=2} by $q^{is}(1-q^{4-3s})$, taking the limit as $s\rightarrow 
4/3$, and applying Lemma \ref{lemma-res} we get that
\[\rho(f\pi^2,i)= \rho(f\pi^2,i) \left[\frac{1-|\pi|_q^{-2}}{1+|\pi|_q^{-1}}\right],\]
which implies that
\begin{equation}\label{Patterson-8}
 \rho(f\pi^2,i)=0.
\end{equation}

Notice that by the Periodicity Theorem (Lemma \ref{periodicity-theorem}),
$\rho(f,i)$ depends on the cubic-free part of $f$. 
From this and equation \eqref{Patterson-8} we can suppose that $f=f_1$ with $f_1$ square-free. 
Write $f=\pi_1\cdots\pi_k$. By \eqref{Patterson-7}, we have
\begin{align*}
\rho(f, i)= & \overline{G_q(f/\pi_k,\pi_k)}|\pi_k|_q^{-\frac{2}{3}}
q^{\frac{8}{3}\deg(\pi_k)} \rho(f/\pi_k,i-2\deg(\pi_k))\\
=& \overline{G_q(f/\pi_k, \pi_k)}|\pi_k|_q^{-\frac{2}{3}}   q^{\frac{4i}{3}-\frac{4}{3}[i-2\deg(\pi_k)]_3}\rho(f/\pi_k, [i-2 \deg(\pi_k)]_3)\\
=& \cdots \\
=& \prod_{j=1}^k \overline{G_q\left (\prod_{\ell=1}^{j-1}\pi_\ell,\pi_j\right )}|f|_q^{-\frac{2}{3}} q^{\frac{4i}{3}-\frac{4}{3}\left[i-2\sum_{j=1}^k\deg(\pi_j)\right]_3}\rho\left (1,\left[i-2\sum_{j=1}^k\deg(\pi_j)\right]_3 \right )
\end{align*}
In the equation above, note that
$$ \prod_{j=1}^k \overline{G_q\left(\prod_{\ell=1}^{j-1}\pi_\ell,\pi_j\right)} = \overline{G_q(1, f)},$$ which follows by induction on the number of prime divisors of $f$ and part (i) of Lemma \ref{gauss}.
This finishes the proof of Lemma \ref{lemma-residue}.
\end{proof}
\subsection{Upper bounds for $\Psi_q(f,u)$ and $\tilde{\Psi}_q(f,u)$}
We will first prove the following result which provides an upper bound for $\Psi_q(f,u)$. 
\begin{thm} \label{first-generating-series}
For $1/2 \leq \sigma \leq 3/2$ and $|u^3-q^{-4}|>\delta$ where $\delta>0$, we have that
$$\Psi_q(f,u)  \ll_\delta |f|_q^{\frac{1}{2} \left( \frac32 - \sigma \right)+\varepsilon},$$
where $u=q^{-s}$ as usual, and $\sigma=\re(s)$. 
\end{thm}
\begin{proof} 
The bound for $\Psi_q(f,q^{-s})$ for  $1 /2< \re(s)\leq 3/2$ and $|u^3-q^{-4}| > \delta$ follows from the  functional equation and the Phragm\'en--Lindel\"of principle. 
It suffices to show that the bound holds for $\psi(f, \pi_\infty^{-i}, q^{-s})$ for $i=0,1,2$ by \eqref{gen-series}.

First, it follows from \eqref{meromorphic} and \eqref{formula-for-P} that for $B=[(1+\deg(f)-i)/3]$ we have
\begin{align*}
\psi(f, \pi_\infty^{-i},u)
=&\frac{u^i P(f,i,u^3)}{1-q^{4}u^3}\\
=&\frac{1}{1-q^3u^3}  \sum_{0\leq j < B}C(f,i+3j)u^{i+3j}+\frac{C(f,i+3B)u^{i+3B}}{(1-q^{4}u^3)(1-q^3u^3)} .
\end{align*}

We now bound $|C(f,k)|$.  Write $F=F_1F_2$ with $(F_1,f)=1$ and $F_2\mid f^\infty$ (by this we mean that the primes of $F_2$ divide $f$.) We use repeatedly that $|G_q(f,F_1F_2)|=|G_q(f,F_1)||G_q(f,F_2)|$. By Lemma \ref{gauss} we have
for $F_2\mid f^\infty$ that $|G_q(f,F_2)|=0$ unless $F_2\mid f^2$. 
We write
\begin{align*}
\sum_{F \in \mathcal{M}_{q,k}} |G_q(f,F)|=&\sum_{j=0}^k \sum_{\substack{F_1 \in \mathcal{M}_{q,j}\\(F_1,f)=1}} 
 |G_q(f,F_1)|\sum_{\substack{F_2 \in \mathcal{M}_{q,k-j}\\F_2\mid f^2}} |G_q(f,F_2)|\\
 \leq &\sum_{j=0}^k \sum_{\substack{F_1 \in \mathcal{M}_{q,j}\\(F_1,f)=1}} q^{j/2}\sum_{\substack{F_2 \in \mathcal{M}_{q,k-j}\\F_2\mid f^2}} q^{k-j}\\
\ll &\sum_{j=0}^k q^{3j/2}q^{k-j}|f|_q^{\varepsilon}\\
 \ll& q^{3k/2}|f|_q^\varepsilon.
 \end{align*}
Thus
\[|C(f,k)| \ll q^{3k/2}|f|_q^\varepsilon.\]

We get that for $\sigma \leq 3/2$ 
\begin{align*}
|\psi(f, \pi_\infty^{-i},q^{-s})|\ll & 
 {\sum_{k=0}^{3B+2}} q^{(3/2-\sigma)k} |f|_q^\varepsilon \ll |f|_q^{3/2- \sigma+\varepsilon},
\end{align*}
with an absolute constant in that region.
In particular, 
\begin{equation} \label{bound-f}
\psi(f, \pi_\infty^{-i},q^{-s}) \ll  |f|_q^{\varepsilon}
\end{equation}
 when $\re(s)=3/2$.

From the functional equation of Remark \ref{remark-fe}, we have for $1/2 \leq \re(s) \leq 3/2$ 
and $|u^3-q^{-4}| > \delta$ that
\begin{equation} \label{our-FE}  
\psi(f, \pi_\infty^{-i}, q^{-s}) = a_1(s) \; |f|_q^{1-s} \psi(f, \pi_\infty^{-i}, q^{s-2}) +
a_2(s) \; |f|_q^{1-s}  \psi(f, \pi_\infty^{i-1-\deg(f)}, q^{s-2}) , \end{equation}
where $a_1(s)$ and $a_2(s)$ are absolutely bounded above and below in the region considered
 (independently of $f$).

Using the bound \eqref{bound-f} and the functional equation gives that
$$\psi(f, \pi_\infty^{-i},q^{-s}) \ll |f|_q^{1/2+\varepsilon}$$ when $\re(s)=1/2$.

We consider the function $\Phi(f,\pi_\infty^{-i},s) = (1-q^{4-3s}) (1-q^{3s-2})\psi(f, \pi_\infty^{-i},q^{-s}) \psi(f, \pi_\infty^{-i},q^{s-2}).$ 
Then $\Phi(f,\pi_\infty^{-i},s)$ is holomorphic in the region $1/2 \leq \re(s) \leq 3/2$, and
$\Phi(f,\pi_\infty^{-i},s) \ll |f|_q^{1/2+\varepsilon}$ for $\re(s) = 3/2$ and $\re(s) = 1/2$. 

Using the Phragm\'en--Lindel\"of principle, 
 it follows that for  $1/2 \leq \re(s) \leq 3/2$, we  have that 
\begin{equation*}
 \Phi(f,\pi_\infty^{-i},s) = (1-q^{4-3s}) (1-q^{3s-2})\psi(f, \pi_\infty^{-i}, q^{-s})  \psi(f, \pi_\infty^{-i}, q^{s-2})  \ll |f|_q^{1/2+\varepsilon}.
 \end{equation*}

Using the functional equation \eqref{our-FE}, this gives 
\begin{equation}
 (1-q^{4-3s}) (1-q^{3s-2})\left[ a_1(s) \psi(f, \pi_\infty^{-i}, q^{s-2})^2  + a_2(s)  \psi(f, \pi_\infty^{-i}, q^{s-2})  \psi(f, \pi_\infty^{i-1-\deg(f)}, q^{s-2})\right]
\ll |f|_q^{\sigma-\frac{1}{2} + \varepsilon} \label{PL-1}
\end{equation}
in the region  $1/2 \leq \re(s) \leq 3/2$ 
and $|u^3-q^{-4}| > \delta$.

If $\deg(f)+1\equiv 2i \pmod{3}$, then the formula above implies that 
\begin{equation}\label{good bound}
\psi(f, \pi_{\infty}^{-i},q^{s-2})  + \psi(f, \pi_{\infty}^{i-1-\deg(f)},q^{s-2}) \ll |f|_q^{\frac{1}{2}\left(\sigma- \frac{1}{2}+\varepsilon\right)}.
\end{equation}

Now assume that $\deg(f)+1 \not \equiv 2i \pmod{3}$.
Similarly we consider the function $\tilde{\Phi}(s)=(1-q^{4-3s})(1-q^{3s-2})\psi(f, \pi_{\infty}^{-i},q^{-s}) \psi(f, \pi_{\infty}^{i-1-\deg(f)},q^{s-2})$. Then, using the same arguments as above we get that
\begin{align} 
(1-q^{4-3s}) (1-q^{3s-2}) & \left[  a_1(s) \psi(f, \pi_{\infty}^{-i},q^{s-2}) \psi(f, \pi_{\infty}^{i-1-\deg(f)},q^{s-2})+a_2(s) \psi(f, \pi_{\infty}^{i-1-\deg(f)},q^{s-2})^2 \right] \nonumber \\
& \ll |f|_q^{\sigma - \frac{1}{2}+\varepsilon}.\label{PL-2}
 \end{align}

Combining the two equations \eqref{PL-1} and \eqref{PL-2}, it would follow that
\begin{align}
&(1-q^{4-3s})(1-q^{3s-2}) \left[ \psi(f, \pi_{\infty}^{-i},q^{s-2}) + \psi(f, \pi_{\infty}^{i-1-\deg(f)},q^{s-2}) \right]\nonumber \\
&\times \left[a_1(s) \psi(f, \pi_{\infty}^{-i},q^{s-2}) + a_2(s) \psi(f, \pi_{\infty}^{i-1-\deg(f)},q^{s-2}) \right] 
\ll|f|_q^{\sigma- \frac{1}{2}+\varepsilon}. \label{combined-upper-bound}
\end{align}
Switching $i$ with $\deg(f)+1-i$ (since $\deg(f)+1 \not \equiv 2i \pmod 3$), we get that there exist absolutely bounded constants $b_1(s)$ and $b_2(s)$ such that
$$ \psi(f, \pi_{\infty}^{i-1-\deg(f)},q^{-s}) = b_1(s) |f|_q^{1-s} \psi(f, \pi_{\infty}^{i-1-\deg(f)},q^{s-2})+b_2(s) |f|_q^{1-s} \psi(f, \pi_{\infty}^{-i},q^{s-2}).$$
 If $(a_1,a_2)$ and $(b_2,b_1)$ are not linearly independent, then from the equation above and \eqref{our-FE} it follows that
 $$ \psi(f, \pi_{\infty}^{i-1-\deg(f)},q^{-s}) = \lambda(s) \psi(f, \pi_{\infty}^{-i},q^{-s}),$$ for some $\lambda(s)$.
 Combining this with equation \eqref{combined-upper-bound}, we get that
 $$ \psi(f, \pi_{\infty}^{-i},q^{s-2}) \ll |f|_q^{\frac{1}{2} \left( \sigma- \frac{1}{2}+\varepsilon\right)},$$ and the conclusion again follows by replacing $2-s$ by $s$.
 
 If $(a_1,a_2)$ and $(b_2,b_1)$ are linearly independent, then 
\begin{align*}
&(1-q^{4-3s})(1-q^{3s-2}) \left[ \psi(f, \pi_{\infty}^{i-1-\deg(f)},q^{s-2})  + \psi(f, \pi_{\infty}^{-i},q^{s-2}) \right] \\
&\times \left[b_1(s) \psi(f, \pi_{\infty}^{i-1-\deg(f)},q^{s-2}) + b_2(s) \psi(f, \pi_{\infty}^{-i},q^{s-2}) \right] \ll |f|_q^{\sigma- \frac{1}{2}+\varepsilon}.
\end{align*}
 From the equation above and \eqref{combined-upper-bound}, by the linear independence condition, we get that 
$$ \left[ \psi(f, \pi_{\infty}^{-i},q^{s-2})  + \psi(f, \pi_{\infty}^{i-1-\deg(f)},q^{s-2}) \right]  \psi(f, \pi_{\infty}^{-i},q^{s-2}) \ll  |f|_q^{\sigma- \frac{1}{2}+\varepsilon},$$ and
$$ \left[ \psi(f, \pi_{\infty}^{-i},q^{s-2})  + \psi(f, \pi_{\infty}^{i-1-\deg(f)},q^{s-2}) \right]  \psi(f, \pi_{\infty}^{i-1-\deg(f)},q^{s-2}) \ll  |f|_q^{\sigma- \frac{1}{2}+\varepsilon}.$$
By summing the two equations above, we recover equation \eqref{good bound} without any restrictions on $i$,
$$ \psi(f, \pi_{\infty}^{-i},q^{s-2})  + \psi(f, \pi_{\infty}^{i-1-\deg(f)},q^{s-2}) \ll |f|_q^{\frac{1}{2}\left(\sigma- \frac{1}{2}+\varepsilon\right)}.$$ Summing over $i=0,1,2$ and replacing $2-s$ by $s$ finishes the proof.

\end{proof}

In order to obtain an upper bound for $\Tilde{\Psi}_q(f,u)$ (recall its definition \eqref{tilde-F}) we first need to relate it to $\Psi_q(f,u)$ which we do in the next lemma.

\begin{lem} \label{our-psi}
Let $f = f_1 f_2^2 f_3^3$ with $f_1, f_2$ square-free and co-prime, and let $f_3^*$ be the product of the primes dividing $f_3$ but not dividing $f_1 f_2$.
Then, 
\begin{eqnarray}
\Tilde{\Psi}_q(f,u)  &=& \prod_{P \mid f_1f_2} \left( 1 - (u^3 q^2)^{\deg(P)} \right) ^{-1}\sum_{a \mid f_3^*} \mu(a) G_q(f_1 f_2^2, a) u^{\deg(a)} 
 \prod_{P\mid a}(1-(u^3q^2)^{\deg(P)})^{-1} \nonumber \\
 && \times \sum_{\ell \mid af_1} \mu(\ell) (u^2q)^{\deg(\ell)} \overline{G_q(1,\ell)} \chi_{\ell}(af_1f_2^2/\ell)   \Psi_q(af_1f_2^2/\ell,u) \label{expression}.
\end{eqnarray}

If $1/2 \leq \sigma \leq 3/2$ and $|u^3-q^{-4}|, |u^3-q^{-2}|>\delta$, then
$$
\Tilde{\Psi}_q(f,u) \ll_\delta |f|_q^{{\frac{1}{2}}(\frac{3}{2}-\sigma)+\varepsilon}.
$$
\end{lem}

\begin{proof} We first show that the last assertion follows from the expression \eqref{expression} for $\Tilde{\Psi}_q(f,u)$.

Suppose that $1/2 \leq \sigma \leq 3/2$ and $|u^3-q^{-4}|, |u^3-q^{-2}|>\delta$. Then, for $\re (s)=\sigma$, 
\begin{eqnarray*}
\tilde{\Psi}_q(f,u)  &\ll& \sum_{a \mid f_3^*}|a|_q^{1/2 - \sigma}
 \sum_{\ell \mid af_1}|\ell|_q^{3/2-2 \sigma}  \left| \frac{af_1 f_2^2}{\ell} \right|_q^{\frac{1}{2}\left(\frac{3}{2}-\sigma\right)+\varepsilon} \\
&\ll& \sum_{a \mid f_3^*}|a|_q^{\frac{5-6 \sigma}{4}+\varepsilon} \sum_{\ell \mid af_1}|\ell|_q^{\frac{3-6\sigma}{4}}  \left| {f_1 f_2^2} \right|_q^{\frac{1}{2}\left(\frac{3}{2}-\sigma\right)+\varepsilon} \\
&\ll& \sum_{a \mid f_3^*}|a|_q^{\frac{5-6 \sigma}{4}+\varepsilon}  \left| {f_1 f_2^2} \right|_q^{\frac{1}{2}\left(\frac{3}{2}-\sigma\right)+\varepsilon} \\
&\ll&  \max\left\{|f_3^{*}|_q^{\varepsilon}, |f_3|_q^{\frac{5-6 \sigma}{4}+\varepsilon}\right\} \left| {f_1 f_2^2} \right|_q^{\frac{1}{2}\left(\frac{3}{2}-\sigma\right)+\varepsilon} \ll |f|_q^{\frac{1}{2}\left(\frac{3}{2}-\sigma\right)+\varepsilon}.
\end{eqnarray*}

We now prove \eqref{expression}. We first remark that by definition of $f_1, f_2, f_3^*$, we have that $(f, F) = 1 \iff (f_1 f_2, F) = 1 \;\mbox{and}\; (f_3^*, F)=1$ with $(f_1 f_2, f_3^*) = 1$.
If $(f_1 f_2 f_3, F)=1$, then $G_q(f_1 f_2^2 f_3^3, F) = \overline{\chi_F}(f_3^3) G_q(f_1 f_2^2, F) =  G_q(f_1 f_2^2, F)$, and 
\begin{eqnarray*}
\tilde{\Psi}_q(f,u) &=& \sum_{(F,f_1 f_2 f_3^*)=1} G_q(f_1 f_2^2, F) u^{\deg(F)}  \\
&= & \sum_{a \mid f_3^*} \mu(a) u^{\deg(a)} \sum_{(F, f_1 f_2)=1} G_q(f_1 f_2^2, a F) u^{\deg(F)}.
\end{eqnarray*}
If $(a, F) \neq 1$, then there is a prime $P$ such that $P^2 \mid aF$ and $P \nmid f_1 f_2^2$, and then
$G_q(f_1 f_2^2, aF) = 0$.
We can then suppose that $(a,F)=1$, and then by Lemma \ref{gauss} (i), we have that
$G_q(f_1 f_2^2, a F) =  G_q(f_1 f_2^2, a) G_q(a f_1 f_2^2, F),$ and
\begin{eqnarray} \label{first-expr}
\tilde{\Psi}_q(f,u) &= & \sum_{a \mid f_3^*} \mu(a) G_q(f_1 f_2^2, a) u^{\deg(a)} \sum_{(F, a f_1 f_2)=1} G_q(a f_1 f_2^2, F) u^{\deg(F)}.
\end{eqnarray}
Notice that $a f_1 f_2$ is square-free and that $a$, $f_1$ and $f_2$ are two-by-two co-prime.

Let $P$ be a prime dividing $f_2$, and we write $f_2 = P f_2'$, and $F = P^i F'$ with $(F' f_2', P)=1$.
Then, by Lemma \ref{gauss},
\begin{eqnarray*}
G_q( af_1 f_2'^2 P^2, P^i F') = G_q(af_1 f_2'^2 P^2, P^i) G_q(af_1 f_2'^2 P^{2+i}, F') =
\begin{cases} 
G_q(af_1 f_2'^2 P^{2}, F') & i=0, \\
- |P|_q^2 G_q(af_1 f_2'^2 P^{2}, F') & i=3, \\
0 & \mbox{otherwise.} 
\end{cases}
\end{eqnarray*}
We remark that we have used that  $G_q(af_1 f_2'^2 P^{5}, F') =  G_q(af_1 f_2'^2 P^{2}, F')$ for the second line, since $(P, F')=1$.
This gives
\begin{eqnarray*}
&& \sum_{(F,a f_1 f_2)=1} G_q(a f_1 f_2^2, F) u^{\deg(F)} \\ &=&  \sum_{(F,a f_1 f_2')=1} G_q(a f_1 f_2^2, F) u^{\deg(F)} - \sum_{\substack{(F,a f_1 f_2' )=1\\ P \mid F}} G_q(a f_1 f_2'^2 P^2, F) u^{\deg(F)}\\
&=& \sum_{(F,af_1 f_2' )=1} G_q(a f_1 f_2^2, F) u^{\deg(F)} + \sum_{(F',a f_1 f_2)=1}  G_q(a f_1 f_2'^2 P^2, F') u^{\deg(F')+3 \deg(P)} q^{2 \deg(P)}\\
&=& \sum_{(F,af_1 f_2')=1} G_q(a f_1 f_2^2, F) u^{\deg(F)} + (u^3 q^2)^{\deg(P)} \sum_{(F',a f_1 f_2)=1}  G_q(a f_1 f_2^2, F') u^{\deg(F')} ,
\end{eqnarray*}
or equivalently
\begin{eqnarray*}
\left( 1 - (u^3 q^2)^{\deg(P)} \right) \sum_{(F, a f_1 f_2)=1} G_q(a f_1 f_2^2, F) u^{\deg(F)} &=&  \sum_{(F,a f_1 f_2')=1} G_q(a f_1 f_2^2, F) u^{\deg(F)}.
\end{eqnarray*}
By induction on the prime divisors of $f_2$, we get
\begin{eqnarray*}
\sum_{(F,a f_1 f_2)=1} G_q(a f_1 f_2^2, F) u^{\deg(F)} =  \prod_{P \mid f_2} \left( 1 - (u^3 q^2)^{\deg(P)} \right) ^{-1} \sum_{(F,a f_1)=1} G_q(a f_1 f_2^2, F) u^{\deg(F)},
\end{eqnarray*}
and plugging in \eqref{first-expr}, we have
\begin{equation}
\tilde{\Psi}_q(f,u) =   \prod_{P \mid f_2} \left( 1 - (u^3 q^2)^{\deg(P)} \right) ^{-1} \sum_{a \mid f_3^*} \mu(a) G_q(f_1 f_2^2, a) u^{\deg(a)} 
 \sum_{(F,a f_1)=1} G_q(a f_1 f_2^2, F) u^{\deg(F)}.\label{psitilde}
\end{equation}

We now do the same thing for $ \sum_{(F,a f_1)=1} G_q(a f_1 f_2^2, F) u^{\deg(F)},$ dealing with the primes dividing $f_1^*:=a f_1$ one by one.

Let $f_1^* = P f_1'$, and we write
\begin{eqnarray*}
\sum_{(F,f_1^*)=1} G_q(f_1^* f_2^2, F) u^{\deg(F)} = \sum_{(F,f_1')=1} G_q(f_1^* f_2^2, F) u^{\deg(F)}
 - \sum_{\substack{(F,f_1')=1\\ F = P^i F', i \geq 1}} G_q(f_1' P f_2^2, P^i F') u^{\deg(F')+i \deg(P)}.
\end{eqnarray*}
Using Lemma \ref{gauss}, we compute that
\begin{eqnarray*}
G_q(f_1' P f_2^2, P^i F') &=& G_q(f_1' P f_2^2, P^i) G_q(f_1' P^{i+1} f_2^2, F')\\ &=&
\begin{cases} G_q(f_1' P f_2^2, F') & i = 0, \\
G_q(f_1' P^{3} f_2^2, F') \epsilon(\chi_{P^2}) \omega(\chi_{P^2}) \chi_{P}(f_1' f_2^2)  |P|_q^{3/2}  & i=2,\\
0 & \mbox{otherwise,}\end{cases}
\end{eqnarray*}
where we recall that $\epsilon(\chi_{P^2})=1$ when $3 \mid \deg(P)$.  

Then,
\begin{align}\label{Chantal-equation}
&\sum_{(F,f_1^*)=1} G_q(f_1^* f_2^2, F) u^{\deg(F)} \nonumber \\  = &\sum_{(F,f_1')=1} G_q(f_1^* f_2^2, F) u^{\deg(F)}
 - \sum_{(F',f_1^*)=1} G_q(f_1' f_2^2, F') u^{\deg(F')+ 2\deg(P)} q^{3 \deg(P)/2}\epsilon(\chi_{P^2})\omega(\chi_{P^2}) \chi_P(f'_1 f_2^2) \nonumber \\
  = & \sum_{(F,f_1')=1} G_q(f_1^* f_2^2, F) u^{\deg(F)}
 - (u^2 q^{3/2})^{\deg(P)}\epsilon(\chi_{P^2})\omega(\chi_{P^2})\chi_P(f'_1 f_2^2)  \sum_{(F',f_1^*)=1} G_q\left(\frac{f_1^* f_2^2}{P}, F'\right) u^{\deg(F')}.
\end{align}




Now we focus on 
\begin{align*}
\sum_{(F,f_1'P)=1} G_q(f_1' f_2^2, F) u^{\deg(F)}=&\sum_{(F,f_1')=1} G_q(f_1' f_2^2, F) u^{\deg(F)}-\sum_{\substack{(F,f_1')=1\\P\mid F}} G_q(f_1' f_2^2, F) u^{\deg(F)}.
 \end{align*}
As before, write $F=P^iF'$. By Lemma \ref{gauss} as always,
\begin{eqnarray*}
G_q(f_1' f_2^2, P^i F') &=& G_q(f_1' f_2^2, P^i) G_q(f_1' f_2^2 P^{i}, F')\\ &=&
\begin{cases}
                                                                G_q(f_1' f_2^2, F') & i=0,\\
G_q(f_1' f_2^2 P, F') \epsilon(\chi_{P}) \omega(\chi_{P}) \chi_{P^2}(f_1' f_2^2)  |P|_q^{1/2}  & i=1,\\

                                                                0& i\geq 2,\\
                                                               \end{cases}
\end{eqnarray*}
and we get
\begin{align*}
&\sum_{(F,f_1'P)=1} G_q(f_1' f_2^2, F) u^{\deg(F)}\\
=&\sum_{(F,f_1')=1} G_q(f_1' f_2^2, F) u^{\deg(F)}-\sum_{(F',Pf_1')=1} G_q(f_1' f_2^2, F'P) u^{\deg(F')+\deg(P)}\\ 
=&\sum_{(F,f_1')=1} G_q(f_1' f_2^2, F) u^{\deg(F)}-(uq^{1/2})^{\deg(P)}\epsilon(\chi_{P}) \omega(\chi_{P}) \chi_{P^2}(f_1' f_2^2)\sum_{(F',Pf_1')=1} G_q(f_1' f_2^2 P, F')  u^{\deg(F')}\\ 
=&\sum_{(F,f_1')=1} G_q(f_1' f_2^2, F) u^{\deg(F)}-(uq^{1/2})^{\deg(P)}\epsilon(\chi_{P}) \omega(\chi_{P}) \chi_{P^2}(f_1' f_2^2)\sum_{(F,f_1^*)=1} G_q(f_1^* f_2^2, F)  u^{\deg(F)}.\\ 
\end{align*}

Now we incorporate the equation above into equation \eqref{Chantal-equation}.
\begin{align*}
&\sum_{(F,f_1^*)=1} G_q(f_1^* f_2^2, F) u^{\deg(F)}\\ 
& = \sum_{(F,f_1')=1} G_q(f_1^* f_2^2, F) u^{\deg(F)}
 - (u^2 q^{3/2})^{\deg(P)} \epsilon(\chi_{P^2})\omega(\chi_{P^2}) \chi_P(f'_1 f_2^2)\sum_{(F',f_1^*)=1} G_q(f_1' f_2^2, F') u^{\deg(F')} \\
=&\sum_{(F,f_1')=1} G_q(f_1^* f_2^2, F) u^{\deg(F)}
 - (u^2 q^{3/2})^{\deg(P)}\epsilon(\chi_{P^2})\omega(\chi_{P^2})\chi_P(f'_1 f_2^2) \sum_{(F,f_1')=1} G_q(f_1' f_2^2, F) u^{\deg(F)}\\
& +(u^3 q^{2})^{\deg(P)}\sum_{(F,f_1^*)=1} G_q(f_1^* f_2^2, F)  u^{\deg(F)}.
\end{align*}

Rearranging, we write
\begin{align*}
&(1-(u^3 q^{2})^{\deg(P)})\sum_{(F,f_1^*)=1} G_q(f_1^* f_2^2, F) u^{\deg(F)}\\
=& \sum_{(F,f_1')=1} G_q(f_1^* f_2^2, F) u^{\deg(F)}
 - (u^2 q^{3/2})^{\deg(P)}\chi_P(f'_1 f_2^2) \epsilon(\chi_{P^2})\omega(\chi_{P^2})\sum_{(F,f_1')=1} G_q(f_1' f_2^2, F) u^{\deg(F)}\\
\end{align*}
or
\begin{align*}
&\sum_{(F,f_1^*)=1} G_q(f_1^* f_2^2, F) u^{\deg(F)}=(1-(u^3 q^{2})^{\deg(P)})^{-1} \sum_{(F,f_1')=1} G_q(f_1^* f_2^2, F) u^{\deg(F)}\\
& - (1-(u^3 q^{2})^{\deg(P)})^{-1}(u^2 q^{3/2})^{\deg(P)}\chi_P(f'_1 f_2^2) \epsilon(\chi_{P^2})\omega(\chi_{P^2})\sum_{(F,f_1')=1} G_q(f_1' f_2^2, F) u^{\deg(F)}.\\
\end{align*}

By applying this idea to each of the primes in the factorization of the square-free polynomial $f_1^*$, we obtain
\begin{align*}
\sum_{(F,f_1^*)=1} G_q(f_1^* f_2^2, F) & u^{\deg(F)}=\prod_{P\mid f_1^*}(1-(u^3q^2)^{\deg(P)})^{-1} \\
& \times \sum_{\ell\mid f_1^*} \mu(\ell) (u^2q^{3/2})^{\deg(\ell)}\overline{\left( \prod_{P\mid \ell}\chi_P\left(\frac{\ell}{P}\right)\right)}\chi_{\ell}\left(\frac{f_1^*f_2^2}{\ell}\right) \left(\prod_{P\mid \ell}\epsilon(\chi_{P^2})\omega(\chi_{P^2})\right)\\
&\times \Psi_q\left(\frac{f_1^*f_2^2}{\ell},u\right).
\end{align*}

Putting everything together in \eqref{psitilde}, we get
\begin{align*}
\tilde{\Psi}_q & (f,u) =  \prod_{P \mid f_2} \left( 1 - (u^3 q^2)^{\deg(P)} \right) ^{-1}\sum_{a \mid f_3^*} \mu(a) G_q(f_1 f_2^2, a) u^{\deg(a)} 
  \sum_{(F,a f_1)=1} G_q(a f_1 f_2^2, F) u^{\deg(F)}\\
=& \prod_{P \mid f_2} \left( 1 - (u^3 q^2)^{\deg(P)} \right) ^{-1}
\sum_{a \mid f_3^*} \mu(a) G_q(f_1 f_2^2, a) u^{\deg(a)} 
\prod_{P\mid af_1}(1-(u^3q^2)^{\deg(P)})^{-1} \\
& \times \sum_{\ell\mid af_1} \mu(\ell) (u^2q^{3/2})^{\deg(\ell)}\overline{\left( \prod_{P\mid \ell}\chi_P\left(\frac{\ell}{P}\right)\right)}\chi_{\ell}\left(\frac{af_1f_2^2}{\ell}\right) \left(\prod_{P\mid \ell}\epsilon(\chi_{P^2})\omega(\chi_{P^2})\right)\Psi_q\left(\frac{af_1f_2^2}{\ell},u\right)\\
=& \prod_{P \mid f_1f_2} \left( 1 - (u^3 q^2)^{\deg(P)} \right) ^{-1}\sum_{a \mid f_3^*} \mu(a) G_q(f_1 f_2^2, a) u^{\deg(a)} 
 \prod_{P\mid a}(1-(u^3q^2)^{\deg(P)})^{-1} \\
 &\times \sum_{\ell\mid af_1} \mu(\ell) (u^2q^{3/2})^{\deg(\ell)} \overline{\left( \prod_{P\mid \ell}\chi_P\left(\frac{\ell}{P}\right)\right)}\chi_{\ell}\left(\frac{af_1f_2^2}{\ell}\right) \left(\prod_{P\mid \ell}\epsilon(\chi_{P^2})\omega(\chi_{P^2})\right)  \Psi_q\left(\frac{af_1f_2^2}{\ell},u\right).
 \end{align*}
 Now note that
 $$ \prod_{P\mid \ell} \left ( \overline{\chi_P}\left(\frac{\ell}{P}\right) \epsilon(\chi_{P^2}) \omega(\chi_{P^2})\right ) = \frac{ \overline{G_q(1,\ell)}}{|\ell|_q^{1/2}},$$ which finishes the proof of the lemma.
\end{proof}

\subsection{Proof of Proposition \ref{big-F-tilde-corrected}}
We are now ready to prove Proposition \ref{big-F-tilde-corrected}.
\begin{proof}
By applying Perron's formula  (Lemma \ref{perron}) for a small  circle $C$ around the origin and using expression \eqref{expression}, we have
\begin{align} \nonumber
\sum_{\substack{F \in \mathcal{M}_{q,d} \\ (F,f)=1}} G_q(f,F) =& \frac{1}{2 \pi i} \oint_{C} \frac{ \Tilde{\Psi}_q(f,u)}{u^d} \, \frac{du}{u} = \frac{1}{2 \pi i} \oint_{C} \prod_{P \mid f_1f_2} \left( 1 - (u^3 q^2)^{\deg(P)} \right) ^{-1}\sum_{a \mid f_3^*} \mu(a) G_q(f_1 f_2^2, a)  \\ 
& \times \prod_{P\mid a}(1-(u^3q^2)^{\deg(P)})^{-1}  \sum_{\ell \mid af_1} \mu(\ell) |\ell|_q \overline{G_q(1,\ell)} \chi_{\ell}\left(\frac{af_1f_2^2}{\ell}\right)\nonumber \\
& \times  \frac{\Psi_q\left(\frac{af_1f_2^2}{\ell},u\right)u^{\deg(a)+2 \deg(\ell)} }{u^d}\, \frac{du}{u}.\label{integral}
\end{align}
Now we write $$ \Psi_q \left (\frac{af_1 f_2^2}{\ell} ,u \right )= (1-u^3q^3) \left[ \psi\left(\frac{af_1f_2^2}{\ell} ,0,u\right)+\psi\left(\frac{af_1f_2^2}{\ell} ,\pi_{\infty}^{-1},u\right)+ \psi\left(\frac{af_1f_2^2}{\ell} ,\pi_{\infty}^{-2},u\right) \right].$$
Each $\psi$ has three poles, at $q^{-4/3} \xi_3^{k}, k=0,1,2$, where $\xi_3=e^{2\pi i/3}$. We compute the residues of the poles in the integral above.
We recall that formula \eqref{meromorphic} gives $$\psi(f , \pi_{\infty}^{-j}, u) = \frac{u^jP(f,j,u^3)}{(1-q^4u^3)},$$ where $ \frac{{u^j}P(f,j,u^3)}{1-q^4u^3}$  is a power series whose nonzero coefficients correspond to monomials with $\deg \equiv j \pmod{3}$, and then
the only $\psi$ which gives a non-zero integral in equation \eqref{integral} comes from $\psi(af_1f_2^2/\ell,\pi_{\infty}^{-j},u)$ with $j$ such that $j + \deg(a) + 2 \deg(\ell) \equiv d \pmod 3$. 
Note that if $j+ \deg(a)+2\deg(\ell) \geq d+1$, the integral in \eqref{integral} is zero because the integrand has no poles inside $C$. Hence we assume that $j+\deg(a)+2\deg(\ell) \leq d$.


In \eqref{integral} we shift the contour of integration to $|u|=q^{-\sigma}$, where $2/3<\sigma<4/3$ and we encounter the poles when $u^3=q^{-4}$.
With $j$ as before, we compute the residue of the integrand at $u^3=q^{-4}$ and this gives 
\[\Res_{u=\xi_3^kq^{-4/3}} \psi\left(\frac{af_1f_2^2}{\ell},\pi_{\infty}^{-j},u\right) u^{\deg(a) +2\deg(\ell)-d-1}=\frac{1}{3}(q^{\frac{4}{3}} \xi_3^{-k})^{d-\deg(a)-2\deg(l)-j} \rho\left(\frac{af_1f_2^2}{\ell},j\right).\]
We get that
\begin{align*}
\sum_{\substack{F \in \mathcal{M}_{q,d} \\ (F,f)=1}} & G_q(f,F)  = \frac{q^{\frac{4}{3}(d - j)}}{\zeta_q(2)} \sum_{\substack{a\mid f_3^* \\ \deg(a) \leq d-j}} \frac{ \mu(a) G_q(f_1 f_2^2,a)}{|a|_q^{\frac{4}{3}}} \prod_{P\mid af_1} \left (1- \frac{1}{|P|_q^2} \right )^{-1} \\
& \times \sum_{\substack{\ell \mid af_1 \\ 2 \deg(\ell) \leq d-j-\deg(a)}}\frac{ \mu(\ell) \overline{G_q(1,\ell)}}{|\ell|_q^{\frac{5}{3}}} \chi_\ell\left(\frac{af_1f_2^2}{\ell}\right) \rho \left (\frac{af_1f_2^2}{\ell},j \right ) +  \frac{1}{2 \pi i} \oint_{|u|=q^{-\sigma}} \frac{ \Tilde{\Psi}_q(f,u)}{u^{d}} \, \frac{du}{u}.
\end{align*}
Using Lemma \ref{lemma-residue} and since $af_1/\ell$ is square-free and co-prime to $f_2$ it follows that
\begin{align*}
\rho\left(\frac{af_1f_2^2}{\ell},j\right) = \delta_{f_2=1} \overline{G_q \left ( 1, \frac{af_1}{\ell}\right )}\left|\frac{\ell}{af_1} \right|_q^{2/3} q^{\frac{4j}{3}- \frac{4}{3} \left[ j+\deg\left(\frac{af_1}{\ell}\right)\right]_3} \rho\left(1, \left[j+\deg\left(\frac{af_1}{\ell}\right)\right]_3\right).
\end{align*}
Note that $j +\deg\left(\frac{af_1}{\ell}\right) \equiv d+\deg(f_1) \pmod 3$, and
$$  G_q(f_1,a) \overline{G_q(1,\ell)} \chi_\ell\left(\frac{af_1}{\ell}\right)\overline{G_q \left (1, \frac{af_1}{\ell} \right )} =G_q(f_1,a) \overline{G_q(1,af_1)} = |a|_q \overline{G_q(1,f_1)},$$ where we used Lemma \ref{gauss}.
Combining the three equations above it follows that
\begin{align}
\sum_{\substack{F \in \mathcal{M}_{q,d} \\ (F,f)=1}} G_q(f,F)  = &\delta_{f_2=1}\frac{q^{\frac{4}{3}(d-[d+\deg(f_1)]_3)} \overline{G_q(1,f_1)}}{\zeta_q(2) |f_1|_q^{\frac{2}{3}}} \rho(1, [d+\deg(f_1)]_3)  \sum_{\substack{a\mid f_3^* \\ \deg(a) \leq d-j}} \frac{\mu(a)}{|a|_q} \nonumber \\
& \times \prod_{P\mid af_1} \left (1- \frac{1}{|P|_q^2} \right )^{-1}  \sum_{\substack{\ell \mid af_1 \\ 2 \deg(\ell) \leq d-j-\deg(a)}} \frac{\mu(\ell)}{|\ell|_q}+ O(q^{\sigma d} |f|_q^{\frac{1}{2} ( \frac{3}{2} - \sigma)+\varepsilon}), \label{interm-res}
\end{align}
where we have used Lemma \ref{our-psi} to bound the integral.

Now using Perron's formula  (Lemma \ref{perron}) for the sum over $\ell$ we have
\begin{equation}  \sum_{\substack{\ell \mid af_1 \\ 2 \deg(\ell) \leq d-j-\deg(a)}} \frac{\mu(\ell)}{|\ell|_q} = \frac{1}{2 \pi i}  \oint \frac{ \prod_{P\mid af_1} \left ( 1-\frac{x^{\deg(P)}}{|P|_q}\right )}{(1-x)x^{ \left[\frac{d-j-\deg(a)}{2}\right]}} \, \frac{dx}{x}, \label{sum-l}
\end{equation} where we are integrating along a small circle around the origin. Let $\alpha(a)=0$ if $\deg(a) \equiv d-j \pmod 2$ and $\alpha(a)=1$ otherwise. Introducing the sum over $a$ and using Perron's formula, it follows that
\begin{align}  
&\sum_{\substack{a\mid f_3^* \\ \deg(a) \leq d-j}} \frac{\mu(a)}{|a|_q} x^{\frac{\deg(a)+\alpha(a)}{2}} \prod_{P\mid a} \left (1- \frac{1}{|P|_q^2} \right )^{-1} \left ( 1-\frac{x^{\deg(P)}}{|P|_q}\right )  \nonumber \\
=& \frac{1+x^{\frac{1}{2}}}{2}\sum_{\substack{a\mid f_3^* \\ \deg(a) \leq d-j}}   \frac{\mu(a)}{|a|_q} x^{\frac{\deg(a)}{2}} \prod_{P\mid a} \left (1- \frac{1}{|P|_q^2} \right )^{-1} \left ( 1-\frac{x^{\deg(P)}}{|P|_q}\right )  \nonumber \\
&+(-1)^{d-j} \frac{1-x^{\frac{1}{2}}}{2}\sum_{\substack{a\mid f_3^* \\ \deg(a) \leq d-j}}  \frac{\mu(a)}{|a|_q} x^{\frac{\deg(a)}{2}} (-1)^{\deg(a)}\prod_{P\mid a} \left (1- \frac{1}{|P|_q^2} \right )^{-1} \left ( 1-\frac{x^{\deg(P)}}{|P|_q}\right )  \nonumber \\
=& \frac{1+x^{\frac{1}{2}}}{2}\frac{1}{2 \pi i} \oint \frac{ \prod_{P\mid f_3^*} \left (1 - \frac{ (x^{\frac{1}{2}} w)^{\deg(P)} \left(1- \frac{x^{\deg(P)}}{|P|_q}\right)}{|P|_q\left(1-\frac{1}{|P|_q^2}\right)} \right ) }{(1-w)w^{d-j}} \, \frac{dw}{w} \nonumber\\
&+ (-1)^{d-j} \frac{1-x^{\frac{1}{2}}}{2}\frac{1}{2 \pi i} \oint \frac{ \prod_{P\mid f_3^*} \left (1 - \frac{ (-x^{\frac{1}{2}} w)^{\deg(P)} \left(1- \frac{x^{\deg(P)}}{|P|_q}\right)}{|P|_q\left(1-\frac{1}{|P|_q^2}\right)} \right ) }{(1-w)w^{d-j}} \, \frac{dw}{w} \nonumber\\
=& \frac{1}{2 \pi i} \oint \frac{ \prod_{P\mid f_3^*} \left (1 - \frac{ (x^{\frac{1}{2}} w)^{\deg(P)} \left(1- \frac{x^{\deg(P)}}{|P|_q}\right)}{|P|_q\left(1-\frac{1}{|P|_q^2}\right)} \right ) }{(1-w^2)w^{d-j}} (1+x^{\frac{1}{2}}w) \, \frac{dw}{w}, \label{sum-a}
\end{align}
where again we are integrating along a small circle around the origin and we did the change of variables $w\rightarrow -w$ to the second integral to reach the last line. Let $\mathcal{R}(x,w)$ denote the Euler product above. Using equations \eqref{sum-l} and \eqref{sum-a} it follows that
\begin{align*}
& \sum_{\substack{a\mid f_3^* \\ \deg(a) \leq d-j}}  \frac{\mu(a)}{|a|_q}  \prod_{P\mid af_1} \left (1- \frac{1}{|P|_q^2} \right )^{-1}   \sum_{\substack{\ell \mid af_1 \\ 2 \deg(\ell) \leq d-j-\deg(a)}} \frac{\mu(\ell)}{|\ell|_q}\\
=&\frac{1}{(2 \pi i)^2} \oint \oint \prod_{P\mid f_1} \left (1- \frac{1}{|P|_q^2} \right )^{-1} \left ( 1-\frac{x^{\deg(P)}}{|P|_q}\right ) \frac{ \mathcal{R}(x,w)}{(1-x)(1-w^2) (x^{\frac{1}{2}}w)^{d-j}} (1+x^{\frac{1}{2}}w) \, \frac{dx}{x} \, \frac{dw}{w}.
\end{align*}
We first shift the contour in the integral over $x$ to $|x|= q^{1-\varepsilon}$ and we encounter a pole at $x=1$. We then shift the contour over $w$ to $|w|=q^{\frac{1}{2}-\varepsilon}$ and encounter a pole at $w=1$. Then 
\begin{align*}
\sum_{\substack{a\mid f_3^* \\ \deg(a) \leq d-j}}  &\frac{\mu(a)}{|a|_q}  \prod_{P\mid af_1} \left (1- \frac{1}{|P|_q^2} \right )^{-1}   \sum_{\substack{\ell\mid af_1 \\ 2 \deg(\ell) \leq d-j-\deg(a)}} \frac{\mu(\ell)}{|\ell|_q} = \prod_{P\mid f_1f_3^*} \left (1+\frac{1}{|P|_q} \right )^{-1} + O(q^{\varepsilon d - d}).
\end{align*}
Using the formula above in \eqref{interm-res} and the fact that $|G_q(1,f_1)| = |f_1|_q^{\frac{1}{2}}$ finishes the proof of the first statement of  Proposition \ref{big-F-tilde-corrected}.

\end{proof}

 \section{The non-Kummer setting}
 
We now assume that $q$ is odd with $q \equiv 2 \pmod 3$.
 We will prove Theorem \ref{thm-non-Kummer}.


 \subsection{Setup and sieving}

Using  Proposition \ref{prop-AFE} and Lemma \ref{non-kummer-desc}, 
we have to compute
\begin{eqnarray*}
\sum_{\substack{\chi  \; \mathrm{primitive \; cubic} \\ \mathrm{genus}(\chi) = g}} L_q(1/2, \chi)  &=&  S_{1, \mathrm{principal}}  + S_{1, \mathrm{dual}}, 
 \end{eqnarray*}
 where
 \begin{equation}
 S_{1, \mathrm{principal}} = \sum_{f \in \mathcal{M}_{q, \leq A}}\frac{1}{q^{\deg(f)/2}} \sum_{\substack{F \in \mathcal{H}_{q^2, g/2+1} \\ P \mid F \Rightarrow P \not\in \F_q[T]}}  \chi_F(f)
 + \frac{1}{1-\sqrt{q}} \sum_{f \in \mathcal{M}_{q, A+1}} \frac{1}{q^{\deg(f)/2}} \sum_{\substack{F \in \mathcal{H}_{q^2, g/2+1} \\ P \mid F \Rightarrow P \not\in \F_q[T]}}  \chi_F(f)
 \label{princ2}
 \end{equation} and
 \begin{equation}
 S_{1, \mathrm{dual}} = \sum_{f \in \mathcal{M}_{q, \leq g-A-1}}\frac{1}{q^{\deg(f)/2}}
 \sum_{\substack{F \in \mathcal{H}_{q^2, g/2+1} \\ P \mid F \Rightarrow P \not\in \F_q[T]}}  \omega(\chi_F) \overline{\chi_F}(f)
+\frac{1}{1-\sqrt{q}} \sum_{f \in \mathcal{M}_{q,g-A}} \frac{1}{q^{\deg(f)/2}} \sum_{\substack{F \in \mathcal{H}_{q^2, g/2+1} \\ P \mid F \Rightarrow P \not\in \F_q[T]}} \omega(\chi_F)  \overline{\chi_F}(f).
 \label{dual2}
 \end{equation}
We will choose $A \equiv 0 \pmod 3$. For the principal term, we will compute the contribution 
from cube polynomials $f$ and bound the contribution from non-cubes. We write
 $$S_{1,\mathrm{principal}} = S_{1,\cube}+S_{1, \neq \cube},$$ where 
$S_{1,\cube}$ corresponds to the sum with $f$ a cube in equation \eqref{princ2} and $S_{1,\neq \cube}$ corresponds to the sum with $f$ not a cube, namely,
\begin{equation}
S_{1,\cube} = \sum_{\substack{f \in \mathcal{M}_{q, \leq A}\\f = \tinycube}} \frac{1}{q^{\deg(f)/2}} \sum_{\substack{F \in \mathcal{H}_{q^2, g/2+1} \\ (F,f) =1 \\ P \mid F \Rightarrow P \not\in \F_q[T]}} 1,
\label{mt}
\end{equation} 
and 
$$
S_{1, \neq \cube} =  \sum_{\substack{f \in \mathcal{M}_{q, \leq A} \\ f \neq \tinycube}} \frac{1}{q^{\deg(f)/2}} \sum_{\substack{F \in \mathcal{H}_{q^2, \frac{g}{2}+1} \\ P \mid F \Rightarrow P \not\in \F_q[T] }} {\chi_F}(f)+ \frac{1}{1-\sqrt{q}} \sum_{f \in \mathcal{M}_{q,A+1}} \frac{1}{q^{\deg(f)/2}}\sum_{\substack{F \in \mathcal{H}_{q^2, \frac{g}{2}+1} \\ P \mid F \Rightarrow P \not\in \F_q[T]}}
 {\chi_F}(f).$$
Since $A \equiv 0 \pmod 3$, note that the second term in \eqref{princ2} does not contribute to the expression \eqref{mt}
for $S_{1,\cube}$.

The main results used to prove Theorem \ref{thm-non-Kummer} are summarized in the following lemmas whose proofs we postpone to the next sections. 
 \begin{lem} \label{mt_expression_nk}
The main term $S_{1, \cube}$ is given by the following asymptotic formula
 \begin{equation*}
S_{1, \cube} = \frac{ q^{g+2} \zeta_q(3/2)}{\zeta_q(3)} \mathcal{A}_{\mathrm{nK}} \left ( \frac{1}{q^2}, \frac{1}{q^{3/2}} \right ) + \frac{q^{g+2-\frac{A}{6}} \zeta_q(1/2)}{\zeta_q(3)} \mathcal{A}_{\mathrm{nK}} \left (\frac{1}{q^2}, \frac{1}{q}  \right ) +O(q^{g-\frac{A}{2}+\varepsilon g}),
\end{equation*}
with $\mathcal{A}_{\mathrm{nK}}(x,u)$ given by equation \eqref{a-euler}. In particular, 
$$ \mathcal{A}_{\mathrm{nK}} \left ( \frac{1}{q^2}, \frac{1}{q^{3/2}} \right ) =\prod_{\substack{R \in \mathbb{F}_q[T] \\ \deg(R)  \,\mathrm{ odd}}}  \left (1- \frac{1}{|R|_q^2+1}  \right ) \prod_{\substack{R \in \mathbb{F}_q[T] \\ \deg(R) \,\mathrm{ even}}} \left (1- \frac{1}{(|R|_q+1)^2}- \frac{2}{|R|_q^{\frac{1}{2}} (|R|_q+1)^2} \right )  ,$$ and
$$ \mathcal{A}_{\mathrm{nK}} \left ( \frac{1}{q^2},\frac{1}{q} \right ) = \prod_{\substack{R \in \mathbb{F}_q[T] \\ \deg(R)  \,\mathrm{ odd}}} \left (1- \frac{1}{|R|_q^2+1}  \right )\prod_{\substack{R \in \mathbb{F}_q[T] \\ \deg(R)  \,\mathrm{ even}}} \left ( 1- \frac{3}{(|R|_q+1)^2} \right ).$$
 \end{lem} 
In combination with the dual term $S_{1,\mathrm{dual}}$ this gives the following result. 
\begin{lem}\label{lemma_cube+dual}
\begin{equation*} 
S_{1,\cube}+S_{1,\mathrm{dual}} =  \frac{q^{g+2} \zeta_q(3/2)}{\zeta_q(3)} \mathcal{A}_{\mathrm{nK}} \left ( \frac{1}{q^2}, \frac{1}{q^{3/2}} \right ) +O \left (q^{g-\frac{A}{2}+\varepsilon g} +q^{\frac{5g}{6}+\varepsilon g} +q^{\frac{3g}{2}-(2-\sigma)A} \right ).
\end{equation*} 
\end{lem}
 We also have the following upper bound for $S_{1, \neq \cube}$.
 \begin{lem} We have that 
$$S_{1, \neq \cube} \ll q^{\frac{g+A}{2}+\varepsilon g}.$$ \label{non-cubes-2}
 \end{lem}

 \subsection{The main term}
 Here we will prove Lemma \ref{mt_expression_nk}. 
In equation \eqref{mt}, write $f=k^3$. Recall that $A \equiv 0 \pmod 3$. Then $S_{1,\cube}$ can be rewritten as
 \begin{equation*}
S_{1,\cube} = \sum_{k \in \mathcal{M}_{q, \leq \frac{A}{3}}} \frac{1}{q^{3 \deg(k)/2}} \sum_{\substack{F \in \mathcal{H}_{q^2, g/2+1} \\ (F,k) =1 \\ P \mid F \Rightarrow P \not\in \F_q[T]}} 1.
\end{equation*}
We first look at the generating series of the sum over $F$.  We use the fact that 
\begin{equation}
\sum_{\substack{D \in \F_q[T]\\D \mid F}} \mu(D) = \begin{cases} 1 & \mbox{if $F$ has no prime divisor in $\F_q[T]$,} \\ 0 & \mbox{otherwise}, \end{cases}
\label{sieve_F}
\end{equation}
where we have taken $\mu$ over $\F_q[T]$. 
Then
\begin{equation}
\sum_{\substack{F \in \mathcal{H}_{q^2} \\ (F,k) =1 \\ P \mid F \Rightarrow P \not\in \F_q[T]}} x^{\deg(F)} 
= \sum_{\substack{F \in \mathcal{H}_{q^2} \\ (F,k)=1}} x^{\deg(F)} \sum_{\substack{D \in \mathbb{F}_q[T] \\ D\mid F}} \mu(D) = \sum_{\substack{D \in \mathbb{F}_q[T] \\ (D,k)=1}} \mu(D) x^{\deg(D)} \sum_{\substack{F \in \mathcal{H}_{q^2} \\ (F,Dk)=1}} x^{\deg(F)}.
\label{sum_F}
\end{equation}
We evaluate the sum over $F$ in the equation above and we have that
$$ \sum_{\substack{F \in \mathcal{H}_{q^2} \\ (F,kD)=1}} x^{\deg(F)} = \prod_{\substack{P \in \mathbb{F}_{q^2}[T] \\ P \nmid Dk}} (1+x^{\deg(P)}) = \frac{ \mathcal{Z}_{q^2}(x)}{\mathcal{Z}_{q^2}(x^2)\displaystyle  \prod_{\substack{P \in \mathbb{F}_{q^2}[T] \\ P \mid Dk}} (1+x^{\deg(P)})},$$
 so from equation \eqref{sum_F} and the above it follows that
$$ \sum_{\substack{F \in \mathcal{H}_{q^2} \\ (F,k) =1 \\ P \mid F \Rightarrow P \not\in \F_q[T]}} x^{\deg(F)}  = \frac{ \mathcal{Z}_{q^2}(x)}{\mathcal{Z}_{q^2}(x^2) \displaystyle \prod_{\substack{P \in \mathbb{F}_{q^2}[T] \\ P\mid k}} (1+x^{\deg(P)})} \sum_{\substack{D \in \mathbb{F}_q[T] \\ (D,k)=1}} \frac{ \mu(D)x^{\deg(D)}}{ \displaystyle \prod_{\substack{P \in \mathbb{F}_{q^2}[T] \\ P\mid D}} (1+x^{\deg(P)})}.$$
Now we write down an Euler product for the sum over $D$ and we have that
\begin{align}
\sum_{\substack{D \in \mathbb{F}_q[T] \\ (D,k)=1}} \frac{ \mu(D)x^{\deg(D)}}{ \displaystyle \prod_{\substack{P \in \mathbb{F}_{q^2}[T] \\ P\mid D}} (1+x^{\deg(P)})} = \prod_{\substack{R \in \mathbb{F}_q[T] \\ (R,k)=1 \\ \deg(R)  \,\mathrm{ odd}}} \left ( 1- \frac{x^{\deg(R)}}{1+x^{\deg(R)}} \right )\prod_{\substack{R \in \mathbb{F}_q[T] \\ (R,k)=1 \\ \deg(R)  \,\mathrm{ even}}} \left ( 1- \frac{x^{\deg(R)}}{(1+x^{\frac{\deg(R)}{2}})^2} \right ), \label{sum_d}
\end{align} where the product over $R$ is over monic, irreducible polynomials. Let $A_R(x)$ denote the first Euler factor above and $B_R(x)$ the second. Then we rewrite the sum over $D$ as
$$ \eqref{sum_d} =  \frac{\displaystyle \prod_{\substack{R \in \mathbb{F}_q[T] \\ \deg(R)  \,\mathrm{ odd}}} A_R(x)  \prod_{\substack{R \in \mathbb{F}_q[T] \\ \deg(R)  \,\mathrm{ even}}} B_R(x)}{\displaystyle \prod_{\substack{R \in \mathbb{F}_q[T] \\ R\mid k \\ \deg(R)  \,\mathrm{ odd}}} A_R(x)  \prod_{\substack{R \in \mathbb{F}_q[T] \\ R\mid k \\ \deg(R)  \,\mathrm{ even}}} B_R(x) },$$ and putting everything together, it follows that
\begin{equation}
\sum_{\substack{F \in \mathcal{H}_{q^2} \\ (F,k) =1 \\ P \mid F \Rightarrow P \not\in \F_q[T]}} x^{\deg(F)}  = \frac{ \mathcal{Z}_{q^2}(x) \displaystyle \prod_{\substack{R \in \mathbb{F}_q[T] \\ \deg(R)  \,\mathrm{ odd}}} A_R(x)  \prod_{\substack{R \in \mathbb{F}_q[T] \\ \deg(R)  \,\mathrm{ even}}} B_R(x) }{\mathcal{Z}_{q^2}(x^2) \displaystyle \prod_{\substack{P \in \mathbb{F}_{q^2}[T] \\ P\mid k}} (1+x^{\deg(P)}) \displaystyle \prod_{\substack{R \in \mathbb{F}_q[T] \\ R\mid k \\ \deg(R)  \,\mathrm{ odd}}} A_R(x)  \prod_{\substack{R \in \mathbb{F}_q[T] \\ R\mid k \\ \deg(R)  \,\mathrm{ even}}} B_R(x)  }  .
\label{sum_F_2}
\end{equation}

We now introduce the sum over $k$ and we have
\begin{align*}
& \sum_{k \in \mathcal{M}_q}  \frac{ u^{\deg(k)} }{ \displaystyle \prod_{\substack{P \in \mathbb{F}_{q^2}[T] \\ P\mid k}} (1+x^{\deg(P)}) \displaystyle \prod_{\substack{R \in \mathbb{F}_q[T] \\ R\mid k \\ \deg(R)  \,\mathrm{ odd}}} A_R(x)  \prod_{\substack{R \in \mathbb{F}_q[T] \\ R\mid k \\ \deg(R)  \,\mathrm{ even}}} B_R(x) } \\
=&\prod_{\substack{R \in \mathbb{F}_q[T] \\ \deg(R)  \,\mathrm{ odd}}} \left[ 1+ \frac{u^{\deg(R)}}{(1+x^{\deg(R)}) A_R(x) (1-u^{\deg(R)})}\right] \prod_{\substack{R \in \mathbb{F}_q[T] \\ \deg(R)  \,\mathrm{ even}}} \left[1+ \frac{u^{\deg(R)}}{(1+x^{\frac{\deg(R)}{2}})^2 B_R(x) (1-u^{\deg(R)})} \right],
\end{align*}
where $R$ denotes a monic irreducible polynomial in $\mathbb{F}_q[T]$. Combining the equation above and \eqref{sum_F_2} we get that the generating series for the double sum over $F$ and $k$ is equal to 
\begin{align*}
\sum_{k \in \mathcal{M}_q} u^{\deg(k)}\sum_{\substack{F \in \mathcal{H}_{q^2} \\ (F,k) =1 \\ P \mid F \Rightarrow P \not\in \F_q[T]}} x^{\deg(F)}  =&\frac{ \mathcal{Z}_{q^2}(x)}{\mathcal{Z}_{q^2}(x^2)} \prod_{\substack{R \in \mathbb{F}_q[T] \\ \deg(R)  \,\mathrm{ odd}}} \frac{1}{(1+x^{\deg(R)})(1-u^{\deg(R)})} \nonumber  \\
& \times \prod_{\substack{R \in \mathbb{F}_q[T] \\ \deg(R)  \,\mathrm{ even}}} \frac{1}{(1+x^{\frac{\deg(R)}{2}})^2} \left ( 1+2x^{\frac{\deg(R)}{2}} + \frac{u^{\deg(R)}}{1-u^{\deg(R)}}\right )\nonumber  \\
=&\mathcal{Z}_q (u) \frac{ \mathcal{Z}_{q^2}(x)}{\mathcal{Z}_{q^2}(x^2)} \mathcal{A}_{\mathrm{nK}}(x,u),
 \label{gen_series}
\end{align*} 
where
\begin{equation}
\mathcal{A}_{\mathrm{nK}}(x,u) =  \prod_{\substack{R \in \mathbb{F}_q[T] \\ \deg(R)  \,\mathrm{ odd}}} \frac{1}{1+x^{\deg(R)}} \prod_{\substack{R \in \mathbb{F}_q[T] \\ \deg(R)  \,\mathrm{ even}}}\frac{1}{(1+x^{\frac{\deg(R)}{2}})^2} \left ( 1+2x^{\frac{\deg(R)}{2}} (1-u^{\deg(R)}) \right ). \label{a-euler}
\end{equation}
 Using Perron's formula  (Lemma \ref{perron}) twice in \eqref{mt} and the expression of the generating series above, we get that
$$ S_{1,\cube} = \frac{1}{(2 \pi i)^2} \oint \oint \frac{ \mathcal{A}_{\mathrm{nK}}(x,u)(1-q^2x^2)}{(1-qu)(1-q^2x)(1-q^{3/2}u) x^{\frac{g}{2}+1}(q^{3/2}u)^{\frac{A}{3}}} \, \frac{dx}{x} \, \frac{du}{u},$$
where we are integrating along circles of radii $|u|<\frac{1}{q^{3/2}}$ and $|x|<\frac{1}{q^2}$. First note that $\mathcal{A}_{\mathrm{nK}}(x,u)$ is analytic for $|x|<1/q, |xu|<1/q, |xu^2|<1/q^2$. We initially pick $|u|=1/q^{\frac{3}{2}+\varepsilon}$ and $|x|=1/q^{2+\varepsilon}$. We shift the contour over $x$ to $|x|=1/q^{1+\varepsilon}$ and we encounter a pole at $x=1/q^2$. Note that the new double integral will be bounded by $O(q^{\frac{g}{2}+\varepsilon g})$. Then
$$ S_{1,\cube} =\frac{q^{g+2}}{\zeta_q(3)} \frac{1}{2 \pi i} \oint \frac{\mathcal{A}_{\mathrm{nK}}( \tfrac{1}{q^2},u) }{(1-qu)(1-q^{3/2}u) (q^{3/2}u)^{\frac{A}{3}}} \, \frac{du}{u}+ O(q^{\frac{g}{2}+\varepsilon g}).$$
Now we shift the contour of integration to $|u|=q^{-\varepsilon}$ and we encounter two simple poles: one at $u=1/q^\frac{3}{2}$ and one at $u=1/q$. We evaluate the residues and then 
\begin{equation*}
S_{1,\cube} = \frac{ q^{g+2} \zeta_q(3/2)}{\zeta_q(3)} \mathcal{A}_{\mathrm{nK}} \left ( \frac{1}{q^2}, \frac{1}{q^{3/2}} \right ) + \frac{q^{g+2-\frac{A}{6}} \zeta_q(1/2)}{\zeta_q(3)} \mathcal{A}_{\mathrm{nK}} \left (\frac{1}{q^2}, \frac{1}{q}  \right ) +O(q^{g-\frac{A}{2}+\varepsilon g}),
\end{equation*} which finishes the proof of Lemma \ref{mt_expression_nk}.

 \subsection{The contribution from non-cubes}

 Recall that $S_{1, \neq \cube}$ is the term with $f$ not a cube in $S_{1, \mathrm{principal}}$ of \eqref{princ2}. Since $A \equiv 0 \pmod 3$, the term we want to bound is equal to 


 $$
S_{1, \neq \cube} =  \sum_{\substack{f \in \mathcal{M}_{q, \leq A} \\ f \neq \tinycube}} \frac{1}{q^{\deg(f)/2}} \sum_{\substack{F \in \mathcal{H}_{q^2, \frac{g}{2}+1} \\ P \mid F \Rightarrow P \not\in \F_q[T] }} {\chi_F}(f)+ \frac{1}{1-\sqrt{q}} \sum_{f \in \mathcal{M}_{q,A+1}} \frac{1}{q^{\deg(f)/2}}\sum_{\substack{F \in \mathcal{H}_{q^2, \frac{g}{2}+1} \\ P \mid F \Rightarrow P \not\in \F_q[T]}}
 {\chi_F}(f).$$
Let $S_{11}$ be the first term above and $S_{12}$ the second. Note that it is enough to bound $S_{11}$, since bounding $S_{12}$ will follow in a similar way. We use equation \eqref{sieve_F} again for the sum over $F$ and we have
\begin{equation}
S_{11}=  \sum_{\substack{f \in \mathcal{M}_{q, \leq A} \\ f \neq \tinycube}} \frac{1}{q^{\deg(f)/2}} \sum_{\substack{D \in \mathcal{M}_{q, \leq \frac{g}{2}+1} \\ (D,f)=1}} \mu(D) \sum_{\substack{F \in \mathcal{H}_{q^2,\frac{g}{2}+1-\deg(D)} \\ (F,D)=1}} \chi_F(f).
\label{non_cube_term}
\end{equation}
Note that we used the fact that $\chi_D(f)=1$ since $D,f \in \mathbb{F}_q[T]$. Now we look at the generating series for the sum over $F$. We have the following.
\begin{align*}
 \sum_{\substack{F \in \mathcal{H}_{q^2} \\ (F,D)=1}} \chi_F(f) u^{\deg(F)} = \prod_{\substack{P \in \mathbb{F}_{q^2}[T] \\ P \nmid Df}} \left ( 1+\chi_P (f) u^{\deg(P)} \right ) = \frac{ \mathcal{L}_{q^2} \left(u, \chi_f\right) }{\mathcal{L}_{q^2}(u^2, \overline{\chi_f})} \prod_{\substack{P \in \mathbb{F}_{q^2}[T] \\ P \nmid f \\ P\mid D}} \frac{  1- \chi_P(f) u^{\deg(P)} }{   1- \overline{\chi_P}(f) u^{2 \deg(P)} }.
\end{align*}
Using Perron's formula  (Lemma \ref{perron}) and the generating series above, we have
\begin{align*}
 \sum_{\substack{F \in \mathcal{H}_{q^2,\frac{g}{2}+1-\deg(D)} \\ (F,D)=1}} \chi_F(f) = \frac{1}{2 \pi i} \oint \frac{ \mathcal{L}_{q^2} \left(u,\chi_f\right) }{\mathcal{L}_{q^2}(u^2, \overline{\chi_f}) u^{\frac{g}{2}+1-\deg(D)}} \prod_{\substack{P \in \mathbb{F}_{q^2}[T] \\ P \nmid f \\ P\mid D}} \frac{  1- \chi_P(f) u^{\deg(P)} }{   1- \overline{\chi_P}(f) u^{2 \deg(P)} } \, \frac{du}{u},
\end{align*}
where we are integrating along a circle of radius $|u|= \frac{1}{q}$ around the origin. Now we use the Lindel\"of bound for the $L$--function in the numerator and a lower bound for the $L$--function in the denominator. We have, by Lemmas \ref{lindelof} and \ref{folednil},
$$ \left|  \mathcal{L}_{q^2} \left(u,\chi_f\right) \right| \ll q^{2 \varepsilon \deg(f)}, \, \, \, \left| \mathcal{L}_{q^2} (u^2, \overline{\chi_f})  \right| \gg q^{-2 \varepsilon\deg(f)}.$$



Then
$$  \sum_{\substack{F \in \mathcal{H}_{q^2,\frac{g}{2}+1-\deg(D)} \\ (F,D)=1}} \chi_F(f) \ll q^{\frac{g}{2}-\deg(D)} q^{4 \varepsilon \deg(f)+2 \varepsilon \deg(D)}.$$
Trivially bounding the sums over $D$ and $f$ in \eqref{non_cube_term} gives a total upper bound of
\begin{equation*}
S_{11}\ll q^{\frac{A+g}{2}+\varepsilon g},
\end{equation*}
and similarly for $S_{12}$.  This finishes the proof of Lemma \ref{non-cubes-2}.

 \subsection{The dual term}
 Here we will evaluate $S_{1,\mathrm{dual}}$ and prove Lemma \ref{lemma_cube+dual}. Recall the expression \eqref{dual2} for $S_{1,\mathrm{dual}}$. We further write $S_{1,\mathrm{dual}} = S_{11,\mathrm{dual}}+S_{12,\mathrm{dual}}$.
 
For $F$ as in the expression \eqref{dual2}, we have that $\chi_F$ is an even primitive character over $\F_q[T]$ of modulus $F\tilde{F}$ (recall that $\tilde{F}$ is the Galois conjugate of $F$).
The modulus has degree $2\deg(F)=g+2$ and by Corollary \ref{sign-GS} the sign of the functional equation is 
$$\omega(\chi_F) = q^{-\frac{g}{2}-1} G(\chi_F),$$
where the Gauss sum is 
$$
G(\chi_F) = \sum_{\alpha \in \F_q[T]/(F \tilde{F})} \chi_F(\alpha) \; e_q \left( \frac{\alpha}{F \tilde{F}} \right).
$$

By the Chinese Remainder Theorem, since $F$ and $\tilde{F}$ are co-prime, if $\beta$ runs over the classes in $\F_{q^2}[T]/(F)$ then 
$\beta \tilde{F}+\tilde{\beta}F$  runs over the classes in $\F_q[T]/(F\tilde{F})$. Then
\begin{align*}
G(\chi_F) =& \sum_{\beta \in \F_{q^2}[T]/(F)} \chi_F(\beta \tilde{F}) \; e_q \left( \frac{\beta \tilde{F}+\tilde{\beta}F}{F \tilde{F}} \right)\\
=& \sum_{\beta \in \F_{q^2}[T]/(F)} \chi_F(\beta ) \; e_{q^2} \left( \frac{\beta}{F} \right)\\
=& G_{q^2}(1, F), 
\end{align*}
where we have used that $\chi_F(\tilde{F})=1$ due to cubic reciprocity. 

Using the fact that
$G_{q^2}(1,F)\overline{\chi_F}(f)=G_{q^2}(f,F)$ when $(f,F)=1$ and $\overline{\chi_F}(f)=0$ otherwise, we get 
 \begin{equation} \label{dual21}
S_{11,\mathrm{dual}}=q^{-\frac{g}{2}-1} \sum_{f \in \mathcal{M}_{q,\leq g-A-1}} \frac{1}{q^{\deg(f)/2}} \sum_{\substack{F \in \mathcal{H}_{q^2,\frac{g}{2}+1} \\ (F,f)=1 \\ P \mid F \Rightarrow P \not\in \F_q[T] }}
G_{q^2}(f,F),
\end{equation} 
and
\begin{equation} \label{dual22}
S_{12,\mathrm{dual}} = \frac{q^{-\frac{g}{2}-1}}{1-\sqrt{q}} \sum_{f \in \mathcal{M}_{q,g-A}} \frac{1}{q^{\deg(f)/2}} \sum_{\substack{F \in \mathcal{H}_{q^2,\frac{g}{2}+1} \\ (F,f)=1 \\ P \mid F \Rightarrow P \not\in \F_q[T]}}
G_{q^2}(f,F).
\end{equation}
 We first prove the following important feature of $G_{q^2}(1,f)$. 
\begin{lem} \label{no-oscillation}
Let $f \in \F_q[T]$ be square-free. Then 
\[G_{q^2}(1,f)=q^{\deg(f)}.\]
\end{lem}
\begin{proof}
As usual, we denote by $\tilde{\alpha}$ the Galois conjugate of $\alpha$. We have
 \begin{align*}
 \overline{G_{q^2}(1,f)}=& \sum_{\alpha \in \F_{q^2}[T]/(f)} \overline{\chi_{f}}(\alpha)e_{q^2}\left(\frac{-\alpha}{f} \right)= \sum_{\alpha \in \F_{q^2}[T]/(f)} \chi_{f}(\tilde{\alpha})e_{q^2}\left(\frac{-\tilde{\alpha}}{f} \right)\\
=& \sum_{\alpha \in \F_{q^2}[T]/(f)} \chi_{f}(\alpha)e_{q^2}\left(\frac{-\alpha}{f} \right)= \chi_{f}(-1) \sum_{\alpha \in \F_{q^2}[T]/(f)} \chi_{f}(\alpha)e_{q^2}\left(\frac{\alpha}{f} \right)\\
=& G_{q^2}(1,f).
\end{align*}
In the first line we used the fact that  $e_{q^2}(-\alpha/f) = e_{q^2}(- \tilde{\alpha}/f)$ which follows because $\tr(\alpha)=\tr(\tilde{\alpha})$.
 In the second line we used that $\chi_{f}(-1)=\Omega^{-1}( (-1)^{ \frac{q^2-1}{3} \deg(f)} ) =1$.

Notice that for $f,g \in \F_q[T]$, $(f,g)=1$, $\chi_f(g)=1$ because
\[\overline{\chi_{f}}(g)=\chi_{\tilde{f}}(\tilde{g})=\chi_{f}(g)\]
which implies that $\chi_{f}(g)\in \R$, hence it has to be equal to 1.

Then by Lemma \ref{gauss}, we have that 
\[G_{q^2}(1,fg)=G_{q^2}(1,f)G_{q^2}(1,g).\]

Now if $P\in \F_q[T]$, then
\[G_{q^2}(1,P)^2=\epsilon(\chi_P)^2\omega(\chi_P)^2|P|_{q^2} = \overline{\epsilon(\chi_P)\omega(\chi_P)|P|_{q^2}^{1/2}}|P|_{q^2}^{1/2} =\overline{G_{q^2}(1,P)}q^{\deg(P)}\]
and from this we conclude that 
\[G_{q^2}(1,P)=q^{\deg(P)}.\]
By multiplicativity, since $f$ is square-free, 
\[G_{q^2}(1,f)=q^{\deg(f)}.\]

\end{proof}

Now we go back to \eqref{dual21} and \eqref{dual22}. Using the sieve \eqref{sieve_F},
we get that
\begin{eqnarray}
 \sum_{\substack{F \in \mathcal{H}_{q^2, \frac{g}{2}+1} \\ (F,f)=1 \\ P \mid F \Rightarrow P \not\in \F_q[T]}} G_{q^2}(f, F) &=& \sum_{\substack{D \in \F_q[T]\\ \deg(D) \leq g/2 + 1\\(D,f)=1}} \mu(D) \sum_{\substack{F \in \mathcal{M}_{ q^2,g/2+1-\deg(D)}\\
(F, f)=1}} G_{q^2}(f, DF) \nonumber \\
&=& \sum_{\substack{D \in \F_q[T]\\ \deg(D) \leq g/2 + 1\\(D,f)=1}} \mu(D) G_{q^2}(f, D)  \sum_{\substack{F \in \mathcal{M}_{ q^2,g/2+1-\deg(D)}\\
(F, Df)=1}}  \chi_F^2(D) G_{q^2}(f, F) \nonumber \\
&=& \sum_{\substack{D \in \F_q[T]\\ \deg(D) \leq g/2 + 1\\(D,f)=1}} \mu(D) G_{q^2}(f, D)   \sum_{\substack{F \in \mathcal{M}_{ q^2,g/2+1-\deg(D)}\\
(F, Df)=1}} G_{q^2}(fD, F), \label{sieve_D}
\end{eqnarray}
where we have used that $G_{q^2}(f,DF)=0$ if $(D,F)\not = 1$, since $(f,DF)=1$.

Using Proposition \ref{big-F-tilde-corrected} (recall that we are working in $\mathbb{F}_{q^2}[T]$) we get that
\begin{align}
\sum_{\substack{F \in \mathcal{M}_{q^2,g/2+1-\deg(D)}\\(F,fD)=1}} &  G_{q^2}(fD,F)  = \delta_{f_2=1} \frac{q^{\frac{4g}{3}+\frac{8}{3}-4\deg(D)-\frac{4}{3}\deg(f_1) - \frac{8}{3} [g/2+1+ \deg(f_1)]_3}}{\zeta_{q^2}(2)}
\overline{G_{q^2}(1,f_1D)} \nonumber\\
&\times \rho(1, [g/2+1+\deg(f_1)]_3)\prod_{\substack{P \in \mathbb{F}_{q^2}[T] \\P\mid fD}} \left (1+ \frac{1}{|P|_{q^2}} \right )^{-1}\nonumber\\
& +O \left(\delta_{f_2=1}q^{\frac{g}{3}+\varepsilon g-\deg(D)(1+2\varepsilon)-\frac{\deg(f)_1}{3}} \right) + \frac{1}{2 \pi i} \oint_{|u|=q^{-2\sigma}} \frac{ \Tilde{\Psi}_{q^2}(fD,u)}{u^{g/2+1-\deg(D)}} \, \frac{du}{u}, \label{et2}
\end{align}
with $\delta_{f_2=1} = 1$ if $f_2=1$ and $\delta_{f_2=1}=0$ otherwise.
Combining equations \eqref{dual21}, \eqref{sieve_D}, \eqref{et2} and Lemma \ref{no-oscillation}, we write 
\begin{equation}
S_{11,\mathrm{dual}} = M_1+E_1, \label{m-e}
\end{equation} where $M_1$ corresponds to the main term in \eqref{et2} and $E_1$ corresponds to the two error terms in \eqref{et2}. 
We have
\begin{align*}
M_1= & \frac{q^{5g/6+5/3}}{\zeta_{q^2}(2)} \sum_{f \in \mathcal{M}_{q,\leq g-A-1}} \frac{\delta_{f_2=1}q^{- \frac{8}{3} [g/2+1+ \deg(f_1)]_3}}{q^{\deg(f)/2+\deg(f_1)/3}} \sum_{\substack{D \in \F_q[T]\\ \deg(D) \leq g/2 + 1\\(D,f)=1}} \mu(D)     q^{-4\deg(D)}
|G_{q^2}(1,D)|^2\\
&\times \rho(1, [g/2+1+\deg(f_1)]_3)\prod_{\substack {P \in \F_{q^2}[T] \\ P\mid fD}} \left (1+ \frac{1}{|P|_{q^2}} \right )^{-1} \\
=& \frac{q^{5g/6+5/3}}{\zeta_{q^2}(2)} \sum_{f \in \mathcal{M}_{q,\leq g-A-1}} \frac{\delta_{f_2=1}q^{- \frac{8}{3} [g/2+1+\deg(f_1)]_3}}{q^{\deg(f)/2+\deg(f_1)/3}} 
\rho(1, [g/2+1+ \deg(f_1)]_3)\\
&\times \prod_{\substack {P \in \F_{q^2}[T] \\P\mid f}} \left (1+ \frac{1}{|P|_{q^2}} \right )^{-1} \sum_{\substack{D \in \F_q[T]\\ \deg(D) \leq g/2 + 1\\(D,f)=1}} \mu(D)q^{-2\deg(D)}
\prod_{\substack {P \in \F_{q^2}[T] \\P\mid D}} \left (1+ \frac{1}{|P|_{q^2}} \right )^{-1}.
\end{align*}

We first treat the sum over $D$. We consider the generating series of the sum over $D$. We have that 
\begin{align*}
&\sum_{\substack{D \in \mathbb{F}_q[T] \\ (D,f)=1}}  \frac{ \mu(D)}{q^{2 \deg(D)}} \prod_{\substack{P \in \mathbb{F}_{q^2}[T] \\ P\mid D}} \left ( 1+ \frac{1}{|P|_{q^2}} \right )^{-1} w^{\deg(D)} \\
=&\prod_{\substack{R \in \mathbb{F}_q[T] \\ \deg(R)  \,\mathrm{ odd} \\ R \nmid f}} \left[ 1- \frac{w^{\deg(R)}}{q^{2 \deg(R)} (1+ \frac{1}{q^{2 \deg(R)}})}\right]\prod_{\substack{R \in \mathbb{F}_q[T] \\ \deg(R)  \,\mathrm{ even} \\ R \nmid f}} \left[ 1- \frac{w^{\deg(R)}}{q^{2 \deg(R)} ( 1+ \frac{1}{q^{\deg(R)}})^2}\right],
\end{align*}
where we have counted the primes in $\F_{q^2}[T]$ by counting the primes of $\F_q[T]$ lying under them.
Recall from Section \ref{section-cubic} that $P\in \F_q[T]$ splits in $\F_{q^2}[T]$ if and only if $\deg(P)$ is even. 

Let $A_{\mathrm{dual},R}(w)$ denote the first factor above and $B_{\mathrm{dual},R}(w)$ the second. Define
\begin{equation*}
\mathcal{J}_{\mathrm{nK}}(w) = \prod_{\substack{R \in \mathbb{F}_q[T]
\\ \deg(R)  \,\mathrm{ odd}}} 
A_{\mathrm{dual},R}(w) \prod_{\substack{R \in \mathbb{F}_q[T] 
\\ \deg(R)  \,\mathrm{ even}}} 
B_{\mathrm{dual},R}(w).
\label{gw}
\end{equation*}
which is absolutely convergent for $|w|<q$. 

Then by Perron's formula  (Lemma \ref{perron}) we have
\begin{align*}
\sum_{\substack{D \in \F_q[T]\\ \deg(D) \leq g/2 + 1\\(D,f)=1}} \mu(D)     q^{-2\deg(D)}  & \prod_{\substack{P \in \mathbb{F}_{q^2}[T] \\ P\mid D}} \left (1+ \frac{1}{|P|_{q^2}} \right )^{-1}  = \frac{1}{2 \pi i} \oint \frac{\mathcal{J}_{\mathrm{nK}}(w)}{w^{g/2+1}(1-w)} \\
& \times \prod_{\substack{R \in \mathbb{F}_q[T] \\ \deg(R)  \,\mathrm{ odd} \\ R \mid f}} A_{\mathrm{dual},R}(w)^{-1} \prod_{\substack{R \in \mathbb{F}_q[T] \\ \deg(R)  \,\mathrm{ even} \\ R\mid f}} B_{\mathrm{dual},R}(w)^{-1} \, \frac{dw}{w}.
\end{align*}

Now we introduce the sum over $f$. 
Using the expression for the sum over $D$ above, we get that 

\begin{align*}
& M_1= \frac{q^{5g/6+5/3}}{\zeta_{q^2}(2)} \sum_{f \in \mathcal{M}_{q,\leq g-A-1}} \frac{\delta_{f_2=1}\rho(1, [g/2+1+ \deg(f_1)]_3)}{q^{\frac{8}{3} [g/2+1+\deg(f_1)]_3}q^{\deg(f)/2+\deg(f_1)/3}} 
\prod_{\substack{R \in \mathbb{F}_q[T] \\ \deg(R)  \,\mathrm{ odd} \\ R \mid f}} \left(1+\frac{1}{q^{2\deg(R)}}\right)^{-1}
\\
&\times \prod_{\substack{R \in \mathbb{F}_q[T] \\ \deg(R)  \,\mathrm{ even} \\ R \mid f}} \left(1+\frac{1}{q^{\deg(R)}}\right)^{-2} \frac{1}{2 \pi i} \oint \frac{\mathcal{J}_{\mathrm{nK}}(w)}{w^{g/2+1}(1-w)}\prod_{\substack{R \in \mathbb{F}_q[T] \\ \deg(R)  \,\mathrm{ odd} \\ R \mid f}} A_{\mathrm{dual},R}(w)^{-1} \prod_{\substack{R \in \mathbb{F}_q[T] \\ \deg(R)  \,\mathrm{ even} \\ R\mid f}} B_{\mathrm{dual},R}(w)^{-1} \, \frac{dw}{w}.
\end{align*}
Let
\begin{align*}
\mathcal{H}_{\mathrm{nK}}(u,w)= &\sum_{f}
\frac{\delta_{f_2=1}}{q^{\deg(f)/2+\deg(f_1)/3}} 
\prod_{\substack{R \in \mathbb{F}_q[T] \\ \deg(R)  \,\mathrm{ odd} \\ R \mid f}} C_R(w)^{-1} \prod_{\substack{R \in \mathbb{F}_q[T] \\ \deg(R)  \,\mathrm{ even} \\ R\mid f}} D_R(w)^{-1} u^{\deg(f)},
\end{align*}
where 
\[C_R(w)=1+\frac{1}{q^{2\deg(R)}}-\frac{w^{\deg(R)}}{q^{2\deg(R)}}, \qquad D_R(w)=\left(1+\frac{1}{q^{\deg(R)}}\right)^2-\frac{w^{\deg(R)}}{q^{2\deg(R)}}.\]

Then we can write down an Euler product for $\mathcal{H}_{\mathrm{nK}}(u,w)$ and we have that
\begin{align*}
 \mathcal{H}_{\mathrm{nK}}(u,w)= & \prod_{\substack{R \in \mathbb{F}_q[T] \\ \deg(R)  \,\mathrm{ odd}}} \left[1+C_R(w)^{-1} \left ( \frac{1}{q^{\deg(R)/3}} \sum_{j=0}^{\infty} \frac{ u^{(3j+1) \deg(R)}}{q^{(3j+1) \deg(R)/2}} + \sum_{j=1}^{\infty} \frac{ u^{3j \deg(R)}}{q^{3j\deg(R)/2}}\right ) \right] \\
 & \times  \prod_{\substack{R \in \mathbb{F}_q[T] \\ \deg(R)  \,\mathrm{ even}}} \left[1+D_R(w)^{-1} \left ( \frac{1}{q^{\deg(R)/3}} \sum_{j=0}^{\infty} \frac{ u^{(3j+1) \deg(R)}}{q^{(3j+1) \deg(R)/2}} + \sum_{j=1}^{\infty} \frac{ u^{3j \deg(R)}}{q^{3j\deg(R)/2}}\right ) \right].
\end{align*}
After simplifying, we have
\begin{align}
\mathcal{H}_{\mathrm{nK}}(u,w) =& \prod_{\substack{R \in \mathbb{F}_q[T] \\ \deg(R)  \,\mathrm{ odd}}} \left[1+C_R(w)^{-1} \left ( \frac{u^{\deg(R)}}{|R|_q^{5/6} (1- \frac{u^{3 \deg(R)}}{|R|_q^{3/2}})} + \frac{ u^{3 \deg(R)}}{ |R|_q^{3/2}-u^{3 \deg(R)}}\right )\right] \nonumber \\
& \times  \prod_{\substack{R \in \mathbb{F}_q[T] \\ \deg(R)  \,\mathrm{ even}}} \left[1+D_R(w)^{-1} \left ( \frac{u^{\deg(R)}}{|R|_q^{5/6} (1- \frac{u^{3 \deg(R)}}{|R|_q^{3/2}})} + \frac{ u^{3 \deg(R)}}{ |R|_q^{3/2}-u^{3 \deg(R)}}\right ) \right] \nonumber \\
=&\mathcal{Z} \left ( \frac{u}{q^{5/6}}\right ) \mathcal{B}_{\mathrm{nK}}(u,w), \label{b-def}
\end{align}
with $\mathcal{B}_{\mathrm{nK}}(u,w)$ analytic in a wider region (for example, $\mathcal{B}_{\mathrm{nK}}(u,w)$ is absolutely convergent for $|u|<q^{\frac{11}{6}}$ and $|uw|< q^{\frac{11}{6}}$). 

We will use Perron's formula  (Lemma \ref{perron}) for the sum over $f$.  Note that if $g/2+1+\deg(f_1) \equiv 0 \pmod 3$, then $\deg(f_1) \equiv g-1 \pmod 3$. In this case by Lemma \ref{lemma-residue}, $\rho(1,0)= 1$. 

If $g/2+1+\deg(f_1) \equiv 1 \pmod 3$, then $\deg(f_1) \equiv g \pmod 3$, and by Lemma \ref{lemma-residue} again we have $\rho(1,1)=\tau(\chi_3) q^2$. Note that $\tau(\chi_3) = q \epsilon(\chi_3)$ and $\epsilon(\chi_3) = (-1)^{\frac{q^2-1}{3}}=1$. Since $q$ is odd, we have $\rho(1,1)=q^3$. Recall that $A \equiv 0 \pmod 3$. Using Perron's formula  (Lemma \ref{perron}) twice depending on whether $\deg(f) \equiv g-1 \pmod 3$ or $\deg(f) \equiv g \pmod 3$, we have
\begin{align*}\label{twoterms}
M_1 =&\frac{q^{5g/6+5/3}}{\zeta_{q^2}(2)} \oint \oint \frac{\mathcal{H}_{\mathrm{nK}}(u,w) \mathcal{J}_{\mathrm{nK}}(w)}{w^{g/2+1}(1-w)} \left[ \frac{1}{u^{g-A-1} (1-u^3)} + \frac{ q^{1/3}}{u^{g-A-3}(1-u^3)} \right] \, \frac{dw}{w} \, \frac{du}{u},
\end{align*}
where we integrate along small circles around the origin. We first shift the contour over $w$ to $|w|=q^{1-\epsilon}$ (since $\mathcal{J}_{\mathrm{nK}}(w)$ is absolutely convergent for $|w|<q$) and encounter the pole at $w=1$. Note that $\mathcal{H}_{\mathrm{nK}}(u,1)$ has a pole at $u=q^{-1/6}$. Let
 \begin{equation}
 \mathcal{K}_{\mathrm{nK}}(u)= \mathcal{B}_{\mathrm{nK}}(u,1) \mathcal{J}_{\mathrm{nK}}(1). \label{f-def}
 \end{equation} Then
\begin{align*}
M_1 =&\frac{q^{5g/6+5/3}}{\zeta_{q^2}(2)} \oint \frac{ \mathcal{K}_{\mathrm{nK}}(u)}{(1-uq^{1/6})(1-u^3) u^{g-A-1}} (1+ q^{1/3}u^2) \, \frac{du}{u}\\
&+ \frac{q^{5g/6+5/3}}{\zeta_{q^2}(2)} \oint_{|u|=q^{-1/6-\varepsilon}} \oint_{|w|=q^{1-\varepsilon}} \frac{\mathcal{H}_{\mathrm{nK}}(u,w) \mathcal{J}_{\mathrm{nK}}(w)}{w^{g/2+1}(1-w)} \left[ \frac{1}{u^{g-A-1} (1-u^3)} + \frac{ q^{1/3}}{u^{g-A-3}(1-u^3)} \right] \, \frac{dw}{w} \, \frac{du}{u}\\
=&\frac{q^{5g/6+5/3}}{\zeta_{q^2}(2)} \oint \frac{ \mathcal{K}_{\mathrm{nK}}(u)}{(1-uq^{1/6})(1-u^3) u^{g-A-1}} (1+ q^{1/3}u^2) \, \frac{du}{u}+O\left(q^{\frac{g}{2}-\frac{A}{6}+\varepsilon g}\right).
\end{align*}
Note that $\mathcal{K}_{\mathrm{nK}}(u)$ is absolutely convergent for $|u|<q^{\frac{1}{6}}$. We shift the contour of integration to $|u|=q^{-\varepsilon}$, we compute the residue at $u=q^{-1/6}$ and we get that
\begin{align}
M_1=& 2q^{g-\frac{A}{6}+2} \frac{ \mathcal{K}_{\mathrm{nK}}(q^{-1/6})}{\zeta_{q^2}(2) (\sqrt{q}-1)}+O\left(q^{\frac{g}{2}-\frac{A}{6}+\varepsilon g} \right) \nonumber\\
&+\frac{q^{5g/6+5/3}}{\zeta_{q^2}(2)} \oint_{|u|=q^{-\varepsilon}} \frac{ \mathcal{K}_{\mathrm{nK}}(u)}{(1-uq^{1/6})(1-u^3) u^{g-A-1}} (1+ q^{1/3}u^2) \, \frac{du}{u}\nonumber\\
=& 2q^{g-\frac{A}{6}+2} \frac{ \mathcal{K}_{\mathrm{nK}}(q^{-1/6})}{\zeta_{q^2}(2) (\sqrt{q}-1)}+O\left(q^{\frac{5g}{6}+\varepsilon g}\right).
\label{main_dual}
\end{align}

Now we consider the error term $E_1$ from equation \eqref{m-e}. The first term coming from the first error in equation \eqref{et2} will be bounded by
\begin{align*} 
\ll &q^{-\frac{g}{2}}\sum_{f \in \mathcal{M}_{q,\leq g-A-1}}\frac{1}{q^{\deg(f)/2}} \sum_{\deg(D)\leq g/2+1} q^{\deg(D)} q^{\frac{g}{3}+\varepsilon g-\deg(D)(1+2\varepsilon)-\frac{\deg(f)_1}{3}}
\ll q^{\left(\frac{1}{2}+\varepsilon\right)g -\frac{A}{6}}
\end{align*}
Then we get that 
\begin{align*}
E_1 =&q^{-g/2-1} \frac{1}{2 \pi i} \oint_{|u|=q^{-2\sigma}} \sum_{f \in \mathcal{M}_{q,\leq g-A-1}} \frac{1}{q^{\deg(f)/2}} \sum_{\substack{D \in \mathbb{F}_q[T] \\ \deg(D) \leq g/2+1 \\ (D,f)=1}} \mu(D) G_{q^2}(f,D) \frac{\Tilde{\Psi}_{q^2}(fD,u)}{u^{g/2+1-\deg(D)}} \, \frac{du}{u} \\
&+O(q^{\frac{g}{2} - \frac{A}{6} + \varepsilon g}),
\end{align*} where recall that $2/3<\sigma<4/3$.

Combining the expressions for $M_1$ and $E_1$ it follows that
\begin{align*}
S_{11,\mathrm{dual}} =&2q^{g-\frac{A}{6}+2} \frac{ \mathcal{K}_{\mathrm{nK}}(q^{-1/6})}{\zeta_{q^2}(2) (\sqrt{q}-1)}+O\left(q^{\frac{5g}{6}+\varepsilon g}\right) \\
& +q^{-g/2-1} \frac{1}{2 \pi i} \oint_{|u|=q^{-2\sigma}} \sum_{f \in \mathcal{M}_{q,\leq g-A-1}} \frac{1}{q^{\deg(f)/2}} \sum_{\substack{D \in \mathbb{F}_q[T] \\ \deg(D) \leq g/2+1 \\ (D,f)=1}} \mu(D) G_{q^2}(f,D) \frac{\Tilde{\Psi}_{q^2}(fD,u)}{u^{g/2+1-\deg(D)}} \, \frac{du}{u}.
\end{align*} 
We treat $S_{12,\mathrm{dual}}$ similarly and since $\deg(f) = g-A$ we have $[g/1+1+ \deg(f_1)]_3 = 1$. Then as before $\rho(1,1)=\tau(\chi_3) = q^3$, and
we get that
\begin{align*}
S_{12,\mathrm{dual}} =&q^{g-\frac{A}{6}+2} \frac{\mathcal{K}_{\mathrm{nK}}(q^{-1/6})}{\zeta_{q^2}(2)(1-\sqrt{q})}+O\left(q^{\frac{5g}{6}+\varepsilon g} \right) \\
&+  \frac{ q^{-\frac{g}{2}-1}}{1-\sqrt{q}} \frac{1}{2 \pi i} \oint_{|u|=q^{-2\sigma}} \sum_{f \in \mathcal{M}_{q,g-A}} \frac{1}{q^{\deg(f)/2}} \sum_{\substack{D \in \mathbb{F}_q[T] \\ \deg(D) \leq g/2+1 \\ (D,f)=1}} \mu(D) G_{q^2}(f,D) \frac{\Tilde{\Psi}_{q^2}(fD,u)}{u^{g/2+1-\deg(D)}} \, \frac{du}{u}.
\end{align*}
Combining the two equations above, we get that
\begin{align}
S_{1,\mathrm{dual}} =&- \frac{q^{g-\frac{A}{6}+2} \mathcal{K}_{\mathrm{nK}}(q^{-1/6})\zeta_q(1/2)}{\zeta_{q^2}(2)} +O\left(q^{\frac{5g}{6}+\varepsilon g}\right) \label{final-dual} \\
&+ q^{-g/2-1} \frac{1}{2 \pi i} \oint_{|u|=q^{-2\sigma}} \sum_{f \in \mathcal{M}_{q,\leq g-A-1}} \frac{1}{q^{\deg(f)/2}} \sum_{\substack{D \in \mathbb{F}_q[T] \\ \deg(D) \leq g/2+1 \\ (D,f)=1}} \mu(D) G_{q^2}(f,D) \frac{\Tilde{\Psi}_{q^2}(fD,u)}{u^{g/2+1-\deg(D)}} \, \frac{du}{u} \nonumber \\
&+  \frac{ q^{-\frac{g}{2}-1}}{1-\sqrt{q}} \frac{1}{2 \pi i} \oint_{|u|=q^{-2\sigma}} \sum_{f \in \mathcal{M}_{q,g-A}} \frac{1}{q^{\deg(f)/2}} \sum_{\substack{D \in \mathbb{F}_q[T] \\ \deg(D) \leq g/2+1 \\ (D,f)=1}} \mu(D) G_{q^2}(f,D) \frac{\Tilde{\Psi}_{q^2}(fD,u)}{u^{g/2+1-\deg(D)}} \, \frac{du}{u}. \nonumber
\end{align}
We have
$$\mathcal{J}_{\mathrm{nK}}(1) =  \prod_{\substack{R \in \mathbb{F}_q[T] \\ \deg(R)  \,\mathrm{ odd}}} \left[1- \frac{1}{|R|_q^2+1} \right]  \prod_{\substack{R \in \mathbb{F}_q[T] \\ \deg(R)  \,\mathrm{ even}}} \left[ 1- \frac{1}{(|R|_q+1)^2} \right],$$ and using the definition \eqref{b-def} for $\mathcal{B}_{\mathrm{nK}}(u,w)$ 
\begin{align*}
\mathcal{B}_{\mathrm{nK}}(q^{-1/6},1) =& \prod_{\substack{R \in \mathbb{F}_q[T] \\ \deg(R)  \,\mathrm{ odd}}} \left [ 1+ \frac{1}{|R|_q(1-\frac{1}{|R|_q^2})} + \frac{1}{|R|_q^{1/2}(|R|_q^{3/2}-\frac{1}{|R|_q^{1/2}})} \right ] \left [ 1- \frac{1}{|R|_q} \right ] \\
& \times  \prod_{\substack{R \in \mathbb{F}_q[T] \\ \deg(R)  \,\mathrm{ even}}} \left [ 1+\frac{1}{1+\frac{2}{|R|_q}} \left ( \frac{1}{|R|_q(1-\frac{1}{|R|_q^2})} + \frac{1}{|R|_q^{1/2}(|R|_q^{3/2}-\frac{1}{|R|_q^{1/2}})} \right ) \right ] \left [ 1- \frac{1}{|R|_q} \right ] \\
=& \prod_{\substack{R \in \mathbb{F}_q[T] \\ \deg(R)  \,\mathrm{ even}}} \left [1+ \frac{1}{(1+ \frac{2}{|R|_q})(|R|_q-1)} \right ]\left [ 1- \frac{1}{|R|_q} \right ] .
\end{align*}
By \eqref{f-def},
\begin{align*} \label{f-euler}
\mathcal{K}_{\mathrm{nK}}(q^{-1/6}) =  \prod_{\substack{R \in \mathbb{F}_q[T] \\ \deg(R)  \,\mathrm{ odd}}}  \left (1- \frac{1}{|R|_q^2+1} \right ) \prod_{\substack{R \in \mathbb{F}_q[T] \\ \deg(R)  \,\mathrm{ even}}} \left (1- \frac{3}{(|R|_q+1)^2} \right ),
\end{align*}
and we have that $\mathcal{K}_{\mathrm{nK}}(q^{-1/6}) = \mathcal{A}_{\mathrm{nK}} (1/q^2, 1/q)$. Since $\zeta_q(3)= \zeta_{q^2}(2)$,  
by using equation \eqref{final-dual} and Lemma \ref{mt_expression_nk} we note that the corresponding terms of size $q^{g-\frac{A}{6}}$ in the expressions for 
$S_{1, \cube}$ and $S_{1, \mathrm{dual}}$
cancel out. Hence
\begin{align*}
S_{1, \cube}  + S_{1, \mathrm{dual}} =& \frac{q^{g+2} \zeta_q(3/2)}{\zeta_q(3)} \mathcal{A}_{\mathrm{nK}} \left ( \frac{1}{q^2}, \frac{1}{q^{3/2}} \right ) +O(q^{g-\frac{A}{2}+\varepsilon g} +q^{\frac{5g}{6}+\varepsilon g}) \\
&+ q^{-\frac{g}{2}-1} \frac{1}{2 \pi i} \oint_{|u|=q^{-2\sigma}} \sum_{f \in \mathcal{M}_{q,\leq g-A-1}} \frac{1}{q^{\deg(f)/2}} \sum_{\substack{D \in \mathbb{F}_q[T] \\ \deg(D) \leq g/2+1 \\ (D,f)=1}} \mu(D) G_{q^2}(f,D) \frac{\Tilde{\Psi}_{q^2}(fD,u)}{u^{g/2+1-\deg(D)}} \, \frac{du}{u} \\
&+  \frac{ q^{-\frac{g}{2}-1}}{1-\sqrt{q}} \frac{1}{2 \pi i} \oint_{|u|=q^{-2\sigma}} \sum_{f \in \mathcal{M}_{q,g-A}} \frac{1}{q^{\deg(f)/2}} \sum_{\substack{D \in \mathbb{F}_q[T] \\ \deg(D) \leq g/2+1 \\ (D,f)=1}} \mu(D) G_{q^2}(f,D) \frac{\Tilde{\Psi}_{q^2}(fD,u)}{u^{g/2+1-\deg(D)}} \, \frac{du}{u}. 
\end{align*}

Now we consider the integral terms above. Note that it is enough to bound the first one. Using Lemma \ref{our-psi}, the term in the second line above is bounded by 
\begin{align*} 
\ll &q^{-\frac{g}{2}}\sum_{f \in \mathcal{M}_{q,\leq g-A-1}}\frac{1}{q^{\deg(f)/2}} \sum_{\deg(D)\leq g/2+1} q^{\deg(D)} q^{\sigma g-3 \sigma \deg(D)+\frac{3}{2}\deg(D)+\deg(f) ( \tfrac{3}{2}-\sigma)}\\
\ll & g q^{\frac{3}{2}g -(2-\sigma)A}
\end{align*}
as long as $\sigma \geq 7/6$. Then
\begin{equation} \label{cube+dual}
S_{1,\cube}+S_{1,\mathrm{dual}} =  \frac{q^{g+2} \zeta_q(3/2)}{\zeta_q(3)} \mathcal{A}_{\mathrm{nK}} \left ( \frac{1}{q^2}, \frac{1}{q^{3/2}} \right ) +O \left (q^{g-\frac{A}{2}+\varepsilon g} +q^{\frac{5g}{6}+\varepsilon g} +q^{\frac{3g}{2}-(2-\sigma)A+\varepsilon g} \right ),
\end{equation}
which finishes the proof of Lemma \ref{lemma_cube+dual}.
\begin{rem}\label{remark56} Note that the error term of size $q^{\frac{5g}{6}}$ can be computed explicitly from equation \eqref{main_dual} by evaluating the residue when $u^3=1$. The other error terms will eventually dominate the term of size $q^{\frac{5g}{6}}$, so we do not carry out the computation. However, we believe this term will persist in the asymptotic formula.
 \end{rem}
 
\subsection{The proof of Theorem \ref{thm-non-Kummer}}
Using Lemmas \ref{lemma_cube+dual} and \ref{non-cubes-2}, we get that
$$  \sum_{\substack{F \in \mathcal{H}_{q^2,\frac{g}{2}+1} \\ P \mid F \Rightarrow P \not\in \F_q[T]}} L_q \left( \frac{1}{2}, \chi_F \right )=  \frac{ q^{g+2} \zeta_q(3/2)}{\zeta_q(3)}  \mathcal{A}_{\mathrm{nK}} \left (\frac{1}{q^2}, \frac{1}{q^{3/2}} \right ) + O \left (q^{\frac{g+A}{2}+\varepsilon g}+q^{\frac{5g}{6}+\varepsilon g} + q^{\frac{3g}{2}-(2-\sigma)A+\varepsilon g} \right ),$$ where $7/6 \leq \sigma<4/3$.
Picking $\sigma= 7/6$ and $A= 3 [g/4]$ finishes the proof of Theorem \ref{thm-non-Kummer}.

\section{The Kummer setting}

We now assume that $q$ is odd with $q \equiv 1 \pmod 3$.
 We will prove Theorem \ref{first-moment-Kummer}.

\subsection{Setup and sieving}

By Lemma \ref{kummer-desc}, we want to compute
\begin{align}
 \sum_{\substack{d_1+d_2 = g+1\\ d_1 + 2 d_2 \equiv 1 \pmod 3}} \sum_{\substack{F_1 \in \mathcal{H}_{q,d_1} \\ F_2 \in \mathcal{H}_{q,d_2} \\ (F_1, F_2) = 1}} L_q \left( \frac{1}{2}, \chi_{F_1 F_2^2} \right)  = S_{2,\rm principal} + S_{2,\rm dual} ,
\label{first-mom}
\end{align}
 where we have from Proposition \ref{prop-AFE} and Lemma \ref{cubic-reciprocity} (cubic reciprocity)
 \begin{align}
 S_{2,\rm principal} =&\sum_{\substack{d_1+d_2 = g+1\\ d_1 + 2 d_2 \equiv 1 \pmod 3}} \sum_{\substack{F_1 \in \mathcal{H}_{q,d_1} \\ F_2 \in \mathcal{H}_{q,d_2} \\ (F_1, F_2) = 1}} \sum_{f \in \mathcal{M}_{q,\leq A}} \frac{\chi_{f}(F_1) \overline{\chi_{f}}(F_2)}{|f|_q^{1/2}} \label{principal-term}, \\
 S_{2,\rm dual} =&\sum_{\substack{d_1+d_2 = g+1\\ d_1 + 2 d_2 \equiv 1 \pmod 3}} \sum_{\substack{F_1 \in \mathcal{H}_{q,d_1} \\ F_2 \in \mathcal{H}_{q,d_2} \\ (F_1, F_2) = 1}} \omega(\chi_{F_1}\overline{\chi_{F_2}}) \sum_{f \in \mathcal{M}_{q,\leq g-A-1}} \frac{\overline{\chi_{f}}(F_1) {\chi_{f}}(F_2)}{|f|_q^{1/2}} .\label{dual-term}
 \end{align}
 We will choose $A \equiv 0 \pmod 3$.
 For the principal term, we will compute the contribution from cube polynomials $f$ and bound the contribution from non-cubes. We write
 \begin{equation*} \label{sum-principal}
 S_{2,\rm principal} = S_{2, \cube}+S_{2, \neq \cube},
 \end{equation*}
 where 
 \begin{equation} \label{contrib-cubes}
 S_{2,\cube} = \sum_{\substack{d_1+d_2 = g+1\\ d_1 + 2 d_2 \equiv 1 \pmod 3}} \sum_{\substack{F_1 \in \mathcal{H}_{q,d_1} \\ F_2 \in \mathcal{H}_{q,d_2} \\ (F_1, F_2) = 1}} \sum_{\substack{f \in \mathcal{M}_{q,\leq A} \\ f = \tinycube}} \frac{\chi_{f}(F_1) \overline{\chi_{f}}(F_2)}{|f|_q^{1/2}},
 \end{equation}
 and 
 \begin{equation} 
 \label{contrib-non-cubes}
S_{2,\neq \cube} = \sum_{\substack{d_1+d_2 = g+1\\ d_1 + 2 d_2 \equiv 1 \pmod 3}} \sum_{\substack{F_1 \in \mathcal{H}_{q,d_1} \\ F_2 \in \mathcal{H}_{q,d_2} \\ (F_1, F_2) = 1}} \sum_{\substack{f \in \mathcal{M}_{q,\leq A} \\ f \neq \tinycube}} \frac{\chi_{f}(F_1) \overline{\chi_{f}}(F_2)}{|f|_q^{1/2}}.
 \end{equation}
The main results used to prove Theorem \ref{first-moment-Kummer} are summarized in the following  lemmas whose proofs we postpone to the next sections.
\begin{lem} \label{lemma-cubes}
The main term $S_{2,\cube}$ is given by the following asymptotic formula
$$ S_{2,\cube} = C_{\mathrm{K},1} g q^{g+1} + C_{\mathrm{K},2} q^{g+1} + D_{\mathrm{K},1} g q^{g+1-\frac{A}{6}}+D_{\mathrm{K},2} q^{g+1- \frac{A}{6}}+ O \left (q^{\frac{g}{3}+\varepsilon g}+q^{g- \frac{5A}{6}+\varepsilon g} \right ),$$
for some explicit constants $C_{\mathrm{K},1}, C_{\mathrm{K},2}, D_{\mathrm{K},1}, D_{\mathrm{K},2}$ (see formula \eqref{main-computation}).
\end{lem}
 We also have the following upper bounds for $S_{2, \neq \cube}$ and $S_{2,\mathrm{dual}}$.
\begin{lem} \label{lemma-non-cubes}
We have that
$$S_{2, \neq \cube} \ll q^{ \frac{A+g}{2}+ \varepsilon g}.$$
\end{lem}
 \begin{lem}\label{lemma-Kummer3}
 The dual term is bounded by
 $$S_{2,\rm dual} \ll q^{(1+\varepsilon)g -\frac{A}{6}} + q^{ \left(\frac{23}{12}-\frac{\sigma}{2}+\varepsilon\right)g- \left( \frac{13}{12}-\frac{\sigma}{2}\right)A} + q^{\frac{3g}{2}-A(2-\sigma)+\varepsilon g},$$ for $7/6 \leq \sigma <4/3$.
 \end{lem}

 We finish the section by sieving out the values of $F_1$ and $F_2$.
\begin{lem}
\label{sieve}
For $f$ a monic polynomial in $\mathbb{F}_q[T]$ the following holds.
\begin{align*}
\sum_{\substack{F_1\in \mathcal{H}_{q,d_1}\\ F_2\in \mathcal{H}_{q,d_2}\\(F_1,F_2)=1}} \chi_{f}(F_1) \overline{\chi_f}(F_2) =&\sum_{\substack{D_1\in \mathcal{M}_{q,\leq d_1/2}\\D_2\in \mathcal{M}_{q\leq d_2/2}}} \mu(D_1) \mu(D_2)\chi_{f}(D_1^2D_2)\\
&\times \sum_{\substack{\deg(H)\leq \min\{d_1-\deg(D_1),d_2-\deg(D_2)\}\\\deg(H)-\deg(D_1,H)\leq d_1-2\deg(D_1)\\\deg(H)-\deg(D_2,H)\leq d_2-2\deg(D_2)\\(H,f)=1 }}\mu(H)\chi_f((D_1,H)^2(D_2,H))\\&\times\sum_{\substack{L_1\in \mathcal{M}_{q, d_1-2\deg(D_1)-\deg(H)+\deg(D_1,H)}\\ L_2\in \mathcal{M}_{q, d_2-2\deg(D_2)-\deg(H)+\deg(D_2,H)}}} \chi_{f}(L_1) \overline{\chi_{f}}(L_2).\\
 \end{align*}

\end{lem}
\begin{proof}
 We have that
 \begin{align*}
\sum_{\substack{F_1\in \mathcal{H}_{q,d_1}\\ F_2\in \mathcal{H}_{q,d_2}\\(F_1,F_2)=1}} \chi_{f}(F_1) \overline{\chi_f}(F_2)=&\sum_{\substack{D_1\in \mathcal{M}_{q,\leq d_1/2}\\D_2\in \mathcal{M}_{q,\leq d_2/2}}} \mu(D_1) \mu(D_2) \chi_{f}(D_1^2 D_2) \sum_{\substack{F'_1\in \mathcal{M}_{q,d_1-2\deg(D_1)}\\ F_2'\in \mathcal{M}_{q,d_2-2\deg(D_2)}\\(D_1F_1',D_2F_2')=1}} \chi_{f}(F'_1)\overline{\chi_{f}}(F_2')\\
 =&\sum_{\substack{D_1\in \mathcal{M}_{q,\leq d_1/2}\\D_2\in \mathcal{M}_{q,\leq d_2/2}}} \mu(D_1) \mu(D_2)\chi_{f}(D_1^2D_2)\sum_{H\in \mathcal{M}_{q,\leq \min\{d_1-\deg(D_1),d_2-\deg(D_2)\}}}\mu(H)\\&\times\sum_{\substack{F_1'\in \mathcal{M}_{q,d_1-2\deg(D_1)}\\ F_2'\in \mathcal{M}_{q,d_2-2\deg(D_2)}\\H\mid (D_1F_1',D_2F_2')}} \chi_{f}(F_1') \overline{ \chi_{f}}(F_2')\\
 \end{align*}
We remark that $H\mid (D_1F_1',D_2F_2')$ is equivalent to $H_1=\frac{H}{(D_1,H)}\mid F_1'$ and $H_2=\frac{H}{(D_2,H)}\mid F_2'$. This gives
\begin{align*}
\sum_{\substack{F_1\in \mathcal{H}_{q,d_1}\\ F_2\in \mathcal{H}_{q,d_2}\\(F_1,F_2)=1}} & \chi_{f}(F_1) \overline{\chi_{f}}(F_2) =\sum_{\substack{D_1\in \mathcal{M}_{q,\leq d_1/2}\\D_2\in \mathcal{M}_{q,\leq d_2/2}}} \mu(D_1) \mu(D_2)\chi_{f}(D_1^2D_2)  \\
& \times \sum_{\substack{\deg(H)\leq \min\{d_1-\deg(D_1),d_2-\deg(D_2)\}\\\deg(H_1)\leq d_1-2\deg(D_1)\\\deg(H_2)\leq d_2-2\deg(D_2) }}\mu(H)\chi_f(H_1H_2^2) \sum_{\substack{F_1''\in \mathcal{M}_{q,d_1-2\deg(D_1)-\deg(H_1)}\\ F_2''\in \mathcal{M}_{q,d_2-2\deg(D_2)-\deg(H_2)}}} \chi_{f}(F_1'') \overline{\chi_{f}}(F_2'').
 \end{align*}
 \end{proof}
We rewrite Lemma \ref{sieve} in the following form.

\begin{cor} \label{Alexandra-sieve}
For $f$ a monic polynomial in $\mathbb{F}_q[T]$ the following holds.
\begin{align*}
\sum_{\substack{F_1\in \mathcal{H}_{q,d_1}\\ F_2\in \mathcal{H}_{q,d_2}\\(F_1,F_2)=1}} & \chi_{f}(F_1) \overline{\chi_f}(F_2) = \sum_{\substack{H \in \mathcal{M}_{q,\leq \min\{d_1,d_2\}} \\ (H,f)=1}} \mu(H) \sum_{\substack{ R_1 \in \mathcal{M}_{q, \leq d_1-\deg(H)}\\R_1\mid H }} \mu(R_1) \chi_f(R_1) \\
& \times \sum_{\substack{R_2\in \mathcal{M}_{q,\leq d_2-\deg(H)}\\R_2\mid H}} \mu(R_2) \chi_f(R_2)^2 \sum_{\substack{D_1 \in \mathcal{M}_{q,\leq \frac{d_1-\deg(H)-\deg(R_1)}{2}} \\ (D_1,H)=1}} \mu(D_1) \chi_f(D_1)^2 \\
& \times \sum_{\substack{D_2 \in \mathcal{M}_{q,\leq \frac{d_2-\deg(H)-\deg(R_2)}{2}} \\ (D_2,H)=1}} \mu(D_2) \chi_f(D_2) \sum_{\substack{L_1\in \mathcal{M}_{q,d_1-2\deg(D_1)-\deg(H)-\deg(R_1)}\\ L_2\in \mathcal{M}_{q,d_2-2\deg(D_2)-\deg(H)-\deg(R_2)}}} \chi_{f}(L_1) \overline{\chi_{f}}(L_2).
\end{align*}
\end{cor}
\begin{proof} This follows by taking  $R_i = (D_i,H)$ in Lemma \ref{sieve}. \end{proof}

\subsection{The main term} \label{sec-main}
Here we will obtain an asymptotic formula for the main term \eqref{contrib-cubes} by proving Lemma \ref{lemma-cubes}. Recall that
$$ S_{2,\cube} = \sum_{\substack{d_1+d_2 = g+1\\ d_1 + 2 d_2 \equiv 1 \pmod 3}} \sum_{\substack{F_1 \in \mathcal{H}_{q,d_1} \\ F_2 \in \mathcal{H}_{q,d_2} \\ (F_1, F_2) = 1}} \sum_{\substack{f \in \mathcal{M}_{q,\leq A} \\ f = \tinycube}} \frac{\chi_{f}(F_1) \overline{\chi_{f}}(F_2)}{|f|_q^{1/2}}.$$
Let $2g+1 \equiv a \pmod 3$ and $g \equiv b \pmod 3$ with $a, b \in \{0,1,2 \}$. Notice that then $1+2a \equiv b \pmod 3.$
Recall that $A \equiv 0 \pmod 3$. Since $d_1+d_2=g+1$ and $d_1+2d_2 \equiv 1 \pmod 3$, it follows that $d_1 \equiv a \pmod 3$. In the equation above, write $f= k^3$. Then the main term $S_{2,\cube}$ can be rewritten as
\begin{equation}
S_{2,\cube} = \sum_{\substack{d_1+d_2 = g+1 \\ d_1 \equiv a \pmod 3}} \sum_{\substack{F_1 \in \mathcal{H}_{q,d_1} \\ F_2 \in \mathcal{H}_{q,d_2} \\ (F_1, F_2) = 1}} \sum_{\substack{k \in \mathcal{M}_{q,\leq \frac{A}{3}} \\ (k, F_1 F_2)=1}} \frac{1}{|k|_q^{3/2}}.
\label{mainpr}
\end{equation}
We consider the generating series
\begin{equation} \label{gen-series-c}
\mathcal{C}_{\mathrm{K}}(x,y,u) = \sum_{\substack{F_1 , F_2 \in \mathcal{H}_q \\ (F_1,F_2)=1}} \sum_{\substack{k \in \mathcal{M}_q \\ (k,F_1F_2)=1}} x^{\deg(F_1)} y^{\deg(F_2)} \frac{u^{\deg(k)}}{|k|_q^{3/2}}.
\end{equation}
Note that
\begin{equation} \sum_{\substack{k \in \mathcal{M}_q \\ (k,F_1F_2)=1}} \frac{u^{\deg(k)}}{|k|_q^{3/2}}  = \prod_{P \nmid F_1 F_2} \left ( 1- \frac{u^{\deg(P)}}{|P|_q^{3/2}}\right )^{-1} = \mathcal{Z}_q\left ( \frac{u}{q^{3/2}} \right ) \prod_{P \mid F_1 F_2}  \left ( 1- \frac{u^{\deg(P)}}{|P|_q^{3/2}}\right ).
\label{sum-k2}
\end{equation}

Let $C_{P,K}(u)$ denote the Euler factor above. 
Now we introduce the sum over $F_2$ and  we have that
\begin{align} 
\sum_{\substack{F_2 \in \mathcal{H}_q \\ (F_2, F_1)=1}}  y^{\deg(F_2)} \prod_{P \mid F_2}  C_{P,K}(u)
=&\prod_P \left ( 1+ y^{\deg(P)} C_{P,K}(u) \right ) \prod_{P\mid F_1} \left ( 1+ y^{\deg(P)} C_{P,K}(u) \right )^{-1}. \label{sum-F_2}
\end{align} 

Let $B_{P,K}(y,u)$ be the $P$-factor when $P\mid F_1$.
Finally, introducing the sum over $F_1$ and combining equations \eqref{sum-k2} and \eqref{sum-F_2}, we have that
\begin{equation}  \sum_{F_1 \in \mathcal{H}_q}x^{\deg(F_1)} \prod_{P \mid F_1} C_{P,K}(u) B_{P,K}(y,u)= \prod_P \left (1+ x^{\deg(P)} C_{P,K}(u) B_{P,K}(y,u) \right ) .
\label{sum-F_1}
\end{equation}
Combining equations \eqref{gen-series-c}, \eqref{sum-k2}, \eqref{sum-F_2} and \eqref{sum-F_1} and simplifying, we get that
\begin{align} 
\mathcal{C}_{\mathrm{K}}(x,y,u) =&\mathcal{Z}_q\left ( \frac{u}{q^{3/2}} \right )  \prod_P \left( 1+ (x^{\deg(P)}+y^{\deg(P)}) \left ( 1- \frac{u^{\deg(P)}}{|P|_q^{3/2}}\right ) \right) \nonumber \\
=& \mathcal{Z}_q\left ( \frac{u}{q^{3/2}} \right ) \mathcal{Z}_q(x) \mathcal{Z}_q(y) \mathcal{D}_{\mathrm{K}}(x,y,u),\label{simplified-c}
\end{align}
where
\begin{align} \label{euler-d}
\mathcal{D}_{\mathrm{K}}(x,y,u) =&\prod_P \left ( 1 - x^{2 \deg(P)} - y^{2 \deg(P)} - (xy)^{\deg(P)} + (x^2y)^{\deg(P)} + (y^2 x)^{\deg(P)}  \right. \\
&- \frac{(ux)^{\deg(P)}}{|P|_q^{3/2}} - \frac{ (uy)^{\deg(P)}}{|P|_q^{3/2}} + \frac{ (x^2u)^{\deg(P)}}{|P|_q^{3/2}}  + \frac{ (y^2u)^{\deg(P)}}{|P|_q^{3/2}}+ \frac{2 (xyu)^{\deg(P)}} {|P|_q^{3/2}}   \nonumber \\
&-\left. \frac{(x^2yu)^{\deg(P)}}{|P|_q^{3/2}} - \frac{ (y^2xu)^{\deg(P)}}{|P|_q^{3/2}}\right). \nonumber
\end{align}
Note that $\mathcal{D}_{\mathrm{K}}(x,y,u)$ has an analytic continuation when $|x|<1, |y|<1, |x^2y| < \frac{1}{q}, |y^2 x| < \frac{1}{q} , |x u| <q^{3/2}, | y u | <q^{3/2}, |x^2u | < \sqrt{q}, |y^2 u|<\sqrt{q}, |xyu|< \sqrt{q}.$
Using equation \eqref{simplified-c} and Perron's formula  (Lemma \ref{perron})  three times in equation \eqref{mainpr}, we get that
\begin{align*}
S_{2,\cube} =&\sum_{\substack{d_1+d_2=g+1 \\ d_1 \equiv a \pmod 3}} \frac{1}{ (2 \pi i)^3} \oint \oint \oint \frac{ \mathcal{D}_{\mathrm{K}}(x,y,u)}{ ( 1 - \frac{u}{\sqrt{q}})  (1-qx) (1-qy) (1-u) u^{A/3} y^{d_2} x^{d_1}} \, \frac{du}{u} \, \frac{dy}{y} \, \frac{dx}{x},
\end{align*} 
where we initially integrate along circles around the origin of radii $|u| = \frac{1}{q^{\varepsilon}}, |x|=|y| = \frac{1}{q^{1+\varepsilon}}.$ 

We first shift the contour over $u$ to $|u| = q^{5/2}$, and encounter two poles: one at $u= 1$ and another at $u=\sqrt{q}$. We compute the residues of the poles and then
\begin{align*}
\oint_{|u| = \frac{1}{q^{\varepsilon}}}  \frac{ \mathcal{D}_{\mathrm{K}}(x,y,u)}{ ( 1 - \frac{u}{\sqrt{q}}) (1-u) u^{A/3}} \, \frac{du}{u} =&\zeta_q(3/2) \mathcal{D}_{\mathrm{K}}(x,y,1) + q^{-\frac{A}{6}} \zeta_q(1/2) \mathcal{D}_{\mathrm{K}}(x,y,\sqrt{q}) \\
&+ \oint_{|u|=q^{5/2}}   \frac{ \mathcal{D}_{\mathrm{K}}(x,y,u)}{ ( 1 - \frac{u}{\sqrt{q}}) (1-u) u^{A/3}} \, \frac{du}{u}.
\end{align*}

Plugging this into the expression for $S_{2,\cube}$ and bounding the new triple integral by $q^{g- \frac{5A}{6} + \varepsilon g}$ give
\begin{align}
S_{2,\cube} =&\zeta_q(3/2) \sum_{\substack{d_1+d_2=g+1 \\ d_1 \equiv a \pmod 3}} \frac{1}{ (2 \pi i)^2} \oint_{|x| = \frac{1}{q^{1+\varepsilon}}} \oint_{|y| = \frac{1}{q^{1+\varepsilon}}} \frac{\mathcal{D}_{\mathrm{K}}(x,y,1)}{ (1-qx)(1-qy) y^{d_2} x^{d_1}} \, \frac{dy}{y} \, \frac{dx}{x} \label{i1} \\
&+ q^{-\frac{A}{6}} \zeta_q(1/2)\sum_{\substack{d_1+d_2=g+1 \\ d_1 \equiv a \pmod 3}} \frac{1}{ (2 \pi i)^2}\oint_{|x| = \frac{1}{q^{1+\varepsilon}}} \oint_{|y| = \frac{1}{q^{1+\varepsilon}}} \frac{\mathcal{D}_{\mathrm{K}}(x,y,\sqrt{q})}{ (1-qx)(1-qy) y^{d_2} x^{d_1}} \, \frac{dy}{y} \, \frac{dx}{x}  \label{i2} \\
&+ O(q^{g- \frac{5A}{6} +\varepsilon g}). \nonumber
\end{align}
We first focus on the first term \eqref{i1}. Note that $\mathcal{D}_{\mathrm{K}}(x,y,1)$ has an analytic continuation for $|x|<1 , |y|<1, |x^2y| < \frac{1}{q}, |y^2 x|< \frac{1}{q}.$

{We remark that in \eqref{i1} we can shift the contours of integration to the smaller circles  $|x| = q^{-3}$ and $|y|=q^{-2}$ without changing the value of the integral as we are not crossing any pole. }

We write $d_1=3k+a$ and compute the sum over $d_1$. Note that $k \leq [(g+1-a)/3]=[g/3]$. Then
\begin{equation*}
\eqref{i1} =\zeta_q(3/2) \frac{1}{(2 \pi i)^2} \oint_{|x|=q^{-3}} \oint_{|y|=q^{-2}} \frac{\mathcal{D}_{\mathrm{K}}(x,y,1)}{(1-qx)(1-qy)(y^3-x^3)} \left[ \frac{y^{2+a-b}}{x^{g+a-b}}-\frac{x^{3-a}}{y^{g+1-a}} \right] \, \frac{dy}{y} \, \frac{dx}{x}. 
\end{equation*} 
We write the integral above as a difference of two integrals.
Note that the second double integral vanishes, because the integrand for the integral over $x$ has no poles inside the circle $|x|=q^{-3}$.

 Hence
\begin{align*}
\eqref{i1} =&\zeta_q(3/2) \frac{1}{(2 \pi i)^2} \oint_{|x|=q^{-3}} \oint_{|y|=q^{-2}} \frac{\mathcal{D}_{\mathrm{K}}(x,y,1)y^{2+a-b}}{(1-qx)(1-qy)(y^3-x^3)x^{g+a-b}} \, \frac{dy}{y} \, \frac{dx}{x}.
\end{align*} 
{Note that for the integral over $y$, the only poles of the integrand inside the circle $|y|=q^{-2}$  are at $y^3=x^3$, so when $y = x \xi_3^i$ for $i \in \{0,1,2 \}$ and $\xi_3=e^{2\pi i/3}$. }
Hence
\begin{align*}
 \frac{1}{2 \pi i}  \oint_{|y|=q^{-2}} \frac{\mathcal{D}_{\mathrm{K}}(x,y,1)y^{2+a-b}}{x^{g+a-b}(1-qy)(y^3-x^3)} \, \frac{dy}{y} =&\frac{1}{3x^{g+1}} \left[ \frac{ \mathcal{D}_{\mathrm{K}}(x,x,1)}{1-qx} + \frac{\mathcal{D}_{\mathrm{K}}(x,\xi_3 x,1)\xi_3^{2+a-b}}{1-q \xi_3 x}\right.\\
 &\left. + \frac{ \mathcal{D}_{\mathrm{K}}(x, \xi_3^2 x,1)\xi_3^{2(2+a-b)}}{1-q \xi_3^2 x}  \right].
\end{align*} 
To compute the integral over $x$, we shift the  the contour of integration to $|x|=q^{-1/3+\varepsilon}$, evaluating the residues at $x=q^{-1}$ corresponding to each of the three functions above. Notice that the first integral has a double pole at $s=1/q$. This gives
\begin{align*}
\eqref{i1} =&\zeta_q(3/2) \frac{1}{2 \pi i} \oint_{|x|=q^{-3}} \frac{1}{3 (1-qx) x^{g+1}}  \left[ \frac{ \mathcal{D}_{\mathrm{K}}(x,x,1)}{1-qx} + \frac{\mathcal{D}_{\mathrm{K}}(x,\xi_3 x,1) \xi_3^{2+a-b}}{1-q \xi_3 x}\right. \\
&\left. + \frac{ \mathcal{D}_{\mathrm{K}}(x, \xi_3^2 x,1) \xi_3^{2(2+a-b)}}{1-q \xi_3^2 x}  \right] \, \frac{dx}{x}\\
=& \frac{\zeta_q(3/2)}{3} \left[ (g+2)q^{g+1} \mathcal{D}_{\mathrm{K}} ( \tfrac{1}{q} , \tfrac{1}{q} ,1) - q^{g} \frac{d}{dx} \mathcal{D}_{\mathrm{K}} (x,x,1) |_{x=1/q} \right. \\
&+q^{g+1} \frac{ \mathcal{D}_{\mathrm{K}} (\tfrac{1}{q}, \tfrac{\xi_3}{q},1) \xi_3^{1+2a}}{1-\xi_3}+q^{g+1} \frac{ \mathcal{D}_{\mathrm{K}} ( \tfrac{\xi_3^2}{q}, \tfrac{1}{q},1)\xi_3^a}{1-\xi_3^2} \\
&\left. + q^{g+1} \frac{\mathcal{D}_{\mathrm{K}} ( \tfrac{1}{q} , \tfrac{\xi_3^2}{q} ,1)\xi_3^{2+a}}{1-\xi_3^2}  + q^{g+1} \frac{\mathcal{D}_{\mathrm{K}} ( \tfrac{\xi_3}{q}, \tfrac{1}{q},1 ) \xi_3^{2a}}{1-\xi_3}\right] + O(q^{\tfrac{g}{3} + \varepsilon g}),
\end{align*}
where
we have used the fact that $1+2a \equiv b \pmod{3}$. 

Since $\mathcal{D}_{\mathrm{K}}(x,y,1) = \mathcal{D}_{\mathrm{K}}(y,x,1)$, we further simplify \eqref{i1} to
\begin{align*}
\eqref{i1} =&\zeta_q(3/2) \frac{q^{g+1}}{3} \left[ (g+2) \mathcal{D}_{\mathrm{K}} ( \tfrac{1}{q} , \tfrac{1}{q},1 ) - \frac{1}{q} \frac{d}{dx} \mathcal{D}_{\mathrm{K}} (x,x,1) |_{x=1/q} \right.\\
&\left. - \frac{ \mathcal{D}_{\mathrm{K}} (\tfrac{1}{q}, \tfrac{\xi_3}{q},1) \xi_3^{2a+2}}{1-\xi_3} - \frac{ \mathcal{D}_{\mathrm{K}} ( \tfrac{\xi_3^2}{q}, \tfrac{1}{q},1) \xi_3^{a+1}}{1-\xi_3^2} \right] + O(q^{\tfrac{g}{3} + \varepsilon g})\\
=& C_{\mathrm{K},1} g q^{g+1}+ C_{\mathrm{K},2} q^{g+1} +O(q^{\frac{g}{3}+\varepsilon g}),
\end{align*} where
\begin{align} \label{c-1}
C_{\mathrm{K},1} =&\zeta_q(3/2) \frac{  \mathcal{D}_{\mathrm{K}} ( \tfrac{1}{q} , \tfrac{1}{q} ,1) }{3},\\
C_{\mathrm{K},2} =&\zeta_q(3/2) \left[  \frac{2  \mathcal{D}_{\mathrm{K}} ( \tfrac{1}{q} , \tfrac{1}{q},1 ) }{3}- \frac{1}{3q} \frac{d}{dx} \mathcal{D}_{\mathrm{K}} (x,x,1) |_{x=1/q} - \frac{ \mathcal{D}_{\mathrm{K}} (\tfrac{1}{q}, \tfrac{\xi_3}{q},1) \xi_3^{g+1}}{3(1-\xi_3)}- \frac{ \mathcal{D}_{\mathrm{K}} ( \tfrac{\xi_3^2}{q}, \tfrac{1}{q},1) \xi_3^{2g+2}}{3(1-\xi_3^2)} \right] \label{c-2},
\end{align} where we used the fact that $2g+1 \equiv a \pmod 3.$
We remark that the constants above are real, which reflects the fact that the sum is a real number.

We similarly compute the term \eqref{i2} and we get that
\begin{align*}
\eqref{i2} =& \frac{ \zeta_q(1/2) q^{g+1-\frac{A}{6}}}{3} \left[ (g+2) \mathcal{D}_{\mathrm{K}} ( \tfrac{1}{q} , \tfrac{1}{q},\sqrt{q} ) - \frac{1}{q} \frac{d}{dx} \mathcal{D}_{\mathrm{K}} (x,x,\sqrt{q}) |_{x=1/q} - \frac{ \mathcal{D}_{\mathrm{K}} (\tfrac{1}{q}, \tfrac{\xi_3}{q},\sqrt{q}) \xi_3^{2a+2}}{1-\xi_3} \right.\\
&\left.- \frac{ \mathcal{D}_{\mathrm{K}} ( \tfrac{\xi_3^2}{q}, \tfrac{1}{q},\sqrt{q}) \xi_3^{a+1}}{1-\xi_3^2} \right] + O(q^{\tfrac{g}{3}-\frac{A}{6} + \varepsilon g}).
\end{align*}
Putting everything together, we get that 
\begin{align} 
S_{2,\cube} =& \zeta_q(3/2) \frac{q^{g+1}}{3} \Bigg[ (g+2) \mathcal{D}_{\mathrm{K}} ( \tfrac{1}{q} , \tfrac{1}{q},1 ) - \frac{1}{q} \frac{d}{dx} \mathcal{D}_{\mathrm{K}} (x,x,1) |_{x=1/q} - \frac{ \mathcal{D}_{\mathrm{K}} (\tfrac{1}{q}, \tfrac{\xi_3}{q},1) \xi_3^{g+1}}{1-\xi_3} \nonumber \\
&- \frac{ \mathcal{D}_{\mathrm{K}} ( \tfrac{\xi_3^2}{q}, \tfrac{1}{q},1) \xi_3^{2g+2}}{1-\xi_3^2} \Bigg] + \frac{ \zeta_q(1/2) q^{g+1-\frac{A}{6}}}{3} \Bigg[ (g+2) \mathcal{D}_{\mathrm{K}} ( \tfrac{1}{q} , \tfrac{1}{q},\sqrt{q} ) - \frac{1}{q} \frac{d}{dx} \mathcal{D}_{\mathrm{K}} (x,x,\sqrt{q}) |_{x=1/q} \nonumber\\
&- \frac{ \mathcal{D}_{\mathrm{K}} (\tfrac{1}{q}, \tfrac{\xi_3}{q},\sqrt{q}) \xi_3^{g+1}}{1-\xi_3}   - \frac{ \mathcal{D}_{\mathrm{K}} ( \tfrac{\xi_3^2}{q}, \tfrac{1}{q},\sqrt{q}) \xi_3^{2g+2}}{1-\xi_3^2} \Bigg] + O\Big(q^{\tfrac{g}{3} + \varepsilon g}\Big)+O\Big(q^{g- \frac{5A}{6}+\varepsilon g} \Big). \label{main-computation}
\end{align}

\subsection{The contribution from non-cubes} \label{sec:lindelof}

Here we will prove Lemma \ref{lemma-non-cubes}. Recall the definition \eqref{contrib-non-cubes} of $S_{2, \neq \cube}$, the term coming from the contribution of non-cube polynomials.

Using the sieve of Corollary \ref{Alexandra-sieve}, we rewrite $S_{2, \neq \cube}$ as
\begin{align*}
S_{2, \neq \cube} =&\sum_{\substack{d_1 + d_2 = g+1 \\ d_1 + 2 d_2 \equiv 1 \pmod 3}}\sum_{\substack{f \in \mathcal{M}_{q,\leq A} \\ f \neq \tinycube}} {\frac{1}{|f|_q^{1/2}}}
\sum_{\substack{H \in \mathcal{M}_{q,\leq \min\{d_1,d_2\}} \\ (H,f)=1}} \mu(H) \sum_{\substack{R_1 \mid H \\ \deg(R_1) \leq d_1-\deg(H)}} \mu(R_1) \chi_f(R_1)\\
& \times  \sum_{\substack{R_2 \mid H \\ \deg(R_2) \leq d_2-\deg(H)}} \mu(R_2) \chi_f(R_2)^2 \sum_{\substack{D_1 \in \mathcal{M}_{q,\leq \frac{d_1-\deg(H)-\deg(R_1)}{2}} \\ (D_1,R_1)=1}} \mu(D_1) \chi_f(D_1)^2 \\
& \times  \sum_{\substack{D_2 \in \mathcal{M}_{q,\leq \frac{d_2-\deg(H)-\deg(R_2)}{2}} \\ (D_2,R_2)=1}} \mu(D_2) \chi_f(D_2)  \sum_{\substack{L_1\in \mathcal{M}_{d_1-2\deg(D_1)-\deg(H)-\deg(R_1)}\\ L_2\in \mathcal{M}_{q,d_2-2\deg(D_2)-\deg(H)-\deg(R_2)}}} \chi_{f}(L_1) \overline{\chi_{f}}(L_2).
\end{align*}
We write $S_{2, \neq \cube} = S_{2, \neq \cube}^{(0)}+S_{2, \neq \cube}^{(1)}+S_{2, \neq \cube}^{(2)},$ according to $\deg(f) \pmod 3$. When $\deg(f) \equiv 1, 2 \pmod 3$, note that the condition that $f$ is not a cube is automatically satisfied. We will bound $S_{2, \neq \cube}^{(1)}$. Bounding the other two terms is similar (see the remark at the end of the proof).
We begin by using  the Poisson summation formula (Proposition \ref{prop-Poisson}) for the sums over $L_1$ and $L_2$ above. Let $\deg(f) =n$. Note that since $|\epsilon(\chi_f)|=1$, we have that
\begin{align}
S_{2, \neq \cube}^{(1)} =&q^{g+2}  \sum_{\substack{d_1 + d_2 = g+1 \\ d_1 + 2 d_2 \equiv 1 \pmod 3}} \sum_{H \in \mathcal{M}_{q,\leq \min \{d_1,d_2\}}} \frac{\mu(H)}{|H|_q^2} \sum_{\substack{R_1 \mid  H \\ \deg(R_1) \leq d_1 - \deg(H)}} \frac{\mu(R_1)}{|R_1|_q}  \sum_{\substack{R_2 \mid H \\ \deg(R_2) \leq d_2 - \deg(H)}}  \frac{\mu(R_2)}{|R_2|_q} \nonumber  \\
& \times \sum_{\substack{D_1 \in \mathcal{M}_{q,\leq \frac{d_1 - \deg(R_1) - \deg(H)}{2}} \\ (D_1, R_1)=1}} \frac{\mu(D_1)}{|D_1|_q^2}  \sum_{\substack{D_2 \in \mathcal{M}_{q,\leq \frac{d_2 - \deg(R_2) - \deg(H)}{2}} \\ (D_2, R_2)=1}} \frac{\mu(D_2)}{|D_2|_q^2} \nonumber\\
& \times  \sum_{\substack{n=0 \\ n \equiv 1 \pmod 3}}^A q^{-3n/2} \sum_{\substack{f \in \mathcal{M}_{q,n} \\ (f,H)=1}} \chi_f( D_1^2 R_1 D_2 R_2^2)\nonumber \\
& \times \sum_{V_1 \in \mathcal{M}_{q,n -d_1 + 2 \deg(D_1) +\deg(R_1)+\deg(H)-1}}  \sum_{V_2 \in \mathcal{M}_{q,n -d_2 + 2\deg(D_2) +\deg(R_2)+\deg(H)-1}} \frac{G_q(V_1, f) \overline{G_q(V_2, f)}}{|f|_q}. \label{s1-non-cube}
\end{align}
Write $f = E^3 B^2 C$, where $B,C$ are square-free polynomials with $(B,C)=1$. Note that $BC^2 \neq 1$ since $f$ is not a cube. Then the sum over $f$ becomes
\begin{align}
& \sum_{\substack{E \in \mathcal{M}_{q,\leq n/3} \\ (E, D_1R_1D_2R_2H)=1}} \sum_{\substack{B \in \mathcal{H}_{q,\leq \frac{n- 3 \deg(E)}{2}} \\ (B,H)=1}} \sum_{\substack{C \in \mathcal{H}_{q,n-3\deg(E)-2 \deg(B)} \\ (C,BH)=1}} \chi_{B^2C} (D_1^2 R_1 D_2 R_2^2) \sum_{V_1 \in \mathcal{M}_{q,n -d_1 + 2\deg(D_1) +\deg(R_1)+\deg(H)-1}} \nonumber \\
& \sum_{V_2 \in \mathcal{M}_{q,n -d_2 + 2\deg(D_2) +\deg(R_2)+\deg(H)-1}} \frac{G_q(V_1, E^3 B^2C) \overline{G_q(V_2, E^3B^2C)}}{|f|_q}. \label{news}
\end{align}
We remark that for fixed $V_1$ and $V_2$, $G_q(V_1,f) \overline{G_q(V_2,f})$ is multiplicative as a function of $f$. 
Indeed, for $(f,h)=1$, we have by Lemma \ref{gauss} and \eqref{prop-GS}
\begin{align*}
G_q(V_1,f h ) \overline{G_q(V_2,f h)} =&\chi_{f}(h)^2 G_q(V_1,f) G_q(V_1,h) \overline{ \chi_{f}(h)^2 G_q(V_2,f) G_q(V_2,h)} \\
=&G_q(V_1,f)  \overline{ G_q(V_2,f) } G_q(V_1,h) \overline{G_q(V_2,h)}.
\end{align*}
Then,
\begin{align*}
G_q(V_1, E^3 B^2C) \overline{G_q(V_2, E^3B^2C)} =&\prod_{\substack{P \mid E \\ P \nmid BC}} G_q(V_1, P^{3 \text{ord}_P(E)}) \overline{ G_q(V_2, P^{3 \text{ord}_P(E)}) } \\
& \times \prod_{P\mid B} G_q(V_1, P^{3 \text{ord}_P(E)+2}) \overline{ G_q(V_2, P^{3 \text{ord}_P(E)+2}) } \\
& \times \prod_{P\mid C} G_q(V_1, P^{3 \text{ord}_P(E)+1}) \overline{ G_q(V_2, P^{3 \text{ord}_P(E)+1}) }.
\end{align*}
We look at each of the three cases above.
\begin{enumerate}
\item If $P\mid E$ and $P \nmid BC$, from Lemma \ref{gauss} we need $P^{3 \text{ord}_P(E)-1} \mid V_1$ in order for the Gauss sum to be nonzero. In this case, we have
$$ G_q(V_1, P^{3 \text{ord}_P(E)})= \left \{\begin{array}{ll}\phi( P^{3 \text{ord}_P(E)}) & \mbox{if }P^{3 \text{ord}_P(E)} \mid V_1,\\-|P|_q^{3 \text{ord}_P(E)-1}& \mbox{if }P^{3 \text{ord}_P(E)-1} || V_1.\end{array}\right .$$
\item If $P\mid B$ (so $P \nmid C$), again from Lemma \ref{gauss}, we need $P^{3 \text{ord}_P(E)+1} || V_1$, and in this case
$$  G_q(V_1, P^{3 \text{ord}_P(E)+2}) = \epsilon( \chi_{P^2}) \omega(\chi_{P^2}) \chi_{P^2}(V_1 P^{-3 \text{ord}_P(E)-1})^2 |P|_q^{3 \text{ord}_P(E)+ \frac{3}{2}}.$$
\item If $P\mid C$ (so $P \nmid B$), we need $P^{3 \text{ord}_P(E)} ||V_1$. Then
$$ G_q(V_1, P^{3 \text{ord}_P(E)+1}) =\epsilon( \chi_{P}) \omega(\chi_P)  \chi_{P}(V_1 P^{-3 \text{ord}_P(E)})^2  |P|_q^{3 \text{ord}_P(E)+ \frac{1}{2}}.$$
\end{enumerate}
Combining all of the above, it follows that in order to have $G_q(V_1, E^3B^2C) \neq 0$, then we must have  \[V_1 = E^3B V_3\prod_{\substack{P\mid E \\ P\nmid BC}} P^{-1},\] with $(V_3,BC)=1$, and we can write
\begin{align}
G_q(V_1,E^3 B^2 C) =&|E|_q^3 |B|_q^{\frac{3}{2}} |C|_q^{\frac{1}{2}} \prod_{\substack{P\mid E\\P\nmid BC\\P\nmid V_3}}(-|P|_q^{-1})\prod_{\substack{P\mid E\\P\mid V_3}}\left (1-\frac{1}{|P|_q}\right )
\chi_{B^2 C} \Big ( V_3\prod_{\substack{P\mid E\\P\nmid BC}}P^{-1} \Big)^2 \chi_C (B)^2 \nonumber \\
& \times \prod_{P\mid B} \chi_{P}(B/P)^2  \prod_{P\mid B} \epsilon(\chi_{P^2})  \omega(\chi_{P^2}) \prod_{P\mid C}\epsilon(\chi_P) \omega(\chi_P). \label{first-GS}
\end{align}

Similarly we can suppose that
$$V_2= E^3 B V_4\prod_{\substack{P\mid E \\ P\nmid BC}}P^{-1},$$ where we have $(V_4,BC)=1$.
Using equation \eqref{first-GS} and the analogous expression for $G_q(V_2, E^3 B^2 C)$, it follows that
\begin{align*}
G_q(V_1,E^3 B^2 C) \overline{ G_q(V_2, E^3B^2C)} =&|E|_q^6 |B|_q^3 |C|_q \prod_{\substack{P \mid E \\ P \nmid BCV_3}} (- |P|_q^{-1}) \prod_{\substack{P \mid E \\ P \nmid BCV_4}} (- |P|_q^{-1}) \\
 & \times  \prod_{P\mid (E,V_3)} \left ( 1 - \frac{1}{|P|_q} \right )  \prod_{P\mid (E,V_4)} \left ( 1 - \frac{1}{|P|_q} \right ) \overline{\chi_{B^2 C}}( V_3) \chi_{B^2 C} ( V_4).
 \end{align*}

 Then, the expression in equation \eqref{news} is
\begin{align}
& \sum_{\substack{E \in \mathcal{M}_{q,\leq n/3} \\ (E, D_1R_1D_2R_2H)=1}} |E|_q^3 \prod_{P\mid E} |P|_q^{-2}  \sum_{\substack{B \in \mathcal{H}_{q,\leq \frac{n- 3 \deg(E)}{2}} \\ (B,H)=1}}  |B|_q \prod_{P\mid (E,B)} |P|_q^2 \sum_{\substack{C \in \mathcal{H}_{q,n-3\deg(E)-2 \deg(B)} \\ (C,BH)=1}} \prod_{P\mid (E,C)} |P|_q^2  \nonumber \\
& \times \chi_{B^2C} (D_1^2 R_1 D_2 R_2^2)  \sum_{V_3 \in \mathcal{M}_{q,e_1}}  {\overline{\chi_{B^2 C}}} ( V_3) \prod_{P\mid (E,V_3)} (1-|P|_q)  
\sum_{V_4 \in \mathcal{M}_{q,e_2}}\chi_{B^2 C} ( V_4) \prod_{P\mid (E,V_4)} (1-|P|_q) ,\label{v3v4}
\end{align}
where 
\begin{equation*} \label{def-e_i}
e_i = \deg(B)+\deg(C) -d_i + 2\deg(D_i) +\deg(R_i)+\deg(H)-1 + \deg \Big( \prod_{\substack{ P\mid E\\ P \nmid BC}}P\Big)
\end{equation*} for $i=1,2$.
Now we look at the generating series for the sum over $V_3$ and get that
\begin{align*}
\sum_{V_3 \in \mathcal{M}_q} u^{\deg(V_3)} \overline{\chi_{B^2 C}}(V_3)  \prod_{P\mid (E,V_3)} (1-|P|_q)  =&\prod_{P \nmid EBC} \left ( 1-  \overline{\chi_{B^2 C}}(P) u^{\deg(P)} \right )^{-1} \\
& \times \prod_{\substack{P \mid E \\ P \nmid BC}} \left ( 1+ (1- |P|_q) \frac{ \overline{\chi_{B^2 C}}(P) u^{\deg(P)}}{1- \overline{\chi_{B^2 C}}(P) u^{\deg(P)}} \right ) \\
=&\mathcal{L}_q(u, \overline{\chi_{B^2 C}}) \prod_{\substack{P \mid E \\ P \nmid BC}} (1- |P|_q \overline{\chi_{B^2 C}}(P) u^{\deg(P)}).
\end{align*}

Using Perron's formula  (Lemma \ref{perron}) for the sums over $V_3$ and $V_4$, we get that
\begin{align*}
 \sum_{V_3 \in \mathcal{M}_{q,e_1}} &  \overline{\chi_{B^2 C}} ( V_3) \prod_{P\mid (E,V_3)} (1-|P|_q)\\
 & = \frac{1}{2 \pi i} \oint_{|u|=q^{-1/2}} \frac{\mathcal{L}_q(u,\overline{\chi_{B^2 C}}) \prod_{\substack{P \mid E \\ P \nmid BC}} (1- |P|_q\overline{\chi_{B^2 C}}(P) u^{\deg(P)}) }{u^{e_1}} \, \frac{du}{u},
 \end{align*} and 
\begin{align*}
 \sum_{V_4 \in \mathcal{M}_{q,e_2}}  & \chi_{B^2 C} ( V_4)  \prod_{P\mid (E,V_4)} (1-|P|_q) \\
 =&\frac{1}{2 \pi i}  \oint_{|u|=q^{-1/2}} \frac{\mathcal{L}_q(u,\chi_{B^2C}) \prod_{\substack{P \mid E \\ P \nmid BC}} (1- |P|_q \chi_{B^2C}(P) u^{\deg(P)}) }{u^{e_2}} \, \frac{du}{u}.
 \end{align*}

Since $B,C$ are square-free and coprime, and $BC \neq 1$ (because $f$ was not a cube), the $L$--functions in the expressions above are primitive of modulus $BC$, and we can use the Lindel\"{o}f bound (Lemma \ref{lindelof}) 
for each of them. We have
$$ | \mathcal{L}_q(u,\overline{\chi_{B^2C}}) | \ll |BC|_q^{\varepsilon}, | \mathcal{L}_q(u,\chi_{B^2C}) | \ll |BC|_q^{\varepsilon},$$
for $|u|=q^{-1/2}$. 
Then,  the double sum over $V_3$ and $V_4$ in \eqref{v3v4} is 
$$  \ll |BC|_q^{1+ \varepsilon} q^{-g/2} |D_1 D_2 H  |_q |R_1 R_2|_q^{\frac{1}{2}} \prod_{\substack{P \mid E \\ P \nmid BC}} |P|_q^2 .$$
Now we use the fact that
$$ \prod_{P\mid E} |P|_q^{-2} \prod_{P\mid  (E,C)} |P|_q^2 \prod_{P\mid (E,B)} |P|_q^2 \prod_{\substack{P \mid E \\ P \nmid BC}} |P|_q^2 = 1,$$ 
and trivially bound the sums over $C, B, E$ to
get that the entire expression in \eqref{v3v4} is bounded by
$$ q^{n(2+\varepsilon) - g/2} |D_1 D_2 H |_q |R_1 R_2|_q^{\frac{1}{2}}.$$
Finally,  trivially bounding the sums over $n, D_{1}, D_2, R_1, R_2, H$ in equation \eqref{s1-non-cube} it follows that 
$$S_{2, \neq \cube}^{(1)} \ll q^{\frac{A+g}{2}+ g \varepsilon}.$$



We remark that bounding $S_{2, \neq \cube}^{(2)}$ is identical to bounding $S_{ \neq \cube}^{(1)}$. When bounding $S_{2, \neq \cube}^{(0)}$, we apply the Poisson summation formula for the sums over $L_1$ and $L_2$ as before, and note that $G_q(0,f)=0$ since $f$ is not a cube. The Poisson summation formula applied to each of the sums over $L_1$ and $L_2$ gives 2 terms in that case, and multiplying through, we will obtain four terms, each of which can be bounded using the same method as before.
In conclusion, we get
$$S_{2, \neq \cube} \ll q^{\frac{A+g}{2}+\varepsilon g}.$$

\subsection{The dual term}

 \label{sec-dual-term}
We now treat the dual term by proving Lemma \ref{lemma-Kummer3}. Recall from equation \eqref{dual-term} that
$$S_{2,\mathrm{dual}} = \sum_{\substack{d_1+d_2 = g+1\\ d_1 + 2 d_2 \equiv 1 \pmod 3}} \sum_{\substack{F_1 \in \mathcal{H}_{q,d_1} \\ F_2 \in \mathcal{H}_{q,d_2} \\ (F_1, F_2) = 1}} \omega(\chi_{F_1}\overline{\chi_{F_2}}) \sum_{f \in \mathcal{M}_{q,\leq g-A-1}} \frac{\overline{\chi_{f}}(F_1) {\chi_{f}}(F_2)}{|f|_q^{1/2}}.$$
Since $d_1 + 2 d_2 \equiv 1 \pmod 3$, by Corollary \ref{sign-GS} and formula \eqref{sign-restriction}, the sign of the functional equation is 
\begin{align*}
\omega(\chi_{F_1}\overline{\chi_{F_2}})
=& \overline{\epsilon(\chi_3)} q^{-(d_1+d_2)/2}G_q(\chi_{F_1}\overline{\chi_{F_2}})\\
=& \overline{\epsilon(\chi_3)} q^{-(d_1+d_2)/2}G_q(1,F_1)\overline{G_q(1,F_2)},
\end{align*}
where $\chi_3$ is defined by \eqref{def-chi3}.
We rewrite the dual sum as
\begin{align}
S_{2,\rm dual} \label{whatever} 
=&  \overline{\epsilon(\chi_3)} q^{-(g+1)/2} \sum_{\substack{d_1+d_2 = g+1\\ d_1 + 2 d_2 \equiv 1 \pmod 3}} \sum_{f \in \mathcal{M}_{q,\leq g-A-1}} \frac{1}{|f|_q^{1/2}}\sum_{\substack{F_1 \in \mathcal{H}_{q,d_1} \\ F_2 \in \mathcal{H}_{q,d_2} \\ (F_1, F_2) = (F_1 F_2, f)=1}} G_q(f,F_1)\overline{G_q(f,F_2)},
\end{align}
where we have used the fact that 
\[\overline{\chi_f}(F_1)G_q(1,F_1)= \begin{cases} \overline{\chi_{F_1}}(f)G_q(1,F_1)=G_q(f,F_1) & (f, F_1)=1, \\
0 & \mbox{otherwise,}
\end{cases} \]
and similarly for $F_2$. 
We first notice that if $F_1$ or $F_2$ are not square-free, then since $(F_1 F_2, f)=1$,  we have by Lemma \ref{gauss} that
$G_q(f,F_1)=0$ or ${G_q(f,F_2)} = 0.$ Therefore, we can write
\begin{eqnarray*}
\sum_{\substack{F_1 \in \mathcal{H}_{q,d_1} \\ F_2 \in \mathcal{H}_{q,d_2} \\ (F_1, F_2) = (F_1 F_2, f)=1}} G_q(f,F_1)\overline{G_q(f,F_2)} &=&
\sum_{\substack{F_1 \in \mathcal{M}_{q,d_1}\\ F_2\in \mathcal{M}_{q,d_2} \\ (F_1, F_2) = (F_1 F_2, f)=1}} G_q(f,F_1)\overline{G_q(f,F_2)} \\
&=&  \sum_{\substack{F_1\in \mathcal{M}_{q,d_1}\\ F_2\in \mathcal{M}_{q,d_2}\\  (F_1 F_2, f)=1}} \sum_{H \mid (F_1, F_2)} \mu(H) G_q(f,F_1)\overline{G_q(f,F_2)} \\
&=&  \sum_{\deg(H) \leq \min{(d_1, d_2)}} \mu(H) \sum_{\substack{F_1\in \mathcal{M}_{q,{d_1 - \deg(H)}} \\ F_2\in \mathcal{M}_{q,{d_2 - \deg(H)}}  \\  (H F_1 F_2, f)=1}} G_q(f,H F_1)\overline{G_q(f,H F_2)}.\\
\end{eqnarray*}
Again, if $(H, F_1) \neq 1$ or $(H, F_2) \neq 1$, then $G_q(f, HF_1)=0$ or $G_q(f, HF_2)=0$. If $(H, F_1 F_2)=1$, we can apply Lemma \ref{gauss} and write 
\begin{eqnarray} \label{sum-over-H}
&& \sum_{\substack{F_1 \in \mathcal{H}_{q,d_1} \\ F_2 \in \mathcal{H}_{q,d_2} \\ (F_1, F_2) = (F_1 F_2, f)=1}} G_q(f,F_1)\overline{G_q(f,F_2)} 
=  \sum_{\substack{\deg(H) \leq \min{(d_1, d_2)}\\ (H,f)=1}}   \mu(H) |H|_q \\
&& \times   \sum_{\substack{F_1\in \mathcal{M}_{q,{d_1 - \deg(H)}}  \\  (F_1, f)=1 \\ (F_1, H)=1}}
G_q(f H, F_1) 
  \sum_{\substack{F_2\in \mathcal{M}_{q,{d_2 - \deg(H)}}  \\  (F_2, f)=1 \\ (F_2, H)=1}}
\overline{G_q(f H, F_2)},\nonumber
\end{eqnarray}
where we have used the fact that $G_q(f,H) \overline{G_q(f,H)} = |H|_q.$
Using equation \eqref{whatever} it follows that
\begin{align}
S_{2,\mathrm{dual}} =&  \overline{\epsilon(\chi_3)} q^{-(g+1)/2} \sum_{\substack{d_1+d_2 = g+1\\ d_1 + 2 d_2 \equiv 1 \pmod 3}} \sum_{f \in \mathcal{M}_{q,\leq g-A-1}} \frac{1}{|f|_q^{1/2}} \sum_{\substack{\deg(H) \leq \min{(d_1, d_2)}\\ (H,f)=1}}   \mu(H) |H|_q \nonumber \\
& \times  \sum_{\substack{F_1\in \mathcal{M}_{q,{d_1 - \deg(H)}}  \\  (F_1, fH)=1 }}
G_q(f H, F_1) 
  \sum_{\substack{F_2\in \mathcal{M}_{q,{d_2 - \deg(H)}}  \\  (F_2, fH)=1 }}
\overline{G_q(f H, F_2)}. \label{dual_bound}
\end{align}
Using Proposition \ref{big-F-tilde-corrected} we have that
\begin{align*}
& \sum_{\substack{F_1\in \mathcal{M}_{q,{d_1 - \deg(H)}}  \\ (F_1,fH)=1}} G_q(fH,F_1) = \delta_{f_2=1} \frac{q^{ \frac{4}{3} (d_1 - \deg(H))-\frac{4}{3} [d_1+\deg(f_1)]_3}}{\zeta_q(2)|f_1H|_q^{\frac{2}{3}}} \overline{G_q(1,f_1H)} \rho(1, [d_1+\deg(f_1)]_3) \\
& \times \prod_{P|fH} \left( 1+ \frac{1}{|P|_q} \right)^{-1} + O \left(  \delta_{f_2=1} \frac{q^{\frac{d_1}{3}-\frac{\deg(H)}{2}+\varepsilon (d_1-\deg(H))}} {|f_1|_q^{\frac{1}{6}}} + q^{\sigma d_1+\left(\frac{3}{4}-\frac{3}{2}\sigma\right) \deg(H)} |f|_q^{\frac{1}{2} ( \frac{3}{2} - \sigma)} \right),
\end{align*}
and a similar formula holds for the sum over $F_2$. Note that the second error term dominates the first error term. Then we have
\begin{align}
& \sum_{\substack{F_1\in \mathcal{M}_{q,{d_1 - \deg(H)}} \\  (F_1, fH)=1 }}
G_q(f H, F_1)   \sum_{\substack{F_2\in \mathcal{M}_{q,{d_2 - \deg(H)}}  \\  (F_2, fH)=1 }}\overline{G_q(f H, F_2)} \nonumber \\
= & \delta_{f_2=1} \frac{ q^{\frac{4(g+1)}{3} -3\deg(H)- \frac{4}{3} \left([d_1+ \deg(f_1)]_3+[d_2+ \deg(f_1)]_3\right)} }{ \zeta_q(2)^2|f_1|_q^{\frac{1}{3}}}  \rho(1, [d_1+ \deg(f_1)]_3)
\overline{\rho(1, [d_2+ \deg(f_1)]_3)} \nonumber
\\
&\times \prod_{P\mid fH} \left ( 1+\frac{1}{|P|_q}\right )^{-2} \nonumber \\ 
& + O\left (\frac{q^{\frac{4d_1}{3}-\frac{3}{2}\deg(H)}}{|f_1|_q^{\frac{1}{6}}} q^{\sigma d_2+\left(\frac{3}{4}-\frac{3}{2}\sigma\right) \deg(H)} |f|_q^{\frac{1}{2} ( \frac{3}{2} - \sigma)} \right)  + O\left (\frac{q^{\frac{4d_2}{3}-\frac{3}{2}\deg(H)}}{|f_1|_q^{\frac{1}{6}}} q^{\sigma d_1+\left(\frac{3}{4}-\frac{3}{2}\sigma\right) \deg(H)} |f|_q^{\frac{1}{2} ( \frac{3}{2} - \sigma)}\right) \label{e21} \\
&+ O \left(    q^{\sigma d_2+\left(\frac{3}{4}-\frac{3}{2}\sigma\right) \deg(H)} |f|_q^{\frac{1}{2} ( \frac{3}{2} - \sigma)}
q^{\sigma d_1+\left(\frac{3}{4}-\frac{3}{2}\sigma\right) \deg(H)} |f|_q^{\frac{1}{2} ( \frac{3}{2} - \sigma)} \right).\label{e31}
\end{align}
Then the main term of $S_{2,\mathrm{dual}}$ is equal to
\begin{align*}
M_\mathrm{dual}=&  \overline{\epsilon(\chi_3)} \frac{q^{\frac{5}{6}(g+1)}}{\zeta_q(2)^2} \sum_{\substack{d_1+d_2 = g+1\\ d_1 + 2 d_2 \equiv 1 \pmod 3}} \sum_{f \in \mathcal{M}_{q,\leq g-A-1}}  \delta_{f_2=1} \frac{ q^{ - \frac{4}{3} \left([d_1+ \deg(f_1)]_3+[d_2+ \deg(f_1)]_3\right)} }{ |f|_q^{1/2}|f_1|_q^{1/3}} \\
  &\times  \rho(1, [d_1+ \deg(f_1)]_3)
\overline{\rho(1, [d_2+ \deg(f_1)]_3)}
\\
& \times \sum_{\substack{\deg(H) \leq \min{(d_1, d_2)}\\ (H,f)=1}}   \frac{\mu(H)}{ |H|_q^2} \nonumber  \prod_{P\mid fH} \left ( 1+\frac{1}{|P|_q}\right )^{-2}.\\ 
 \end{align*}
Notice that the product of the terms involving $\rho$ is nonzero only when $d_1+\deg(f_1)\equiv 1 \pmod{3}$ (and therefore $d_2+\deg(f_1)\equiv 0 \pmod{3}$). By Lemma \ref{lemma-residue},
\begin{align*}
M_\mathrm{dual}=&  \frac{q^{\frac{5}{6}g+1}}{\zeta_q(2)^2} \sum_{\substack{d_1+d_2 = g+1\\ d_1 + 2 d_2 \equiv 1 \pmod 3}} \sum_{f \in \mathcal{M}_{q,\leq g-A-1}}  \delta_{f_2=1} \frac{1}{ |f|_q^{1/2}|f_1|_q^{1/3}} \\
& \times \sum_{\substack{\deg(H) \leq \min{(d_1, d_2)}\\ (H,f)=1}}   \frac{\mu(H)}{ |H|_q^2} \nonumber  \prod_{P\mid fH} \left ( 1+\frac{1}{|P|_q}\right )^{-2}, 
 \end{align*}
where we have also used that $\tau(\chi_3)=\epsilon(\chi_3)\sqrt{q}$.

We look at the generating series of the sum over $H$. We have
\begin{align*}
 \sum_{(H,f)=1} \frac{\mu(H)}{|H|_q^2} \prod_{P\mid H} \left ( 1+\frac{1}{|P|_q}\right )^{-2} w^{\deg(H)} = \prod_{P \nmid f} \left (1- \frac{w^{\deg(P)}}{(|P|_q+1)^2} \right ).
\end{align*}

Let $R_P(w)$ denote the $P$--factor above and let $\mathcal{R}_{\mathrm{K}}(w) = \prod_P R_P(w)$. By Perron's formula, we get that
\begin{align*}
\sum_{\substack{\deg(H) \leq \min\{d_1,d_2\} \\ (H,f)=1}} \frac{\mu(H)}{|H|_q^2} \prod_{P\mid H} \left ( 1+\frac{1}{|P|_q}\right )^{-2} = \frac{1}{2 \pi i} \oint \frac{ \mathcal{R}_{\mathrm{K}}(w) \prod_{P\mid f} R_P(w)^{-1}}{(1-w) w^{ \min\{d_1,d_2\}}} \, \frac{dw}{w}.
\end{align*}

Recall from Section \ref{sec-main} that  $d_1 \equiv a \pmod 3, 2g+1 \equiv a \pmod 3, g \equiv b \pmod 3$ and $A \equiv 0 \pmod 3$. Then we need $\deg(f) \equiv b \pmod 3$. 
Now we look at the sum over $f$. The generating series is
\begin{align*}
 \sum_{f} \delta_{f_2=1}&  \frac{u^{\deg(f)}}{|f|_q^{1/2} |f_1|_q^{1/3} }  \prod_{P\mid f} \left ( 1+ \frac{1}{|P|_q} \right )^{-2} R_P(w)^{-1} \\
 =&\prod_P \left[ 1+ \frac{1}{R_P(w) (1+\frac{1}{|P|_q})^2} \left ( \frac{1}{|P|_q^{1/3}} \sum_{j=0}^{\infty} \frac{u^{(3j+1) \deg(P)}}{|P|_q^{(3j+1)/2}} +  \sum_{j=1}^{\infty} \frac{u^{3j \deg(P)}}{|P|_q^{3j/2}}  \right )\right] \\
 =&\prod_P \left[1+ \frac{1}{R_P(w)(1+\frac{1}{|P|_q})^2 } \frac{u^{\deg(P)} ( 1+ \frac{u^{2 \deg(P)}}{|P|_q^{2/3}})}{|P|_q^{5/6} ( 1- \frac{ u^{3 \deg(P)}}{|P|_q^{3/2}})} \right].
 \end{align*}
 Let 
 $$\mathcal{E}_{\mathrm{K}}(u,w) = \prod_P R_P(w) \left[1+ \frac{1}{R_P(w) (1+\frac{1}{|P|_q})^2 } \frac{u^{\deg(P)} ( 1+ \frac{u^{2 \deg(P)}}{|P|_q^{2/3}})}{|P|_q^{5/6} ( 1- \frac{ u^{3 \deg(P)}}{|P|_q^{3/2}})} \right] = \mathcal{Z}_q \left (\frac{u}{q^{5/6}} \right ) \mathcal{U}_{\mathrm{K}}(u,w).$$
 Write $\deg(f) = 3k+b$. Since $\deg(f) \leq g-A-1$ and $g-A-1 \equiv b-1 \pmod 3$, we have by Perron's formula
 \begin{align}
&\mathcal{R}_{\mathrm{K}}(w) \sum_{\substack{f \in \mathcal{M}_{q,\leq g-A-1} \\ \deg(f) \equiv b \pmod{3}}} \delta_{f_2=1}  \frac{1}{ |f_1|_q^{1/3} |f|_q^{1/2}}  \prod_{P\mid f} \left ( 1+ \frac{1}{|P|_q} \right )^{-2} R_P(w)^{-1}\nonumber \\
=&\frac{1}{2 \pi i} \oint \sum_{k=0}^{(g-A-3-b)/3} \frac{1}{u^{3k+b}} \mathcal{E}_{\mathrm{K}}(u,w) \, \frac{du}{u} \nonumber \\
 =&\frac{1}{2 \pi i} \oint \frac{\mathcal{E}_{\mathrm{K}}(u,w)}{(1-u^3) u^{g-A-3}} \, \frac{du}{u},\label{sum_f}
 \end{align}
 where we are integrating along a small circle around the origin. 
 
 Introducing the sum over $d_1$, we have
 \begin{align} \label{sum-d}
M_{\mathrm{dual}} = \frac{q^{\frac{5}{6}g+1}}{\zeta_q(2)^2} \frac{1}{(2 \pi i)^2} \oint \oint \frac{\mathcal{U}_{\mathrm{K}}(u,w)}{ (1-uq^{1/6})(1-u^3)  u^{g-A-3} (1-w) } \sum_{\substack{d_1+d_2=g+1 \\ d_1+2d_2\equiv 1 \pmod 3}} w^{- \min \{d_1,d_2 \}} \, \frac{dw}{w} \, \frac{du}{u}.
 \end{align}
Note that since $d_1 \equiv a \pmod 3$, we have that $d_2 \equiv a-1 \pmod 3.$ For simplicity of notation, let $\alpha = [a-1]_3$. We rewrite the sum over $d_1, d_2$ as
 \begin{align*}
 \sum_{\substack{d_1+d_2=g+1 \\ d_1+2d_2\equiv 1 \pmod 3}} w^{- \min \{d_1,d_2 \}} = \sum_{k=0}^{ [ (g+1-2a)/6]} \frac{1}{w^{3k+a}} + \sum_{k=0}^{[(g-1-2\alpha)/6]} \frac{1}{w^{3k+\alpha}}.
 \end{align*}
 Assume that $g$ is odd. 
 We have
 $$ [ (g+1-2a)/6]= \frac{g-1-2a}{6}, \,\, \,  [(g-1-2 \alpha)/6]=\frac{g-3-2 \alpha}{6}.$$
 Then using the above in \eqref{sum-d} we get that
 \begin{align*}
 M_{\mathrm{dual}} =&\frac{q^{\frac{5}{6}g+1}}{\zeta_q(2)^2} \frac{1}{(2 \pi i)^2} \oint \oint \frac{\mathcal{U}_{\mathrm{K}}(u,w)(1+w)}{ (1-uq^{1/6}) (1-u^3)  u^{g-A-3}  (1-w) (1-w^3) w^{\frac{g-1}{2}}} \, \frac{dw}{w} \, \frac{du}{u}.
 \end{align*}
%
 Note that we have a pole at $u=q^{-1/6}$. 
\kommentar{ We compute the residue at $u=q^{-1/6}$ while moving the integral just before the second pole at $u=q^{1/6}$ and obtain
 \begin{align*}
 M_{\mathrm{dual}} =& - q^{g-\frac{A}{6}+1} \frac{\zeta_q(1/2)}{\zeta_q(2)^2} \frac{1}{2\pi i}\oint \frac{\mathcal{U}_{\mathrm{K}}(q^{-1/6},w)(1+w)}{(1-w)   (1-w^3) w^{\frac{g-1}{2}} } \, \frac{dw}{w}\\
 &+\frac{q^{\frac{5}{6}g+1}}{\zeta_q(2)^2} \frac{1}{(2 \pi i)^2} \oint_{|u|=q^{\frac{1}{6}-\varepsilon}} \oint_{|w|=q^{-\varepsilon}} \frac{\mathcal{U}_{\mathrm{K}}(u,w)(1+w)}{ (1-uq^{1/6}) (1-u^3)  u^{g-A-3}  (1-w) (1-w^3) w^{\frac{g-1}{2}}} \, \frac{dw}{w} \, \frac{du}{u}\\
=&- q^{g-\frac{A}{6}+1} \frac{\zeta_q(1/2)}{\zeta_q(2)^2} \frac{1}{2\pi i}\oint \frac{\mathcal{U}_{\mathrm{K}}(q^{-1/6},w)(1+w)}{(1-w)   (1-w^3) w^{\frac{g-1}{2}} } \, \frac{dw}{w}+O\left(q^{\left(\frac{2}{3}+\varepsilon\right)g+\frac{A}{6}} \right).
 \end{align*}
 \acom{We also encounter the pole at $u^3=1$ which gives an explicit term of size $q^{5g/6}$. So we can either compute it explicitly (though we won't be able to detect it in the asymptotic formula because of the size of error terms), or write the error term as $O(q^{5g/6+\varepsilon g})$ instead of $q^{2g/3+A/6+\varepsilon g}$.} }

We compute the residue at $u=q^{-1/6}$ while moving the integral just before the poles at $u^3=1$ and obtain
 \begin{align}
 M_{\mathrm{dual}} &= - q^{g-\frac{A}{6}+1} \frac{\zeta_q(1/2)}{\zeta_q(2)^2} \frac{1}{2\pi i}\oint \frac{\mathcal{U}_{\mathrm{K}}(q^{-1/6},w)(1+w)}{(1-w)   (1-w^3) w^{\frac{g-1}{2}} } \, \frac{dw}{w} \nonumber \\
 &+\frac{q^{\frac{5}{6}g+1}}{\zeta_q(2)^2} \frac{1}{(2 \pi i)^2} \oint_{|u|=q^{-\varepsilon}} \oint_{|w|=q^{-\varepsilon}} \frac{\mathcal{U}_{\mathrm{K}}(u,w)(1+w)}{ (1-uq^{1/6}) (1-u^3)  u^{g-A-3}  (1-w) (1-w^3) w^{\frac{g-1}{2}}} \, \frac{dw}{w} \, \frac{du}{u} \nonumber \\
=&- q^{g-\frac{A}{6}+1} \frac{\zeta_q(1/2)}{\zeta_q(2)^2} \frac{1}{2\pi i}\oint \frac{\mathcal{U}_{\mathrm{K}}(q^{-1/6},w)(1+w)}{(1-w)   (1-w^3) w^{\frac{g-1}{2}} } \, \frac{dw}{w}+O\left(q^{\frac{5g}{6}+\varepsilon g} \right). \label{pole_1}
 \end{align}

  In the integral above we have a double pole at $w=1$ and simple poles at $w = \xi_3, w=\xi_3^2$. We have
 \begin{equation*}
\mathcal{U}_{\mathrm{K}} (q^{-1/6},w) = \prod_P \left (1- \frac{1}{|P|_q} \right ) \left ( 1- \frac{w^{\deg(P)}}{(|P|_q+1)^2}+ \frac{1}{(|P|_q+1) (1-\frac{1}{|P|_q^2})}\right ):= \mathcal{H}_{\mathrm{K}}(w).
 \end{equation*}

We compute the residue of the double pole at $w=1$ and get that it is equal to
 \begin{equation*}
 \label{res-1}
 -\frac{g+2}{3} \mathcal{H}_{\mathrm{K}}(1) + \frac{2 \mathcal{H}_{\mathrm{K}}'(1)}{3}. 
 \end{equation*} 
Note that
$$\frac{\mathcal{H}_{\mathrm{K}}(1)}{\zeta_q(2)^2} = \prod_P \frac{(|P|_q^2+2|P|_q-2)(|P|_q-1)^2}{|P|_q^4} = \mathcal{D}_{\mathrm{K}}(1/q,1/q,\sqrt{q}),$$
where recall that $\mathcal{D}_{\mathrm{K}}(x,y,u)$ is defined by \eqref{euler-d}.

Now we compute the residue of the pole at $w=\xi_3^{-1}$ which is equal to
\begin{equation*} \label{res-2}
\frac{\mathcal{H}_{\mathrm{K}}(\xi_3^2) (1+\xi_3^2)}{(1-\xi_3^2)^2 (1-\xi_3)} \xi_3^{\frac{g-1}{2}}= - \frac{\mathcal{H}_{\mathrm{K}}(\xi_3^2)}{3(1-\xi_3^2)} \xi_3^{2g+2}.
\end{equation*}

The residue at $w=\xi_3$ is equal to 
\begin{equation*} \label{res-3}
\frac{\mathcal{H}_{\mathrm{K}}(\xi_3) (1+\xi_3)}{(1-\xi_3)^2(1-\xi_3^2)} \xi_3^{g-1}= - \frac{\mathcal{H}_{\mathrm{K}}(\xi_3)}{3(1-\xi_3)} \xi_3^{g+1}.
\end{equation*} 


Putting everything together, we have
\begin{align*}
 M_{\mathrm{dual}} =& q^{g-\frac{A}{6}+1} \frac{\zeta_q(1/2)}{\zeta_q(2)^2} \left(-\frac{g+2}{3} \mathcal{H}_{\mathrm{K}}(1) + \frac{2 \mathcal{H}_{\mathrm{K}}'(1)}{3}- \frac{\mathcal{H}_{\mathrm{K}}(\xi_3^2)}{3(1-\xi_3^2)} \xi_3^{2g+2}- \frac{\mathcal{H}_{\mathrm{K}}(\xi_3)}{3(1-\xi_3)} \xi_3^{g+1}\right)\\
 &+ q^{g-\frac{A}{6}+1} \frac{\zeta_q(1/2)}{\zeta_q(2)^2} \frac{1}{2\pi i}\oint_{|w|=q^{1-\varepsilon}} \frac{\mathcal{H}_{\mathrm{K}}(w)(1+w)}{(1-w)   (1-w^3) w^{\frac{g-1}{2}} } \, \frac{dw}{w}+O\left(q^{\left(\frac{5}{6}+\varepsilon\right)g} \right)\\
=&q^{g-\frac{A}{6}+1} \frac{\zeta_q(1/2)}{\zeta_q(2)^2} \left(-\frac{g+2}{3} \mathcal{H}_{\mathrm{K}}(1) + \frac{2 \mathcal{H}_{\mathrm{K}}'(1)}{3}- \frac{\mathcal{H}_{\mathrm{K}}(\xi_3^2)}{3(1-\xi_3^2)} \xi_3^{2g+2}- \frac{\mathcal{H}_{\mathrm{K}}(\xi_3)}{3(1-\xi_3)} \xi_3^{g+1}\right)\\
&+O\left(q^{\frac{5g}{6}+\varepsilon g} \right).
\end{align*}

 \begin{rem}\label{remark56b} As in Remark \ref{remark56}, the  error term of size 
$q^{\frac{5g}{6}}$ can be computed explicitly
by evaluating the residue when $u^3=1$ in \eqref{pole_1}. The other error terms will eventually dominate the term of size $q^{\frac{5g}{6}}$, so we do not carry out the computation. However, we believe this term will persist in the asymptotic formula.
 \end{rem} 
 
 
Now assume that $g$ is even. Then 
\[ [ (g+1-2a)/6]= \frac{g-4-2a}{6}, \,\, \,  [(g+1-2 \alpha)/6]=\frac{g-2 \alpha}{6}.\]
Similarly as before, we get that
\begin{align*}
 M_{\mathrm{dual}} =&\frac{q^{\frac{5}{6}g+1}}{\zeta_q(2)^2} \frac{1}{(2 \pi i)^2} \oint \oint \frac{\mathcal{U}_{\mathrm{K}}(u,w)(1+w^2)}{ (1-uq^{1/6}) (1-u^3)  u^{g-A-3}  (1-w) (1-w^3) w^{\frac{g}{2}}} \, \frac{dw}{w} \, \frac{du}{u}\\
=&- q^{g-\frac{A}{6}+1} \frac{\zeta_q(1/2)}{\zeta_q(2)^2} \frac{1}{2\pi i}\oint \frac{\mathcal{U}_{\mathrm{K}}(q^{-1/6},w)(1+w^2)}{(1-w)   (1-w^3) w^{\frac{g}{2}} } \, \frac{dw}{w}+O\left(q^{\frac{5g}{6}+\varepsilon g} \right).
 \end{align*}
Then the residues give
\[\frac{\mathcal{H}_{\mathrm{K}}(\xi_3^2) (1+\xi_3)}{(1-\xi_3^2)^2 (1-\xi_3)} \xi_3^{\frac{g}{2}}= - \frac{\mathcal{H}_{\mathrm{K}}(\xi_3^2)}{3(1-\xi_3^2)} \xi_3^{2g+2},\] and
\[\frac{\mathcal{H}_{\mathrm{K}}(\xi_3) (1+\xi_3^2)}{(1-\xi_3)^2 (1-\xi_3^2)} \xi_3^{g}= - \frac{\mathcal{H}_{\mathrm{K}}(\xi_3)}{3(1-\xi_3)} \xi_3^{g+1},\]
so 
\begin{align*}
 M_{\mathrm{dual}} 
=&q^{g-\frac{A}{6}+1} \frac{\zeta_q(1/2)}{\zeta_q(2)^2} \left(-\frac{g+2}{3} \mathcal{H}_{\mathrm{K}}(1) + \frac{2 \mathcal{H}_{\mathrm{K}}'(1)}{3}- \frac{\mathcal{H}_{\mathrm{K}}(\xi_3^2)}{3(1-\xi_3^2)} \xi_3^{2g+2}- \frac{\mathcal{H}_{\mathrm{K}}(\xi_3)}{3(1-\xi_3)} \xi_3^{g+1}\right)\\
&+O\left(q^{\frac{5g}{6}+\varepsilon g} \right).
\end{align*}

 We remark that assuming $g$ even leads to the same asymptotic formula as before. 

We now bound the mixed terms \eqref{e21} and \eqref{e31} in $S_{2,\mathrm{dual}}$. 
For the terms of the type \eqref{e21} we have
\begin{align*}
\ll & q^{-\frac{g}{2}} \sum_{d_1+d_2 = g+1} q^{\frac{4d_1}{3}+\sigma d_2} \sum_{f \in \mathcal{M}_{q,\leq g-A-1}} \frac{1}{|f|_q^{\frac{\sigma}{2}-\frac{1}{4}}|f_1|_q^{\frac{1}{6}}} \sum_{\substack{\deg(H) \leq \min{(d_1, d_2)}\\ (H,f)=1}}    q^{\left(\frac{1}{4}-\frac{3}{2}\sigma\right) \deg(H)}.
\end{align*}
\kommentar{\acom{If we use the observation on page $55$ then we only need to bound the second term above.}
\mcom{The first line seems to come from some older term that we're ignoring already. 
Don't we need to account for (80)?
\[\ll q^{(\sigma-\frac{1}{2})g} \sum_{d_1+d_2 = g+1} \sum_{f \in \mathcal{M}_{q,\leq g-A-1}} |f|_q^{1-\sigma}\sum_{\substack{\deg(H) \leq \min{(d_1, d_2)}\\ (H,f)=1}}    q^{\left(\frac{5}{2}-3\sigma\right) \deg(H)}
\]
}}
\kommentar{Bounding trivially the sum over $H$, the first term is 
\[\ll  g q^{-\frac{A}{6}}\sum_{d_1+d_2 = g+1}q^{d_1+\varepsilon d_2}\ll q^{(1+\varepsilon)g-\frac{A}{6}}\]}
Setting $\sigma \geq 5/6$, and bounding trivially the sum over $H$, it follows that these terms are bounded by
\begin{equation*}
\ll g q^{\frac{5}{6}g+\left(\frac{13}{12}-\frac{\sigma}{2}\right)(g-A)+\varepsilon g}\ll q^{\left(\frac{23}{12}-\frac{\sigma}{2}+\varepsilon\right)g-\left(\frac{13}{12}-\frac{\sigma}{2} \right)A}. \label{bound2} 
\end{equation*}
\kommentar{\acom{We can choose $\sigma$ more optimally to give an error term of size $q^{10g/11+\varepsilon g}$ overall. If we set 
$$ \frac{A+g}{2} =\Big( \frac{23}{12}-\frac{\sigma}{2}+\varepsilon \Big)g-\left(\frac{13}{12} -\frac{\sigma}{2} \right)A,$$ we get that
$$A= g \Big(1-\frac{2}{19-6\sigma} \Big).$$
We want $$\frac{1}{2} (A+g) = g \Big(1- \frac{1}{19-6\sigma} \Big)$$ to be minimal, which happens when choosing $\sigma= \frac{4}{3}-\epsilon.$ With this choice of $\sigma$ we have $A= g \frac{9+6 \epsilon}{11+6\epsilon}$ and then we get a total error term of size $q^{10g/11+\varepsilon g}$ (after relabeling the $\epsilon$). Let me know if you agree.
}}
We now bound the error term coming from \eqref{e31}. This term will be bounded by 
\begin{align*}
& \ll q^{\sigma g - \frac{g}{2} } \sum_{d_1+d_2=g+1} \sum_{f \in \mathcal{M}_{q,\leq g-A-1}} |f|_q^{1-\sigma} \sum_{\substack{\deg(H) \leq \min{(d_1, d_2)}\\ (H,f)=1}} |H|_q^{1+\frac{3}{2}-3\sigma} \\
& \ll q^{\sigma g - \frac{g}{2} + \varepsilon g + (g-A)(2-\sigma)} \ll q^{\frac{3g}{2}-A(2-\sigma)+\varepsilon g}
\end{align*}
as long as $\sigma \geq 7/6$.

 Then the error from $S_{2,\mathrm{dual}}$ will be bounded by
$$ E_{\mathrm{dual}} \ll q^{\left(\frac{23}{12}-\frac{\sigma}{2}+\varepsilon\right)g-\left(\frac{13}{12}-\frac{\sigma}{2} \right)A} +  q^{\frac{3g}{2}-A(2-\sigma)+\varepsilon g}.$$
This finishes the proof of Lemma \ref{lemma-Kummer3}.

\subsection{The proof of Theorem \ref{first-moment-Kummer}}

Combining Lemmas \ref{lemma-cubes}, \ref{lemma-non-cubes} and \ref{lemma-Kummer3}, it follows that
\begin{align*}
&\sum_{\substack{d_1+d_2 = g+1\\ d_1 + 2 d_2 \equiv 1 \pmod 3}} \sum_{\substack{F_1 \in \mathcal{H}_{q,d_1} \\ F_2 \in \mathcal{H}_{q,d_2} \\ (F_1, F_2) = 1}} L_q \left( \frac{1}{2}, \chi_{F_1 F_2^2} \right)  =C_{\mathrm{K},1} g q^{g+1} + C_{\mathrm{K},2} q^{g+1} + D_{\mathrm{K},1} g q^{g+1-\frac{A}{6}}+D_{\mathrm{K},2} q^{g+1- \frac{A}{6}}\\
&+ O \left (q^{\frac{g}{3}+\varepsilon g}+q^{g- \frac{5A}{6}+\varepsilon g} +q^{ \frac{A+g}{2}+ \varepsilon g}+q^{(1+\varepsilon)g -\frac{A}{6}} + q^{ \left(\frac{23}{12}-\frac{\sigma}{2}+\varepsilon\right)g- \left( \frac{13}{12}-\frac{\sigma}{2}\right)A} + q^{\frac{3g}{2}-A(2-\sigma)+\varepsilon g}\right ),\\ 
\end{align*}
where $7/6 \leq \sigma<4/3$. Picking  $\sigma = \frac{13-2 \sqrt{7}}{6}$ and  
$A = 3\left[\frac{ g( \sqrt{7}-1)}{6}\right]$ (so that $A \equiv 0 \pmod 3$) gives a total upper bound of 
size $q^{g \frac{1+\sqrt{7}}{4}+\varepsilon g}$ and finishes the proof of Theorem \ref{first-moment-Kummer}.

\kommentar{\acom{

If $\sigma=5/4$ and $A=19g/23$ then the total upper bound will be of size $q^{21g/23+\varepsilon g}$ which is better than $q^{11g/12+\varepsilon g}$. 

I think the optimal choice for $\sigma$ is $\sigma = \frac{13-2 \sqrt{7}}{6}g$ and we pick $A = g \frac{2+\sqrt{7}}{3+\sqrt{7}}$. This will give a total upper bound of size $q^{ \frac{ 5+ 2 \sqrt{7}}{6+2 \sqrt{7}} g+\varepsilon g }= q^{g \frac{1+\sqrt{7}}{4}+\varepsilon g}$.
}
\mcom{I also get that the optimal choice is $\sigma=\frac{13-2 \sqrt{7}}{6}$. I get $A=\frac{(\sqrt{7}-1)g}{2}$
and the error term as $q^{\frac{(1+\sqrt{7})g}{4}+\varepsilon g}$. Notice that 
\[\frac{1+\sqrt{7}}{4}=0.91143782776614764762540393840981510643\]
while
\[11/12=0.9166666... \]
}}

 \bibliographystyle{amsalpha}

\bibliography{Bibliography}

\end{document}